\documentclass[11pt]{article}

\usepackage{latexsym,amsmath,amssymb,amstext,amsthm,euscript,array,enumerate,mathrsfs}
\usepackage{color}
\usepackage{comment}
\usepackage[utf8]{inputenc}   
\usepackage[english]{babel} 
\usepackage{hyperref} 
\hypersetup{pdfstartview=, hidelinks} 
\usepackage{bbm} 
\usepackage{microtype} 
\usepackage{url} 
\allowdisplaybreaks 

\theoremstyle{plain}
\newtheorem{Theorem}{Theorem}[section]
\newtheorem{Lemma}[Theorem]{Lemma}
\newtheorem{Proposition}[Theorem]{Proposition}
\newtheorem{Corollary}{Corollary}[Theorem]
\newtheorem{Definition}[Theorem]{Definition}
\newtheorem*{Hp}{Assumptions}
\newtheorem*{HpSingle}{Assumption}

\theoremstyle{definition}
\newtheorem{Remark}[Theorem]{Remark}
\newtheorem{Example}[Theorem]{Example}

\numberwithin{equation}{section}

\def\esssup_#1{\underset{#1}{\mathrm{ess\,sup\, }}} 
\def\essinf_#1{\underset{#1}{\mathrm{ess\,inf\, }}} 

\usepackage{mathtools}
\mathtoolsset{showonlyrefs}

\def\qed{{\hfill\hbox{\enspace${ \square}$}} \smallskip}
\def\sqr#1#2{{\vcenter{\vbox{\hrule height .#2pt \hbox{\vrule
 width .#2pt height#1pt \kern#1pt \vrule
width .#2pt} \hrule height .#2pt}}}}
\def\square{\mathchoice\sqr54\sqr54\sqr{4.1}3\sqr{3.5}3}

\def\ds{\begin{displaystyle}}
\def\eds{\end{displaystyle}}
\def\dis{\displaystyle }
\def\<{\langle }
\def\>{\rangle }

\def \I{\mathbb{I}}
 
\def \R{\mathbb{R}}

\def \E{\mathbb{E}} 
\def \F{\mathbb{F}} 
\def \G{\mathbb{G}}
\def \P{\mathbb{P}} 
\def \Q{\mathbb{Q}}

\def \H{\mathbb{H}}

\def \Fc{{\cal F}}

\def\cala{{\cal A}} 
\def\calb{{\cal B}} 
\def\calc{{\cal C}} 
\def\cald{{\cal D}}

\def\calf{{\cal F}}
\def\calg{{\cal G}}
\def\calh{{\cal H}}

\def\calp{{\cal P}}

\def\calx{{\cal X}}
\def\caly{{\cal Y}}

\def\call{{\cal L}} 
\def\cals{{\cal S}}
\def\calw{{\cal W}}

\def \eps{\varepsilon}

\definecolor{SilviaColor}{RGB}{255,0,255}
\definecolor{MarcoColor}{RGB}{100,100,255}
\definecolor{HuyenColor}{RGB}{255,140,0}

\def\beqs{\begin{eqnarray*}}
\def\enqs{\end{eqnarray*}}
\def\beq{\begin{eqnarray}}
\def\enq{\end{eqnarray}}

\allowdisplaybreaks

\title{
Optimal control of  McKean-Vlasov systems under partial observation and 
hidden Markov switching}

\author{Marco Fuhrman \footnote{Dipartimento di Matematica, Università degli Studi di Milano, {\sf marco.fuhrman at unimi.it} This author is a member of  INDAM-GNAMPA.}   
\and 
Huy\^en Pham\footnote{Ecole Polytechnique, CMAP \sf huyen.pham at polytechnique.edu  
The work of this author is  partially supported by the Chair ``Risques Financiers", and the 
Chair Finance \& Sustainable Development / the FiME Lab (Institut Europlace de Finance)}
\and 
Silvia Rudà \footnote{Dipartimento di Matematica, Università degli Studi di Milano, {\sf silvia.ruda at unimi.it}. This author is a member of INDAM-GNAMPA}
}

\date{}

\begin{document}

\maketitle

\begin{abstract}
We study a class of mean-field control problems under partial observation. The controlled dynamics are of McKean–Vlasov type and are subject to regime switching driven by a hidden Markov chain. The observation process depends on the control and on the joint distribution of the state and control, which prevents the direct application of standard filtering techniques.

The main contribution of this paper is to show how this distribution dependence can be handled within a change-of-probability framework, leading to a well-posed separated control problem. We derive a Zakai equation with a specific structure for the unnormalized filter, and show that the corresponding value function satisfies a dynamic programming principle. This yields a Bellman equation posed on a convex subset of a Wasserstein space, characterizing the optimal control problem under partial observation. 
\end{abstract}


\vspace{5mm}

\noindent {\bf Keywords:} Mean-field control; McKean-Vlasov dynamics; partial observation; hidden Markov switching; Zakai equation; dynamic programming equation in  Wasserstein space.

\vspace{5mm}

\noindent {\bf MSC Classification:} 93E20; 60G35; 60J27; 49L20.

\section{Introduction
}

We consider an optimal control problem for a large population system evolving under partial observation. The dynamics of the controlled state is of McKean–Vlasov type, reflecting mean-field interactions, and is subject to regime switching driven by a finite-state Markov chain.  The regime process models latent environmental factors and is not directly observed by the controller. 

The problem is formally described as follows. Let $(\Omega,\calf,\Q)$ be a probability space 
supporting a $d$-dimensional Brownian motion $B$, and a continuous-time Markov chain $M$ with finite state space $S$, independent of $B$.   
The controller observes a process $Y$, whose dynamics is given by 
\begin{align}\label{stateeq1}
    \left\{
\begin{array}{rcl}
dY_s&=& \sigma(Y_s)\,h(Y_s,\Q_{Y_s,\alpha_s},\alpha_s,M_s)\,ds + \sigma(Y_s)\,dB_s,\qquad s\in   [0,T],
\\
Y_0&=&y_0 \in \R^d, 
    \end{array}
    \right.
\end{align}
where $\alpha$ is an admissible control process taking values in  a control set $A$, and $\Q_{Y_s,\alpha_s}$ denotes the law of the pair $(Y_s,\alpha_s)$ under $\Q$. The coefficients depend on the unobserved regime process $M$, 
which thus plays the role of a hidden signal in the sense of filtering theory. The special form of the drift in \eqref{stateeq1} involves no loss of generality since, for technical reasons, we will assume $\sigma$ to be invertible. 
The control objective is to maximize a reward functional
\begin{align}\label{rewardfunz1}
    J(\alpha)=
\E_\Q\left[\int_0^T f(Y_s,\Q_{Y_s,\alpha_s},\alpha_s,M_s)\,ds + g(Y_T,\Q_{Y_T},M_T)\right].
\end{align}
Suitable regularity and growth conditions on the coefficients $h,\sigma,f,g$ will be imposed; in particular, $h$ will be required to be bounded.
  
The combination of a hidden regime process with a control-and distribution-dependent observation mechanism places this problem outside the scope of classical partially observed control models. Indeed, standard filtering approaches for partially observed optimal control problems  rely on changes of probability measure, under which the distributions appearing in the coefficients are themselves modified. 
The main novelty of this paper is to show how this difficulty can be overcome in the present mean-field setting. We derive an equivalent separated control problem of McKean–Vlasov type, whose analysis involves specific technical challenges. In particular, the resulting dynamics do not satisfy standard Lipschitz conditions, and the associated dynamic programming approach leads to a Hamilton–Jacobi–Bellman equation posed on a convex subset of a Wasserstein space.

\paragraph{Related literature.} 
Mean-field and McKean–Vlasov control problems have been extensively studied in recent years, notably through stochastic maximum principles or programming methods, see for instance the monographs \cite{BenFreYam13}, \cite{CarmonaDelarue1}, \cite{CarmonaDelarue2}. Extensions to systems with regime-switching dynamics, where the coefficients are modulated by an exogenous Markov chain, have also been investigated, see e.g. \cite{BenDjeTemYam20}, \cite{NguYinNgu21}. 

Partial observation in mean-field control was introduced in \cite{BucLiMa17}, whose authors consider a stochastic control problem in which the dynamics and cost depend on the conditional law of the state given the observation filtration, and derive necessary optimality conditions via a system of conditional mean-field backward stochastic differential equations. In this framework, however, the observation process is assumed to be independent of both the control and the conditional distribution, and not to involve hidden regime dynamics. We also mention \cite{DjehicheTembine} for an approach to a mean-field control problem under partial observation based on a suitable stochastic maximum principle. For the discrete time case the reader may consult \cite{ChichKh23} and \cite{WanWangXiong25}.
 
More generally, partially observed stochastic control problems can be addressed either via direct stochastic maximum principle methods, see \cite{Tang98}, or through dynamic programming techniques combined with stochastic filtering. 
In the latter approach, the control problem is reformulated in terms of the filter or belief process, often via a change of probability measure leading to Zakai-type equations; see, for instance, \cite{BensoussanFiltering, ElliottFiltering, Nisio15book}. While alternative formulations may be available in specific models, notably in certain financial applications \cite{CohMerKnoch25, coxkallbladWang2025, RiederBauerle, FreyGabihWunderlich},  
the filtering-based dynamic programming approach remains the standard methodology.

Such methods are well understood for classical partially observed diffusion control and for hidden Markov models with control-independent observations. 
Their extension to mean-field control problems involving regime-switching dynamics and observation processes that depend on the control and on the conditional distribution remains largely unexplored. The present paper contributes to this direction by developing a dynamic programming framework for a mean-field control problem under partial observation, where the system is driven by a hidden Markov chain and the observation process depends on the control and on the conditional joint distribution of state and control.

\paragraph{Overwiew of our method and main contributions.}
Postponing precise statements and technical details, we briefly describe the main ideas of our approach. 
The starting point is a change of probability measure. Defining the exponential martingale process $L$ as
\begin{align*}
    L_{s}^{-1} =\exp\left(
-\int_0^s h(Y_r,\Q_{Y_r,\alpha_r},\alpha_r,M_r)\,dB_r -\frac12 \int_0^s| h(Y_r,\Q_{Y_r,\alpha_r},\alpha_r,M_r)|^2\,dr
    \right), \qquad s\in [0,T], 
\end{align*}
and introducing the so-called reference probability $\P$ by 
 $d\P=L_T^{-1}d\Q$,  Girsanov's theorem yields  a Brownian motion 
\[
W_s=B_s+\int_0^s h(Y_r,\Q_{Y_r,\alpha_r},\alpha_r,M_r)\,dr, \qquad s\in [0,T],
\]
under $\P$. Under this measure, the regime process $M$ remains a Markov chain with the same transition matrix, and the observation process satisfies 
\begin{align*}
dY_s &=\;  \sigma(Y_s)\,dW_s.  
\end{align*}
As a consequence, the filtration generated by $Y$ coincides with the filtration generated by $W$, and the control is adapted to this filtration.

We then  introduce  the unnormalized conditional distribution of the hidden regime
\begin{align*}
\rho_s(\phi)=\E\,[\phi(M_s)\,L_s\,|\, \calf^W_s], \qquad s\in[0,T],
\end{align*}
which satisfies a controlled Zakai equation.  The original reward functional can be rewritten under $\P$ in terms of $\rho$. Moreover, the conditional laws appearing in the coefficients can be expressed as weighted versions of the corresponding laws under $\P$, leading to a formulation entirely under the reference probability.

Since the regime space $S$ is finite of cardinality $N$, the Zakai equation can be rewritten as a finite-dimensional stochastic differential equation with values in $\R_+^N$. This yields a reformulated control problem with full observation, driven by the Brownian motion $W$, whose state variables are the pair $(Y,X)$, where $X$ represents the unnormalized filter with components $X_s^i$ $=$ $\rho_s(1_{\{i\}})$, $i$ $\in$ $S$.

A key difficulty remains: although the reformulated problem is fully observed, the driving Brownian motion and the admissible control set depend on the control itself - a typical feature of approaches based on changes of measure. To overcome this issue, we introduce a separated control problem formulated directly on the reference probability space. Under suitable assumptions, we prove that this separated problem is equivalent, in a weak sense, to the original partially observed control problem.

The separated problem is then studied using dynamic programming. Due to the special structure of the Zakai equation, the resulting controlled dynamics are neither linear nor globally Lipschitz, requiring a careful well-posedness and stability analysis. We introduce a lifted value function, prove a law-invariance property, and define a value function on a convex subset of a Wasserstein space. Using these results, we establish a dynamic programming principle and, following similar arguments as  in \cite{DjePosTan22} or  \cite{CosGozKhaPhaRos23}, characterize the value function as a viscosity solution to a Hamilton--Jacobi--Bellman equation posed on this domain. 
We emphasize that the special structure of our problem leads to the introduction of an HJB equation which is not defined on a whole Wasserstein space, contrary to the large majority of existing results. Moreover, in order to deal with significant examples, the integrability order of the Wasserstein space may be larger that $2$ (and even different for different marginals).
Uniqueness of viscosity solutions is not addressed in this paper and requires a separate treatment, possibly under  additional assumptions as in \cite{CosGozKhaPhaRos24}, \cite{BayEkrHeZha26}.

\paragraph{Outline of the paper.} 
The paper is organized as follows. In Section \ref{sec-separated} we introduce the separated control problem and state the standing assumptions. We then establish its equivalence with the original partially observed mean-field control problem presented in the introduction. 
Section 3 develops a dynamic programming approach for the separated problem. We prove well-posedness, introduce a lifted value function, and derive stability results despite the lack of Lipschitz continuity. We then establish law invariance, continuity of the value function on a convex Wasserstein domain, and a dynamic programming principle.
In Section \ref{sec-HJB} we study the associated Hamilton–Jacobi–Bellman equation. We provide a verification theorem and prove that the value function is a viscosity solution to the HJB equation.
 In Section \ref{SectionExamples}, we present several examples illustrating the generality of our assumptions and covering a number of significant cases; we collect in the Appendix \ref{AppendixProofs} the proofs of several technical results.

\section{The separated control problem}\label{sec-separated}

 \subsection{Assumptions and formulation}
Let \(q \geq 1\). In what follows, for any metric space \((\cals, d_{\cals})\) we denote by \(\calp_q(\cals)\) the Wasserstein space of order \(q\), defined as the space of Borel probability measures \(\mu\) on \(\cals\) such that \(||\mu||_q^q:=\int_{\cals} d_S(x,x_0)^q \, \mu(dx)\) is finite for any \(x_0 \in \cals\). We endow \(\calp_q(\cals)\) with the Wasserstein distance \(\calw_q\) defined as follows: for any \(\mu\), \(\nu \in \calp_q(\cals)\) 
\[ \calw_q(\mu,\nu):= \inf_{\substack{\pi \in \calp_q(\cals \times \cals): \pi(\cdot \times \cals)=\mu(\cdot), \\ \pi(\cals \times \cdot)=\nu(\cdot)}} \bigg( \int_{\cals \times \cals} d_{\cals}(x,y)^q \pi(dx,dy)  \bigg)^{\frac{1}{q}}.  
\]

\begin{Hp}[$\boldsymbol{A.1}$] \label{HpA1}
\begin{enumerate}
\item[] 
    \item[(i)] 
 $(\Omega,\calf,\P)$ is a complete probability space.
    \item[(ii)] $M$ is a  continuous-time Markov chain  on a finite set $S$ with cardinality $N$, having transition rates $\lambda(i,j)$ ($i,j\in S, i\neq j)$  and initial distribution $\pi_0$. 
We define   $\lambda(i,i)=-\sum_{j\neq i}\lambda(i,j)$, and denote by   $\Lambda$   the $N\times N$ matrix with entries $\lambda(i,j)$.
    \item[(iii)] 
$W$ is a standard Brownian motion in $\R^d$, independent of $M$.
    \item[(iv)] 
$A$, the set of control actions, is a Polish space with a  bounded metric (for instance, a bounded Borel set of a Euclidean space). $T>0$ denotes the time horizon of the control problem.
     \item[(v)] 
For some $q\ge 1$,
the function
\(
 h : \R^d\times \calp_q(\R^d\times A)\times A \times S\to \R^d
\)
is measurable and bounded: there exists $C\ge0$ such that
\[
|h(y,\mu,a,i)|\le C, \qquad
y\in\R^d,\,\mu\in \calp_q(\R^d\times A),\,a\in A,\,i\in S. 
\]
 \item[(vi)] 
The functions
\[
 f: \R^d\times \calp_q(\R^d\times A)\times A \times S\to \R,\qquad
g:  \R^d\times \calp_q(\R^d\times A)  \times S\to  \R,
\]
are measurable and satisfy the following growth condition: there exist  constants $C,\ell\ge0$ and an increasing function $\chi:\R_+\to\R_+$ such that
\begin{align}
\label{growthfg}
&|f( y, \nu , a,i)| + |g( y, \nu , i)|\le  C\,\Big(1+|y|^\ell+\chi(\|\nu\|_{W_q})\Big),
\end{align}
for every $\nu \in \calp_q(\R^d\times A)$, $y\in\R^d$, $a\in A$,  $i\in S$.

 \item[(vii)]  
The function
\(
\sigma: \R^d\to\R^{d\times d}
\)
is globally Lipschitz: there exists $C\ge0$ such that
\[
|\sigma(y)-\sigma(y')|\le C\,|y-y'|, \qquad y,y,'\in\R^d.
\]
Moreover,  the   matrix $\sigma(y)$ is invertible for every $y\in\R^d$.

\item[(viii)]
$y_0\in\R^d$ is    a (deterministic) initial state.

\end{enumerate}
\end{Hp}

\bigskip

We define the set $\cala^W$ of admissible control processes as the collection of all processes $\alpha:\Omega\times [0,T]\to A$ which are predictable with respect to $\F^W=(\calf^W_t)_{t\ge0}$, the completed Brownian  filtration of $W$.

It is convenient to adopt the following notation. We identify $f(y,\mu,a,\cdot)$ and $g(y,\mu,\cdot)$ with  $\R^N$-valued functions setting, for  $y\in\R^d$, $a\in A$ and  $\mu\in \calp_q(\R^d \times A)$,
\[
f(y,\mu,a)=
(f(y,\mu,a,i))_{i},
\quad
g(y,\mu )=
(g(y,\mu ,i))_{i},
\]
and we identify $h$ with the $\R^{N\times d}$-valued function
\[
 h\,(y,\mu,a)=
(  h^k(y,\mu,a,i))_{i,k},
\]
where  $h^k$ ($k=1,\ldots,d$) are the components of  $h$.

We are ready to state the separated optimal control problem: the observation process $Y$ is defined as the solution to
\begin{align} \label{eq:1Y}
\left\{
\begin{array}{rcl}dY_s&=&\sigma(Y_s)\,dW_s,\qquad s\in [0,T],
    \\
    Y_0&=&y_0.
    \end{array}
    \right.
\end{align}
Since $\sigma$ is assumed to be globally Lipschitz and invertible, the (completed) filtrations   generated by $Y$ and by $W$ are the same.
For $\alpha\in\cala^W$  the controlled   equation is
\begin{align}\label{eq:1}
\left\{
\begin{array}{rcl}
dX_s &=& \Lambda^T X_s  \,ds + diag(X_s)\,h \big(Y_s,\E\,\left[\langle X_s,1_N\rangle\,|\, (Y_s,\alpha_s)=(y,a))\right]\P_{Y_s,\alpha_s}(dy\,da),\alpha_s\big) dW_s,
    \\
    X_0&=&\pi_0,\qquad\qquad\qquad\qquad\qquad\qquad\qquad\qquad\qquad\qquad\qquad s\in [0,T],
    \end{array}
    \right.
\end{align}
and the reward functional is 
\begin{align}\label{funz:1}
J(\alpha)&= 
\nonumber
\; \E\bigg[\int_0^T \langle X_s ,f\big(Y_s, \E\,\left[\langle X_s,1_N\rangle\,|\, (Y_s,\alpha_s)=(y,a)\right]\,\P_{Y_s,\alpha_s}(dy\,da),\alpha_s\big)\rangle\,ds 
\\&\qquad \qquad + \langle X_T , g\big(Y_T, \E\,\left[\langle X_s,1_N\rangle\,|\, Y_T=y\right]\,\P_{Y_T}(dy)\big)\rangle \bigg].
\end{align}
Here, $\Lambda^T$ is the transpose matrix of $\Lambda$,  $diag(X_s)$ denotes the $N\times N$ matrix with the vector $X_s$ on the diagonal and the other entries equal to zero;   moreover, $1_N$ is the vector in $\R^N$ with all entries equal to \(1\) and   $\langle\cdot,\cdot\rangle$ is the scalar product in $\R^N$. We denote by 
\begin{align}
(y,a) \mapsto \E\,\left[\langle X_s,1_N\rangle\,|\, (Y_s,\alpha_s)=(y,a)\right]
\end{align}
any  measurable function $q_s:\R^d\times A\to\R$ 
such that $\E\,\left[\langle X_s,1_N\rangle\,|\, Y_s,\alpha_s\right]=q_s(Y_s,\alpha_s)$; $q_s$ is defined up to a $\P_{Y_s,\alpha_s}$-null set. 
We use the notation  
$$
\E\,\left[\langle X_s,1_N\rangle\,|\, (Y_s,\alpha_s)=(y,a)\right]\, \P_{Y_s,\alpha_s}(dy\,da)
\qquad \text{or simply}\qquad
\E\,\left[\langle X_s,1_N\rangle\,|\, (Y_s,\alpha_s)=\cdot\right]\, \P_{Y_s,\alpha_s} 
$$
for the probability measure on $\R^d\times A$ admitting the density $(y,a)\mapsto\E\,\left[\langle X_s,1_N\rangle\,|\, (Y_s,\alpha_s)=(y,a)\right]$ with respect to $  \P_{Y_s,\alpha_s}(dy\,da)$, provided that  $\E[\langle X_s,1_N\rangle]$ $=$ $1$. We define in a similar way terms like $\E\,\left[\langle X_s,1_N\rangle\,|\, Y_T=y\right]\,\P_{Y_T}(dy)$.
We note that in \eqref{eq:1} the initial distribution $\pi_0$
is identified with a vector in $\R^N_+$ satisfying $\langle \pi_0,1_N\rangle=1$.
\begin{Remark}
The function $q_s$ introduced above is not unique, but it is rather determined only up to $\P_{Y_s,\alpha_s}$-almost sure equality. This raises the question of measurability properties of the coefficients in \eqref{eq:1}-\eqref{funz:1}. We will see later, in paragraph \ref{subsec:spacesmeas}, that the measure 
$\E\,\left[\langle X_s,1_N\rangle\,|\, (Y_s,\alpha_s)=\cdot\right]\, \P_{Y_s,\alpha_s}$ can be written in the form $\Gamma(\P_{X_sY_s,\alpha_s})$ for an appropriate function $\Gamma$. The  explicit form of $\Gamma$ yields  well-posedness of the coefficients in \eqref{eq:1}-\eqref{funz:1}. 
\end{Remark}

In what follows, on the function $h$ we are going to impose a suitable form of Lipschitz condition, presented below. 
Recall that the number   $q\ge1$ and the function $\chi$ were introduced in Assumptions (\nameref{HpA1}).

\bigskip

\begin{HpSingle}[$\boldsymbol{A.2}$] \label{HphLipschitz}
    Let $r\ge2$. There exists a  constant $L\ge 0$ such that, whenever $\rho\in\calp_r(\R^d\times A)$,  and  $\nu_1,\nu_2\in\calp_q(\R^d\times A)$ are absolutely continuous with respect to $\rho$ with densities $\phi_1,\phi_2$ respectively,   the inequality
\begin{align*}
  &  \left| h( y, \nu_1 , a,i)- h( y, \nu_2 , a,i)\right|
\\&\quad\quad\le L\left(\int_{\R^d\times A} |\phi_1(y',a)-\phi_2(y',a)|^2\,  \rho(dy'\,da)\right)^{1/2}
    \cdot\Big( 1+ \chi(\|\rho\|_{\calw_r}) +\chi(\|\nu_1\|_{\calw_q})+\chi(\|\nu_2\|_{\calw_q} 
\Big)
\end{align*}
holds for every $y\in\R^d$, $a\in A$, $i\in S$. 
\end{HpSingle} 

\bigskip

Here we use the notation 
\[
\|\rho\|_{\calw_r}=\calw_r(\rho,\delta_{(0,a_0)})=\left(\int_{\R^d\times A} (|y|+d(a,a_0))^r\,  \rho(dy\,da)\right)^{1/r},
\]
(where $a_0\in A$ is any given point  and \(d\) denotes a complete metric on \(A\)) 
and similarly for $\|\nu_1\|_{\calw_q}$, $\|\nu_2\|_{\calw_q} $. In Section \ref{SectionExamples} the reader will find examples where Assumption (\nameref{HphLipschitz}) holds true. The occurrence of the terms containing the function $\chi$ is required to cover important concrete examples. However we note  that, in spite of the boundedness of $h$ and Assumption (\nameref{HphLipschitz}), the controlled state equation does not satisfy a global Lipschitz condition, even in the simpler case $\chi=0$.

The next result will be proved later in greater generality and we assume it for the moment: compare Theorem \ref{existstateeqprovv} and Proposition \ref{Jwelldef}.

\begin{Theorem} \label{existstateeq} Suppose that $p\ge 2$ and $r\ge 2$ satisfy $q\le r(1  -  p^{-1})$. Under Assumptions (\nameref{HpA1}) and (\nameref{HphLipschitz}), for every 
$\alpha\in\cala^W$    there exists a unique solution $(X,Y)$ to equations \eqref{eq:1Y}-\eqref{eq:1}, $\F^W$-adapted with continuous trajectories, such that $\langle X,1_N\rangle$ is a martingale and  satisfying
\[
    \E\, \Big[\sup_{s\in [0,T] }|Y_s|^r\Big]<\infty,\quad     
    \E\, \Big[
    \sup_{s\in [0,T] }|X_s|^p\Big]<\infty,\quad 
\P(X_s\in\R^N_+)=1, \quad \E\,[ \langle X_s,1_N\rangle]=1, \qquad s\in [0,T].
    \]
Moreover, if the constant $\ell$ in \eqref{growthfg} satisfies  $\ell\le r(1-p^{-1})$  then the reward functional $J(\alpha)$ is well defined and finite.
\end{Theorem}

Under the conditions of this Theorem,  the
separated problem is well defined and we can introduce its value 
\begin{equation}
\label{valueseparatedoriginal}
V_0=\sup_{\alpha\in\cala^W}
J(\alpha).
\end{equation}

 \subsection{Equivalence with the partially observed control problem}

In this paragraph we clarify the connection of the separated control problem and the original partially observed optimal control problem 
\eqref{stateeq1}-\eqref{rewardfunz1}.
Throughout, we suppose that Assumptions (\nameref{HpA1}) and (\nameref{HphLipschitz}) hold true and that the constants  $\ell$ in \eqref{growthfg} and \(q \geq 1\) satisfy the conditions  $q, \ell\le r(1-p^{-1})$.   

Denote by $\F^M=(\calf^M_t)_{t\ge0}$ the filtration generated by the Markov chain $M$ and by $\G=(\calg_t)_{t\ge0}$ the filtration defined setting  $\calg_t=\calf^W_t\vee \calf^M_t$. Using the independence of $W$ and $M$ it is not difficult to verify that $W$ is a Wiener process with respect to $\G$.
For every $\alpha\in\cala^W$, we denote by $X^\alpha$ the solution to the state equation   \eqref{eq:1}
given by  Theorem \ref{existstateeq}.
Let us also define the processes 
\[
h^\alpha_s= 
h \big(Y_s,\E\,\left[\langle X_s^\alpha,1_N\rangle\,|\, (Y_s,\alpha_s)=\cdot\right]\,\P_{Y_s,\alpha_s},\alpha_s,M_s\big), \qquad s\in [0,T],
\]
\[
L^\alpha_s=\exp\left(\int_0^s h^\alpha_r\,dW_r- \frac12 \int_0^s |h^\alpha_r|^2\,dr\right),
\qquad
B^\alpha_s=W_s-\int_0^s
h^\alpha_r\,dr,
\qquad s\in [0,T].
\]
The process $h^\alpha$ is progressively measurable with respect to $\G$, so that the stochastic integral in the definition of $L^\alpha$ is well defined. Since $h^\alpha$ is also bounded, $L^\alpha$ is a $\G$-martingale and we can define a probability $\Q^\alpha$ setting $d\Q^\alpha=L^\alpha_T\,d\P$. With this notation we have the following result.

\begin{Proposition}\label{equivpartialobs}
 \begin{enumerate}
     \item [(i)]
     Under $\Q^\alpha$, the process $B^\alpha$ is a $\G$-Brownian motion on $[0,T]$.
     \item [(ii)]      Under $\Q^\alpha$, the process $M$ remains a Markov chain with transition rates $\lambda(i,j)$ and initial distribution $\pi_0$.
     \item [(iii)] $M$ and $B^\alpha$ are indepedent under $\Q^\alpha$.
     \item [(iv)] 
     For $i\in S$ and $s\in[0,T]$ we have
     \begin{align}
X_s^i=\E\, [1_{\{M_s=i\}}\,L^\alpha_s\,|\, \calf^W_s].
     \end{align}
     \item [(v)] We have 
     \begin{align}
\E\,\left[\langle X_s^\alpha,1_N\rangle\,|\, (Y_s,\alpha_s)=\cdot\right]\,\P_{Y_s,\alpha_s}=
\Q^\alpha_{Y_s,\alpha_s},
\end{align}
and the process $Y$ satisfies 
\begin{align}\label{stateeq1bis}
    \left\{
\begin{array}{rcl}
dY_s&=& \sigma(Y_s)\,h(Y_s,\Q^\alpha_{Y_s,\alpha_s},\alpha_s,M_s)\,ds + \sigma(Y_s)\,dB^\alpha_s,\qquad s\in   [0,T],
\\
Y_0&=&y_0.
    \end{array}
    \right.
\end{align}
\item[(vi)] 
The reward functional $J(\alpha)$ 
defined in \eqref{funz:1} satisfies
\begin{align}\label{rewardfunz1bis}
    J(\alpha)=
\E_{\Q^\alpha}\left[\int_0^T f(Y_s,\Q^\alpha_{Y_s,\alpha_s},\alpha_s,M_s)\,ds + g(Y_T,\Q^\alpha_{Y_T},M_T)\right].
\end{align}

     \end{enumerate}   
\end{Proposition}

\begin{proof}
The proofs of $(i)$-$(iii)$ essentially follow from Proposition 3.13 in \cite{BainCrisan}. We sketch a proof for the reader's convenience. 

 \begin{enumerate}
\item [(i)] This follows immediately from the Girsanov theorem.
\item [(ii)] Since $L^\alpha_0=1$, $\P$ and $\Q^\alpha$ agree on $\calg_0$. It follows that the initial distribution of $M$ under $\P$ and $\Q^\alpha$ is the same and it is equal to $\pi_0$.

Let $T_n$ denote the jumps times of the Markov chain $M$ and, for every $j\in S$, let us define the counting processes associated with $M$:
\[
N^j_s= \sum_{n\ge 1} 1_{T_n\le s }\,1_{M_{T_n}=j}, \qquad s\ge 0.
\]
Let us write $\lambda(i,j)$ instead of $\lambda_{ij}$, for typographical convenience. It is known that, given the initial distribution, the law of $M$ is entirely determined by the fact that  the compensated counting processes
\[
m^j_s:=N^j_s- \int_{0}^s  \lambda(M_{r-},j)\,dr, \qquad s\ge 0,
\]
are $\F^M$-martingales   under $\P$. By the independence of $M$ and $W$,  $m^j$ are also $\G$-martingales. 
Since $L^\alpha$ is a continuous $\G$-martingale   while $m^j$ have finite variation trajectories, they are orthogonal, meaning that their products $m^jL^\alpha$ are $\G$-martingales under $\P$. Then a standard verification shows that 
$m^j$ are $\G$-martingales under $\Q^\alpha$. Since $m^j$ are adapted with respect to $\F^M$ we conclude that  $m^j$ are $\F^M$-martingales under $\Q^\alpha$. By the characterization property mentioned above,   $M$ is a Markov chain with transition rates $\lambda(i,j)$ under $\Q^\alpha$.

\item[(iii)] Let us define another filtration
$\H=(\calh_t)_{t\ge0}$ setting $\calh_t=\calf^W_t\vee\calf^M_\infty$. 
Since $\H$ is an enlargement of $\G$, the process $h^\alpha$ is also progressively measurable with respect to $\H$.
Since  $W$ is independent of   $M$,  $W$ is a Wiener process with respect to $\H$ as well, under $\P$. We may now apply the Girsanov theorem in the filtration $\H$ and conclude that 
$B^\alpha$ is a Wiener process with respect to $\H$  under $\Q^\alpha$, and in particular independent of $\calh_0$. Since $\calf^M_\infty\subset \calh_0$ (in fact, they are equal up to null sets) we conclude that $B^{\alpha}$ is independent of  $M$.

\item[(iv)] By the definition of $B^\alpha$ we have
$dY_s= \sigma(Y_s) \big(h^\alpha_s\,ds+dB^\alpha_s\big)$, namely $Y$ satisfies
\begin{align}\label{obsproc}
\left\{
\begin{array}{rcl}
    dY_s &=& \sigma(Y_s) \Big( h \big(Y_s,\E\,\left[\langle X_s^\alpha,1_N\rangle\,|\, (Y_s,\alpha_s)=\cdot\right]\,\P_{Y_s,\alpha_s},\alpha_s,M_s\big)\,ds + dB^\alpha_s \Big),
    \\
    Y_0&=&y_0.
    \end{array}
    \right.
\end{align}
We look at the processes $M,Y$ from the point of view of filtering theory: in the probability space $(\Omega,\calf,\Q^\alpha)$, $M$ is the (unobserved) signal process and $Y$ is the observation process, which generates the filtration $\F^W$. The processes  $M,Y,\alpha,B^\alpha$ are $\G$-progressive, $M$ is a Markov chain and $B^\alpha$ is an independent Brownian motion. We note that
\[
(L^\alpha_s)^{-1}=\exp\left(-\int_0^s h^\alpha_r\,dB^\alpha_r- \frac12 \int_0^s |h^\alpha_r|^2\,dr\right),
\qquad s\in [0,T],
\]
is a $\G$-martingale under $\Q^\alpha$
and 
$d\P=(L^\alpha_T)^{-1}\, d\Q^\alpha$.  For $i\in S$,
let us denote by $\rho_s$ the unnormalized conditional law of  $M_s$, namely 
$\rho_s^i=\E\, [1_{\{M_s=i\}}\,L^\alpha_s\,|\, \calf^W_s].
$  More precisely, we define $\rho^i$ to be the optional projection of $1_{\{M=i\}}\,L^\alpha$ on the filtration $\F^W$. Noting that $\alpha$ is $\F^W$-predictable, we can deduce in the usual way that the Zakai filtering equation holds: it takes the form,
 for every $i\in\ S$,
\begin{align*}
\left\{
\begin{array}{rcl}
    d\rho_s^i &=& \sum_{j\in S}\lambda_{ji}\,\rho^j_s  \,ds + \rho_s^i\, h \Big(Y_s,\E\,\left[\langle X_s^\alpha,1_N\rangle\,|\, (Y_s,\alpha_s)=\cdot\right]\,\P_{Y_s,\alpha_s},\alpha_s,i\Big)\,dW_s,
    \\
    \rho_0&=&\pi_0.
    \end{array}
    \right.
\end{align*}
Since this equation coincides with the equation \eqref{eq:1} satisfied by $X^\alpha$ we conclude that $\rho$ coincides with $X^\alpha$ and point (iv) is proved.

\item[(v)] In view of \eqref{obsproc} it is enough to prove the equality
\begin{align}\label{identiflaw}
\E\,\left[\langle X_s^\alpha,1_N\rangle\,|\, (Y_s,\alpha_s)=\cdot\right]\,\P_{Y_s,\alpha_s}=
\Q^\alpha_{Y_s,\alpha_s}.
\end{align}
For any bounded measurable $\phi:\R^d\times A\to\R$,
\begin{align*}
\int_{\R^d\times A}\phi(y,a)\,\Q^\alpha_{Y_s,\alpha_s}(dy\,da) = \E_{\Q^\alpha}[\phi(Y_s,\alpha_s)]= \E\, [\phi(Y_s,\alpha_s)L^\alpha_s]=
\E\, [\phi(Y_s,\alpha_s)\,\E\,[L^\alpha_s\,|\,\calf_s^W]]    .
\end{align*}
From point (iv) it follows that 
$\E\,[L^\alpha_s\,|\,\calf_s^W]= \langle X_s^\alpha,1_N\rangle$ and we obtain
\begin{align*}
\int_{\R^d\times A}\phi(y,a)\,\Q^\alpha_{Y_s,\alpha_s}(dy\,da) &= \E\, [\phi(Y_s,\alpha_s)\,\langle X_s^\alpha,1_N\rangle]
= \E\, [\phi(Y_s,\alpha_s)\,\E\,[\langle X_s^\alpha,1_N\rangle\,|\,Y_s,\alpha_s]]
\\&=
\int_{\R^d\times A}\phi(y,a)\,
\E\,[\langle X_s^\alpha,1_N\rangle\,|\,(Y_s,\alpha_s)=(y,a)]\,\P_{Y_s,\alpha_s}(dy\, da),
\end{align*}
which proves \eqref{identiflaw}.

\item[(vi)] Arguing as for \eqref{identiflaw} one can prove that
\begin{align}\label{identiflawdue}
\E\,\left[\langle X_T^\alpha,1_N\rangle\,|\, Y_T=\cdot\right]\,\P_{Y_T}=
\Q^\alpha_{Y_T}.
\end{align}
Then the conclusion follows by standard arguments, computing the right-hand side of \eqref{funz:1}  using conditioning, and taking into account point (iv)  and the equalities  \eqref{identiflaw}-\eqref{identiflawdue}.
\end{enumerate} 
\end{proof}

This proposition shows that the separated problem introduced above is equivalent to a weak version of the partially observed control problem. Indeed, for every admissible control $\alpha$, the reward functional $J(\alpha)$ of the separated problem can be written via equations 
\eqref{stateeq1bis}-\eqref{rewardfunz1bis}
which coincide with 
\eqref{stateeq1}-\eqref{rewardfunz1} except for the fact that the Brownian motion $B^\alpha$  as well as the probability $\Q^\alpha$ depend on $\alpha$. On the contrary, in the separated problem the process $W$ and the filtration $\F^W$ do not depend on $\alpha$ (which was not the case for the original partially observed control problem). 

In the rest of this paper we will deal with the separated control problem \eqref{eq:1Y}-\eqref{funz:1} or its extensions,   leaving aside the original partially observed problem, except in Proposition \ref{AeAW} and Remark \ref{verifoptpartialobs} where we explain how the separated problem gives a solution to the partially observed problem formulated above.

\section{Dynamic programming and value function}
\label{sec-dynprog}

\subsection{Dynamic programming}

In this section we start the study of  the separated control problem \eqref{eq:1Y}-\eqref{funz:1} by the dynamic programming approach. The first step consists in embedding the separated problem into a larger family of optimization problems. First, we will consider a possibly different starting time $t\in[0,T]$ instead of time $0$. Second, in order to deal with the dependence of the coefficients on the marginal laws of the controlled process, we also need to replace the deterministic initial condition $y_0$ and $\pi_0$ (for the processes $Y$ and $X$) by random variables  $\xi$ and $\eta$. It is convenient to require that they depend on an additional source of noise and for this reason, beside Assumptions (\nameref{HpA1}) and (\nameref{HphLipschitz}), we will assume the following.

\bigskip

\begin{HpSingle}[$\boldsymbol{A.3}$] \label{HpU}
There exists a random variable $U$ having uniform distribution on $(0,1)$ and independent of $M,W$ defined
on the probability space  $(\Omega,\calf,\P)$.

We introduce the filtration $\F=(\calf_t)_{t\ge0}$, where $\calf_t=\calf_t^W\vee \sigma(U)$, and we denote by $\cala$  the collection of all $\F$-predictable processes $\alpha:\Omega\times [0,T]\to A$.
\end{HpSingle}
\bigskip

In what follows, it is convenient to consider the larger class $\cala$ as the class of admissible control processes.
We   rewrite the state equations \eqref{eq:1Y}-\eqref{eq:1}    incorporating the previous  extensions: for any  time $t\in [0,T]$ and any $\alpha\in\cala$, 
we consider the following system for a controlled pair $(X_s,Y_s)_{s\in [t,T]}$ with values in $\R^N\times \R^d$:
\begin{align} \label{provveq:1}
\left\{
\begin{array}{rcl}
    dX_s &=& \Lambda^T X_s  \,ds + diag(X_s)\,h \Big(Y_s,\E\,\left[\langle X_s,1_N\rangle\,|\, (Y_s,\alpha_s)=(y,a)\right]\,\P_{Y_s,\alpha_s}(dy\,da),\alpha_s\Big)\,dW_s,
\\dY_s&=&\sigma(Y_s)\,dW_s,
    \\
    X_t&=&\eta,
    \\
    Y_t&=&\xi,
    \end{array}
    \right.
\end{align}
and the reward functional
\begin{align}\label{provvfunz:1}
J(t,\eta,\xi,\alpha)&=
\nonumber
\E\bigg[\int_t^T \big\langle X_s ,f\big(Y_s, \E\,\left[\langle X_s,1_N\rangle\,|\, (Y_s,\alpha_s)=(y,a)\right]\,\P_{Y_s,\alpha_s}(dy\,da),\alpha_s\big)\big\rangle\,ds 
\\&\quad + \big\langle X_T , g\big(Y_T, \E\,\left[\langle X_T,1_N\rangle\,|\, Y_T=y\right]\,\P_{Y_T}(dy)\big)\big\rangle \bigg].
\end{align}
The equations 
\eqref{provveq:1}-\eqref{provvfunz:1} provide the form of the separated problem required for dynamic programming arguments.
Note that, in order to reduce to a Markovian setting, we introduce the process $Y$ as another component of the state, although it is not affected by the control.
The   random initial condition $(\eta,\xi)$, valued in $\R^N\times \R^d$, is assumed to be   $\calf_t$-measurable and to satisfy appropriate integrability conditions stated below.

In what follows, we use the symbol $\|\cdot\|_p$ to denote the usual $L^p$ norm for any random variable with values in Euclidean space.
Our first result is  well-posedness of the state equation 
\eqref{provveq:1}:

\begin{Theorem} \label{existstateeqprovv} Suppose that Assumptions (\nameref{HpA1}), (\nameref{HphLipschitz}), (\nameref{HpU}) hold true and assume that $p\ge 2$ and $r\ge 2$ satisfy
\[
q\le r\left(1-\frac{1}{p}\right).
\]
Let $(\eta,\xi)$ be $\R^N\times \R^d$-valued and $\calf_t$-measurable satisfying 
\[
 \|\xi\|_r<\infty, \qquad \|\eta\|_p <\infty, \qquad
\P(\eta\in\R^N_+)=1,\qquad
\E\,[\langle \eta,1_N\rangle]=1.
\]

Then, for every 
$\alpha\in\cala$,    there exists a unique solution $(X,Y)$ to equation \eqref{provveq:1}, where  $X$ and $Y$ are $\F$-adapted with continuous trajectories,  $\langle X,1_N\rangle$ is a martingale,  and 
    \[
    \E\, \Big[\sup_{s\in [t,T] }|Y_s|^r\Big]<\infty,\quad 
    \E\, \Big[\sup_{s\in [t,T] }|X_s|^p\Big]<\infty,\quad 
    \P(X_s\in\R^N_+)=1, \quad \E\,[ \langle X_s,1_N\rangle]=1, \qquad s\in [t,T].
    \]
Moreover there exists a constant $C_0$, depending only on $p$, $r$, $T$, $\Lambda$, $\sup |h|$ and the Lipschitz constant of $\sigma$, such that
\begin{align}
    \label{stap}
    \E\, \Big[\sup_{s\in [t,T] }|X_s|^p\Big]&\le C_0\, \|\eta\|_p^p,
\\
    \label{stapY}
    \E\, \Big[\sup_{s\in [t,T] }|Y_s|^r\Big]&\le  C_0\,(1+ \|\xi\|_r^r).
\end{align}
\end{Theorem}

\begin{proof}
We firstly introduce the following notation: denote by \(\cals^p\) the Banach space of continuous, \(\F\)-adapted processes \(Z\) such that 
\[ ||Z||_{\cals^p}:= \Big( \E \Big[ \sup \limits_{s \in [t,T]} |Z_s|^p \Big] \Big)^{\frac{1}{p}} < +\infty, \]
identified up to indistinguishability. The space  \(\cals^r\) is defined analogously.

Since \(\sigma\) is globally Lipschitz and \(\xi \in L^r(\Omega, \calf, \P)\), by standard well-posedness results for SDEs there exists a unique solution \(Y \in \cals^r\) to the second equation in \eqref{provveq:1}. Moreover, \(Y\) satisfies estimate \eqref{stapY}.

Now denote by \(\cals_0^p\) the closed subset of \(\cals^p\) containing the processes \(Z\) such that \(\langle Z,1_N \rangle\) is a martingale with 
\begin{align}
    \label{martX}
    \P(Z_s\in\R^N_+)=1, \quad \E\,[ \langle Z_s,1_N\rangle]=1, \qquad s\in [t,T].
\end{align}
It is thus sufficient to prove that there exists a unique \(X \in \cals_0^p\) such that, $\P$-{a.s.},
\begin{equation} \label{Xdaptofisso}
    X_s= \eta + \int_{t}^s \Lambda^T X_r dr + \int_{t}^s diag(X_r) h(Y_r, \E[\langle X_r,1 \rangle|(Y_r,\alpha_r)=(y,a)] \P_{Y_r,\alpha_r}(dy,da), \alpha_r)\,dW_r,
\end{equation}
for  $ s\in[t,T]$. 
We first verify that the measure 
$\E\,\left[\langle X_s,1_N\rangle\,|\, (Y_s, \alpha_s)=(y,a)\right]\,\P_{Y_s, \alpha_s}(dy, da)$
belongs to $\calp_q(\R^d \times A)$ whenever $X\in\cals^p_0$, \(Y \in \cals^r\). Indeed, it is a probability measure by \eqref{martX}. Moreover, we have 
\begin{align}
 \Big\| \E\,\left[\langle X_s,1_N\rangle\,|\, (Y_s, \alpha_s)=\cdot\right]\,\P_{Y_s, \alpha_s} \Big\|_{\calw_q}^q\! &=\!
    \int_{\R^d} (|y| + d(a,a_0))^q \,\E\,\left[\langle X_s,1_N\rangle\,|\, (Y_s, \alpha_s)=(y,a)\right]\,\P_{Y_s, \alpha_s}(dy, da) \nonumber \\
    &=\nonumber
  \E\,\left[ (|Y_s| +d(\alpha_s,a_0))^q   \,\E\,\left[\langle X_s,1_N\rangle\,|\, Y_s, \alpha_s\right]\,\right]
   \\
    &=
  \E\,\left[ (|Y_s| +d(\alpha_s,a_0))^q   \, \langle X_s,1_N\rangle \right] \nonumber
  \\& \le 2^{q-1} 
  \E\,\left[ |Y_s|^q   \, \langle X_s,1_N\rangle \right]
  +2^{q-1}
  \E\,\left[ d(\alpha_s,a_0)^q   \, \langle X_s,1_N\rangle \right], \label{MomentofqsPs}
\end{align}
where $a_0\in A$ is any given point. The second summand is finite since the metric on \(A\) is bounded. For the first summand we note that, by the H\"older  inequality, 
\begin{align}
    \E\,\left[ |Y_s|^q   \, \langle X_s,1_N\rangle \right] &\leq \bigg( \E\,\left[|\langle X_s,1_N\rangle|^p \right] \bigg)^{\frac{1}{p}} \bigg( \E\,\left[ |Y_s|^\frac{qp}{p-1} \right] \bigg)^{\frac{p-1}{p}} \nonumber \\
    &\leq \bigg( \E\,\left[|X_s|^p\right] \bigg)^{\frac{1}{p}}  \bigg( \E\,\left[ |Y_s|^r\right]^\frac{qp}{r(p-1)} \bigg)^\frac{p-1}{p} \nonumber \\
    &\leq ||X||_{\cals^p} ||Y||_{\cals^r}^q<\infty, \label{FirstSummandMomentofqsPs}
\end{align}
where we used  the hypothesis \(q\le r\left(1-\frac{1}{p}\right)\) which implies that \(\frac{r(p-1)}{qp} \geq 1\).

Now let us define a map \(\Phi: \cals_0^p \rightarrow \cals^p\) such that \(H=\Phi(X)\) satisfies the following equation:
\begin{equation} \label{DefPhi}
    H_s= \eta + \int_t^s \Lambda^T H_r dr + \int_t^s diag(H_r) h(Y_r, \E [\langle  X_r,1 \rangle | (Y_r, \alpha_r)=(y,a) ] \P_{Y_r,\alpha_r}(dy,da), \alpha_r) dW_r.
\end{equation}
\(\Phi\) is well defined because it maps \(X\) into the unique solution \(H\) to a linear SDE (with \(h\) bounded) with initial condition \(\eta \in L^p\).  Noting that $\langle \Lambda^T H_r,1_N\rangle= \langle  H_r,\Lambda1_N\rangle= 0$, since $\Lambda1=0$, it can be proved that $\langle H,1_N\rangle$ is a martingale with expectation equal to $1$. We next verify the following  Lemma, whose proof is postponed:  

\begin{Lemma}\label{lemmainvariance}
    $\P(H_s\in\R^N_+)=1$.
\end{Lemma}

 By Lemma \ref{lemmainvariance} we conclude that  $H\in\cals^p_0$; as a consequence, $\Phi$  maps this space into itself.
Moreover, $X$ is a solution to \eqref{Xdaptofisso} if and only if it is a fixed point of $\Phi:\cals^p_0\to \cals^p_0$. We will use the following a priori estimate:

\begin{Lemma} \label{lemmastap}
There exists a constant $C_0$, depending only on $p$, $T$, $\Lambda$, $\sup |h|$, such that
\begin{align}
    \label{stapbis}
    \E\, \Big[\sup_{s\in [t,T] }|H_s|^p\Big]\le C_0\, \E\,[|\eta|^p].
\end{align} 
\end{Lemma}

\begin{proof}[ Proof of   Lemma \ref{lemmastap}]
Denoting by $C$ any constant that only depends on on $p$, $T$, $\Lambda$, $\sup |h|$ - not necessarily the same at any occurrence - we have
\begin{align*}
\sup_{u\in[t,s]} |H_u|^p &\le  C\,|\eta|^p+
C\, \left(\int_t^s |\Lambda^T H_r|  \,dr\right)^p 
\\&+ C\sup_{u\in[t,s]}\left| \int_t^u diag(H_r)\,h \Big(Y_r,\E\,\left[\langle X_r,1_N\rangle\,|\, (Y_r,\alpha_r)=(y,a)\right]\,\P_{Y_r,\alpha_r}(dy\,da),\alpha_r\Big)\,dW_r\right|^p,
\end{align*}
and by the Burkholder-Davis-Gundy and the H\"older inequalities, and the boundedness of $h$,
\begin{align*}
\E\sup_{u\in[t,s]} |H_u|^p &\le  C\,\E\,|\eta|^p+
C\,  \int_t^s \E\,|  H_r|^p  \,dr  
\\&+ C\,\E\left( \int_t^s \left| diag(H_r)\,h \Big(Y_r,\E\,\left[\langle X_r,1_N\rangle\,|\, (Y_r, \alpha_r)=(y,a)\right]\,\P_{Y_r, \alpha_r}(dy, da),\alpha_r\Big)\right|^2\,dr\right)^{p/2}
\\&\le C\,\E\,|\eta|^p+
C\,  \int_t^s \E\,|  H_r|^p  \,dr+ C\,\E\left( \int_t^s  |  H_r|^2\,\,dr\right)^{p/2}
\\&\le C\,\E\,|\eta|^p+
C\,  \int_t^s \E\,|  H_r|^p  \,dr.
\end{align*}
The conclusion follows from the Gronwall Lemma.
  
\end{proof}

We continue the proof of the theorem. It is enough to show that $\Phi$ has a unique fixed point. Indeed,  in this case the estimate \eqref{stap} follows from the previous lemma. Let us define a closed subset of $\cals_0^p$ setting
\[
\calb=\{Z\in\cals_0^p\;:\; 
\E\, \Big[\sup_{s\in [t,T] }|Z_s|^p\Big]\le C_0\, \E\,[|\eta|^p]\}.
\]
Lemma \ref{lemmastap} implies that $\Phi(\cals_0^p)\subset \calb$ and hence that any fixed point of $\Phi$ belongs to $\calb$. We will then show that the restricted map $\Phi:\calb\to\calb$ has a unique fixed point.  
To this end we will prove that a suitable iteration $\Phi^n$ of the map $\Phi$ is a contraction for the norm $\|\cdot\|_{\cals^p}$. Take $X,X'\in\calb$ and note that, by the previous lemma and the already proved inequality \eqref{stapY}, we have
\begin{align}
    \label{XXprimoball}
\|X\|_{\cals^p}+\|X'\|_{\cals^p} +\|Y\|_{\cals^r}\le K,
\end{align}
where, here and below, we denote by $K$ a constant depending only on the constants in Assumptions (\nameref{HpA1}) and (\nameref{HphLipschitz}) as well as  on $\Lambda$,  $\|\eta\|_p$, $\|\xi\|_r$ and  the function $\chi$. The  value of $K$ may change from line to line. 
Now define $H=\Phi(X)$, $H'=\Phi(X')$, $\bar X=X-X'$, $\bar H=H-H'$. Then $\bar H_0=0$ and subtracting equation \eqref{Xdaptofisso} for \(H\), \(H'\), we have
\begin{align*}
    &d\bar H_s = \Lambda^T \bar H_s  \,ds + diag(\bar H_s)\,h \Big(Y_s,\E\,\left[\langle X_s,1_N\rangle\,|\, (Y_s,\alpha_s)=(y,a)\right]\,\P_{Y_s,\alpha_s}(dy,da),\alpha_s\Big)\,dW_s 
    \\& +diag( H'_s)\,
    h \Big(Y_s,\E\,\left[\langle X_s,1_N\rangle\,|\, (Y_s,\alpha_s)=(y,a)\right]\,\P_{Y_s,\alpha_s}(dy,da),\alpha_s\Big) \,dW_s \\ &-diag( H'_s)
    h \Big(Y_s,\E\,\left[\langle X'_s,1_N\rangle\,|\, (Y_s,\alpha_s)=(y,a)\right]\,\P_{Y_s,\alpha_s}(dy,da),\alpha_s\Big)
    \,dW_s.
\end{align*}
Proceeding as in Lemma \ref{lemmastap}, using standard estimates, the
Burkholder-Davis-Gundy  
and the H\"older inequalities, and the boundedness of $h$, we obtain
\begin{align}
\E\sup_{u\in[t,s]} |\bar H_u|^p &\le  
K\,  \int_t^s \E\,| \bar H_r|^p  \,dr  
\nonumber
\\&\quad 
+ K\,\E\bigg( \int_t^s |H'_r|^2\,\bigg|h \Big(Y_r,\E\,\left[\langle X_r,1_N\rangle\,|\, (Y_r, \alpha_r)=(y,a)\right]\,\P_{Y_r, \alpha_r}(dy\,da),\alpha_r\Big) \nonumber
\\&\qquad\qquad\qquad
-h \Big(Y_r,\E\,\left[\langle X'_r,1_N\rangle\,|\, (Y_r, \alpha_r)=(y,a)\right]\,\P_{Y_r, \alpha_r}(dy,da),\alpha_r\Big)\bigg|^2\,dr\bigg)^{p/2}.
\label{stimaHbar}
\end{align}
Let us denote for short
\begin{align*}
\nu_1(dy\,da)&= \E\,\left[\langle X_r,1_N\rangle\,|\, (Y_r, \alpha_r)=(y,a)\right]\,\P_{Y_r, \alpha_r}(dy\,da),
\\    
\nu_2(dy\,da)&= \E\,\left[\langle X_r',1_N\rangle\,|\, (Y_r, \alpha_r)=(y,a)\right]\,\P_{Y_r, \alpha_r}(dy\,da),
\end{align*}
and let us  note that $\nu_1$ and $\nu_2$ are  absolutely continuous with respect to  \(\P_{Y_r, \alpha_r} \). 
By (\nameref{HphLipschitz}),
\begin{align}
    &\bigg|h \Big(Y_r,\nu_1,\alpha_r\Big) - h \Big(Y_r,\nu_2,\alpha_r\Big)\bigg|^2 \nonumber\\
    &\leq L^2 \bigg( \int_{\R^d \times A} \bigg|\E\,\left[\langle X_r,1_N\rangle\,|\, (Y_r, \alpha_r)=(y,a)\right]- \E\,\left[\langle X'_r,1_N\rangle\,|\, (Y_r, \alpha_r)=(y,a)\right] \bigg|^2 \P_{Y_r, \alpha_r}(dy\,da) \bigg)\nonumber \\
    &\cdot \bigg( 1 + \chi(\| \P_{Y_r, \alpha_r} \|_{\calw_r}) + \chi(\| \nu_1 \|_{\calw_q}) + \chi(\|\nu_2 \|_{\calw_q})\bigg)^2 .
    \label{usodiA2}
\end{align}
As we proved in \eqref{MomentofqsPs}-\eqref{FirstSummandMomentofqsPs},  
\[ \|\nu_1\|_{\calw_q}^q=\| \E\,\left[\langle X_r,1_N\rangle\,|\, (Y_r, \alpha_r)\right]\,\P_{Y_r, \alpha_r} \|_{\calw_q}^q \leq C ||X||_{\cals^p} (1 + ||Y||_{\cals^r}^q)\le K, \]
where the last inequality follows from \eqref{XXprimoball}. Similarly, we have $\|\nu_2\|_{\calw_q}\le K$. Moreover,
\begin{align*}
    \| \P_{Y_r, \alpha_r} \|_{\calw_r}^r &=  \int_{\R^d \times A} (|y| + d_A(\alpha_r,a_0))^r\, \P_{Y_r, \alpha_r}(dy,da) 
    \\
    &
    \leq C \int_{\R^d \times A} (|y|^r + 1) \, \P_{Y_r, \alpha_r}(dy,da) \leq C(1 + ||Y||_{\cals^r}^r)\le K.
\end{align*}
From \eqref{usodiA2}, recalling that \(\chi\) is an increasing function, 
it follows that
\begin{align*}
    \bigg|h \Big(Y_r,\nu_1,\alpha_r\Big) - h \Big(Y_r,\nu_2,\alpha_r\Big)\bigg|^2  
    &\leq K   \int_{\R^d \times A} \bigg|\E\,\left[\langle \bar X_r,1_N\rangle\,|\, (Y_r, \alpha_r)=(y,a)\right] \bigg|^2 \P_{Y_r, \alpha_r}(dy\,da)  \\
    &=K\, \E\,\left[| \E\,\left[\langle \bar X_r,1_N\rangle \,|\, (Y_r, \alpha_r)\right]|^2 \right] \\
    &\le K\,\E\,\left[|\langle \bar X_r,1_N\rangle |^2 \right] \leq K\,\E\,\left[|\bar X_r|^2 \right] .
\end{align*}
 Substituting in \eqref{stimaHbar}, we deduce that
\begin{align*}
    \E\sup_{u\in[t,s]} |\bar H_u|^p &\le  
K\,  \int_t^s \E\,| \bar H_r|^p  \,dr + K\,\E\bigg( \int_t^s |H'_r|^2 \E\,\left[|\bar X_r|^2 \right]\,dr\bigg)^{\frac{p}{2}} \\
&\leq K\,  \int_t^s \E\,| \bar H_r|^p  \,dr + K\,\E \int_t^s |H'_r|^p \bigg(\E\,\left[|\bar X_r|^2 \right]  \bigg)^{\frac{p}{2}} dr\\
&\leq K\,  \int_t^s \E\,| \bar H_r|^p  \,dr +K\, \int_t^s \E[ |H'_r|^p] \E\,\left[|\bar X_r|^p \right] dr,
\end{align*}
where we used  the Jensen inequality, since \(p \geq 2\). By Lemma \ref{lemmastap} we have $\E[ |H'_r|^p]\le K$ and we conclude that
\begin{align*}
    &\E\sup_{u\in[t,s]} |\bar H_u|^p \le K\,  \int_t^s \E\,| \bar H_r|^p  \,dr + K
    \int_t^s \E\,\left[|\bar X_r|^p \right] dr.
\end{align*} 
By the Gronwall Lemma we deduce that
\[  \E\sup_{u\in[t,s]} |\bar H_u|^p \le K \int_t^s \E\,\left[|\bar X_r|^p \right] dr. \]
Define, for $n\ge1$,
\[
\Delta^0_s=\E\,\Big[\sup_{u\in[t,s]} |X_u-X'_u|^p\Big], \qquad
\Delta^n_s=\E\,\Big[\sup_{u\in[t,s]} |\Phi^n(X)_u-\Phi^n(X')_u|^p\Big]. 
\]
We have proved that 
\(
\Delta^1_s\le K\int_t^s\Delta^0_r\,dr.\) Iterating this inequality we obtain
\[
\Delta^n_s\le \frac{s^{n+1}K^n}{n!}\Delta^0_s.
\]
Choosing $n$ large enough to guarantee that 
$T^{n+1}K^n<{n!}$, we conclude that $\Phi^n$ is a contraction and therefore $\Phi$ has a unique fixed point. 
\end{proof}

We conclude the reasoning presenting the proof of Lemma \ref{lemmainvariance}.
\begin{proof}[Proof of Lemma \ref{lemmainvariance}]
We write equation \eqref{DefPhi}, which defines  $H$, in scalar form as follows:
\begin{align}
dH^i_s=\sum_{j\in S}  \lambda_{ji} H^j_s  \,ds +   H^i_s\,h^i_s
\,dW_s, \quad s\in[t,T],
\end{align}
where we set
\[ h^i_s := h(Y_s, \E [\langle  X_s,1 \rangle | (Y_s, \alpha_s)=(y,a) ] \P_{Y_s,\alpha_s}(dy,da), \alpha_s, i). \]
Next we use  the so-called robust resolution of the Zakai equation: we define
\[
E^i_s=\exp\left(-\int_t^s h^i_r\,dW_r\right), \quad
N^i_s=H^i_s\,E^i_s, \qquad s\in[t,T].
\]
By the It\^o rule one finds
$dE^i_s=-E^i_s h^i_s\,dW_s
+\frac{1}{2} \, E^i_s|h^i_s|^2\,ds
$ and
\begin{align*}
    dN^i_s&= \sum_{j\in S}  \lambda_{ji} H^j_s  E^i_s\,ds +   H^i_sE^i_s\,h^i_s
\,dW_s
-H^i_sE^i_s h^i_s\,dW_s
+\frac{1}{2} \, H^i_sE^i_s|h^i_s|^2\,ds -H^i_sE^i_s |h^i_s|^2\,ds
\\& =  \sum_{j\in S}  \lambda_{ji} H^j_s  E^i_s\,ds -\frac{1}{2} \, H^i_sE^i_s|h^i_s|^2\,ds
\\& =  \sum_{j\in S}  \lambda_{ji} N^j_s  E^i_s(E^j_s)^{-1}\,ds -\frac{1}{2} \, N^i_s |h^i_s|^2\,ds.
\end{align*}
We conclude that $N$ satisfies an ordinary differential equation with stochastic coefficients of the form $(d/ds)N_s=A(s)N_s$ where the matrix $A(s)$ has entries
\[
A_{ij}(s)=
\begin{cases}
    \lambda_{ji}   E^i_s(E^j_s)^{-1}, &i\neq j,
    \\
 \lambda_{ii}   -\frac{1}{2} \,   |h^i_s|^2   ,&i=j.
\end{cases}
\]
In order to prove that 
    $\P(H_s\in\R^N_+)=1$, it is sufficient to verify that 
    $\P(N_s\in\R^N_+)=1$, since $E^i_s>0$.
Since $N_t=H_t=\eta\in\R^N_+$, we only need to show that the set $\R^N_+$ is invariant for the flow of $N$. To this purpose we apply 
the  Nagumo invariance theorem. The tangent cone $T(x)\in\R^N$ to the set $\R^N_+$ at each point $x\in\R^N_+$ is easily computed as follows:
\[\text{setting}\; J(x)=\{j\in\{1,\ldots,N\}\,:\, x_j=0\},
\; \text{we have}\;
T(x)=\{v\in\R^N\,:\, v_i\ge 0\text{ for every } i\in J(x)\}.
\]
To
apply 
 Nagumo's theorem we  have to check that $A(s)x\in T(x)$ for every $x\in\R^N_+$. We have
 \[
(A(s)x)_i= 
 \sum_{j\in S} A_{ij}(s)x_j=
 \sum_{j\notin J(x)} A_{ij}(s)x_j.
 \]
If $i\in J(x)$ and $j\notin J(x)$ then clearly $i\neq j$ and since $A_{ij}(s)\ge 0$ for $i\neq j$ and $x_j\ge 0$ for every $j$ we see that
$(A(s)x)_i\ge 0$ whenever $i\in J(x)$, as required. It follows that the flow of $N$ leaves $\R^N_+$ invariant and the thesis holds.

A statement of Nagumo's theorem can be found in Theorem 1.2 of \cite{Aubin}. One can also consult Theorem 2.3 in Chapter 0 of \cite{ClarkeLedyaevSternWolenski}. We need the extension to nonautonomous systems (i.e. with time-dependent coefficients) which can be treated as indicated in Section 3 of Chapter 4 in the same book. Finally, the continuity requirement with respect to $t$ can easily be achieved by mollification of the coefficients $t\mapsto h_t^i$. 
\end{proof}

\begin{Remark}\label{remwellpose}
    \begin{enumerate}
\item[(i)]  
The requirement that $X_s\in\R^N_+$ is justified since we continue to identify $\R^N_+$ with the set of nonnegative measures on  $ S$. 
 \item[(ii)] 
 We note   an invariance property for the flow $s\mapsto 
\P_{X_s,Y_s}$ of the joint laws of $X_s$ and $Y_s$. 
For $\mu\in\calp(\R^N\times \R^d)$, denote by  $\mu_1\in\calp(\R^N)$ and $\mu_2\in\calp(\R^d)$ the  first and second marginal of $\mu$. Next define 
\begin{align}\label{defdpr}
&    
\cald_{p,r}=\bigg\{ \mu\in\calp(\R^N\times \R^d)\,:\, \mu_1\in\calp_p(\R^N),\, \mu_2\in\calp_r(\R^d)
\\&\qquad\qquad \mu_1(\R^N_+)=1,\, \langle \bar \mu_1,1_N\rangle= \int_{\R^N}\langle x,1_N\rangle\,\mu_1(dx)=1, \bigg\},
\end{align}
where  $\bar \mu_1\in\R^N$ denotes the mean of $\mu_1$. By our assumptions, $\P_{\eta,\xi}$ belongs to $\cald_{p,r}$ and the theorem implies that
$\P_{X_s,Y_s}\in\cald_{p,r}$. Thus, $\cald_{p,r}$ is left invariant by the flow $s\mapsto 
\P_{X_s,Y_s}$, for any control process $\alpha\in\cala$. \qed 
    \end{enumerate}
\end{Remark}

For the solution processes $Y$ and $X$ constructed in the previous result we use the notation $Y_s^{t,\xi}$ and $X_s^{t,\eta,\xi,\alpha}$, noting that $Y$ does not depend either on $\eta$ or on $\alpha$. 
We also have the following consequence, whose   proof follows  from the uniqueness of the solution and it  is postponed to the appendix. 

\begin{Corollary} \label{flowcor}\emph{(Flow property).} 
Under the assumptions of Theorem \ref{existstateeqprovv}, for $0\le t\le s\le r\le T$ and for   $\alpha\in\cala$, we have, $\P$-a.s.,
\begin{align}\label{flowprop}
\left(Y_r^{t,\xi}, X_r^{t,\eta,\xi,\alpha}
    \right)=
\left(Y_r^{s,Y_s^{t,\xi}  }, X_r^{s,X_s^{t,\eta,\xi,\alpha},Y_s^{t,\xi},\alpha}
    \right).
\end{align}
\end{Corollary}

\begin{Corollary} \label{continizi}
Under the assumptions of Theorem \ref{existstateeqprovv}, if $\eta$ and $\xi$ are $\calf_0$-measurable then we have, as $s-t\downarrow 0$, 
\begin{align*}
\sup_{\alpha\in\cala}\|X_s^{t,\eta,\xi,\alpha}-\eta\|_p\to 0, \qquad
\|Y_s^{t,\xi}-\xi\|_r\to 0.
\end{align*}
\end{Corollary}

\begin{proof}
 From the state equation \eqref{provveq:1}, using the Burkholder-Davis-Gundy inequalities and the boundedness of $h$ we obtain
\begin{align*} \|X_s^{t,\eta,\xi,\alpha}-\eta\|_p^p&\le  
C\,  \E\int_t^s |   X^{t,\eta,\xi,\alpha}_r|^p  \,dr    
+ C\,\E\,\bigg[\bigg( \int_t^s |X^{t,\eta,\xi,\alpha}_r|^2 \,dr\bigg)^{p/2}\bigg]
\le C\, \E\int_t^s \|   X^{t,\eta,\xi,\alpha}_r\|_p^p  \,dr   
\end{align*}
 and this tends to $0$ uniformly for $\alpha\in\cala$, by the estimate    \eqref{stap}.
Similarly, using \eqref{stapY}, one proves that $\|Y_s^{t,\xi}-\xi\|_r\to 0$.
\end{proof}

After proving the well-posedness of the state equation we consider the corresponding reward functional.

\begin{Proposition}\label{Jwelldef}
Suppose that Assumptions (\nameref{HpA1}), (\nameref{HphLipschitz}), (\nameref{HpU}) hold true and assume that $p\ge 2$ and $r\ge 2$ satisfy
\[
q\le r\left(1-\frac{1}{p}\right),
\qquad
\ell\le r\left(1-\frac{1}{p}\right),
\]
where $\ell$ is the constant in \eqref{growthfg}.
Let $(\eta,\xi)$ satisfy the assumptions of 
 Theorem \ref{existstateeqprovv}. 
Then the reward functional $J(t,\eta,\xi,\alpha)$ is well defined and finite.
\end{Proposition}

\begin{proof}
 Denote \(\mu_s= \E\,\left[\langle X_s,1_N\rangle\,|\, (Y_s, \alpha_s)\right]\,\P_{Y_s, \alpha_s}\) for simplicity.
Then
\begin{align*} 
\| \mu_s \|_{\calw_q}
\le C \|X\|_{\cals^p}^{1/q} (1 +  \|Y\|_{\cals^r})
 \le C\, \|\eta\|_p^{1/q}(1+\|\xi\|_r),
\end{align*}
where the first inequality follows from  \eqref{MomentofqsPs}-\eqref{FirstSummandMomentofqsPs} and the second one from \eqref{stap}-\eqref{stapY}.
 Denoting by $p'$ the exponent conjugate to $p$ and using the growth condition \eqref{growthfg} we have
\begin{align*}
\bigg|\E\bigg[\int_t^T \langle X_s ,f\Big(Y_s, \mu_s,\alpha_s\Big)\rangle\,ds \bigg]\bigg|
  &  \le
\left(\E \int_t^T | X_s|^p\,ds\right)^{1/p}\left( \E\int_t^T\Big|f\Big(Y_s, \mu_s,\alpha_s\Big)\Big|^{p'}\,ds \right)^{1/p'},
\\
 & \le C \,
 \left( \int_t^T \| X_s \|_p^p\,ds\right)^{1/p}\!\left( \int_t^T 
 (1+ \E\,[|Y_s|^{\ell p'}]+\chi(\|\mu_s\|_{\calw_q})^{p'})\,ds\right)^{1/p'}
 \\&
    \le C \left( \int_t^T \| X_s \|_p^p\,ds\right)^{1/p}
    \!\left( \int_t^T(1+\|Y_s\|_{\ell p'}^{\ell p'})\,ds \right)^{1/p'} .
\end{align*}
Due to the condition $\ell\le r(1-p^{-1})$ we have  $\ell p'\le r$, and using again 
\eqref{stap}-\eqref{stapY} and the H\''older inequality  we conclude that
\begin{align}\nonumber
\bigg|\E\bigg[\int_t^T \langle X_s ,f\Big(Y_s, \mu_s,\alpha_s\Big)\rangle\,ds \bigg|&  
\le C \left( \int_t^T \| X_s \|_p^p\,ds\right)^{1/p}
    \left( \int_t^T(1+\|Y_s\|_{r}^{\ell p'})\,ds \right)^{1/p'} 
\\&\nonumber
    \le C \,\|\eta\|_p (T-t)^{1/p}(1+\|\xi\|_r^\ell)(T-t)^{1/p'}
    \\&\label{estimJint}
    =C \,\|\eta\|_p  (1+\|\xi\|_r^\ell)(T-t)
    .
\end{align}
The term containing $g$ is treated in a similar way and we conclude that the reward functional $J(t,\eta,\xi,\alpha)$ is well defined and finite. 
\end{proof}

\vspace{1mm}

Under the assumptions of Proposition \ref{Jwelldef},
  the separated optimal control problem 
\eqref{provveq:1}-\eqref{provvfunz:1} 
is well posed  and we can define the lifted value function
\[
V(t,\eta,\xi)=\sup_{\alpha\in\cala} J(t,\eta,\xi,\alpha),
\]
for any $t\in[0,T]$, $\xi\in L^r(\Omega,\calf_t,\P;\R^d)$ and $\eta\in L^p(\Omega,\calf_t,\P;\R^N)$ satisfying  $\P(\eta\in\R^N_+)=1$, $
\E\,[\langle \eta,1_N\rangle]=1$. 
For the lifted value function
we have   the following dynamic programming principle, whose proof is standard and will be omitted.

\begin{Proposition}\label{propDPP}
Under the assumptions of Proposition \ref{Jwelldef}, for $0\le t\le s\le T$ we have
\begin{align}\nonumber
V(t,\eta,\xi)&=\sup_{\alpha\in\cala}
\bigg\{
\E\left[\int_t^s \big\langle X_r ,f\big(Y_r, \E\,\left[\langle X_r,1_N\rangle\,|\, (Y_r,\alpha_r)=(y,a)\right]\,\P_{Y_r,\alpha_r}(dy\,da),\alpha_r\big) \big\rangle\,dr\right] 
\\&\quad \quad\quad\;\;\;+  V(s, X_s,Y_s)  \bigg\},
\label{DPP}
\end{align}
where $Y_r=Y_r^{t,\xi}$ and $X_r=X_r^{t,\eta,\xi,\alpha}$ for $r\in [t,T]$. 
\end{Proposition}

We note that in the definition of the lifted value function $V$ we are taking the supremum over $\alpha\in\cala$, whereas in the separated problem of the previous section we had $\alpha\in\cala^W$. In the latter case the control is only adapted to the observation process $W$, whereas $\alpha\in\cala$ also depends on the auxiliary uniform random variable $U$. However,
we have the following result, whose proof is postponed to the Appendix:
\begin{Proposition}\label{AeAW} Suppose that  Assumptions  (\nameref{HpA1}), (\nameref{HphLipschitz}), (\nameref{HpU})  hold true and assume that $p\ge 2$ and $r\ge 2$ satisfy
\[
q\le r\left(1-\frac{1}{p}\right),
\qquad
\ell\le r\left(1-\frac{1}{p}\right).
\]
Suppose that $\eta=\pi_0$ and $\xi=y_0$ are deterministic.  Finally assume that
  \begin{align}
\label{rightcontatzero}
\sup_{  \alpha\in\cala}|J(t,\eta,\xi,\alpha)-J(0,\eta,\xi,\alpha)|\to 0, \qquad \text{ as }t\to 0.
  \end{align}
In particular, this holds when Assumptions  (\nameref{HpUniformContinuityh}), (\nameref{HpUniformContinuityfg}) below      are verified as well. 
Then we have
\[\sup_{\alpha\in\cala} J(0,\eta,\xi,\alpha)=\sup_{\alpha\in\cala^W} J(0,\eta,\xi,\alpha).
    \]
\end{Proposition}

\bigskip
 
Starting from the following paragraph all the Assumptions  (\nameref{HpA1})-(\nameref{HpUniformContinuityfg}) will be assumed to hold.  
Proposition \ref{AeAW} justifies the dynamic programming approach formulated above, since under the stated assumptions we   conclude that
\[
V(0,\pi_0,y_0)=V_0
\]
where the latter was defined in \eqref{valueseparatedoriginal} as the value of the separated problem \eqref{eq:1Y}-\eqref{funz:1} (which coincides with the value of the partially observed problem in the weak form 
\eqref{stateeq1bis}-\eqref{rewardfunz1bis}). Moreover, if the $\sup$ in the definition of $V(0,\pi_0,y_0)$ is attained by a control $\alpha\in\cala$ which also belongs to $\cala^W$, then $\alpha$ is optimal for the problem \eqref{eq:1}-\eqref{funz:1} and for the  partially observed control problem: below we will present a situation where this actually happens, see Remark \ref{verifoptpartialobs}.

\subsection{Stability results}

In this paragraph we establish  a  stability result for the state equation of the separated problem and for the corresponding reward functional, under variations of the initial conditions.
These results may be viewed as complements to Theorem 
\ref{existstateeqprovv} and to Proposition \ref{Jwelldef} and they are preliminary to further results on continuity properties of the value function that will be proved later. We need the following additional assumptions, corresponding to a uniform continuity assumption on the function $y\mapsto  h( y, \nu , a,i)$ and a local uniform continuity assumption on 
$\nu\mapsto  h( y, \nu , a,i)$. Recall the notation $\|\nu\|_{\calw_q}= \calw_q(\nu,\delta_{(0,a_0)})$, where $a_0$ is a given point in $A$, and
the increasing function $\chi$ introduced in Assumptions (\nameref{HpA1}).

\begin{Hp}[$\boldsymbol{A.4}$] \label{HpUniformContinuityh}
  There exist  functions $\omega_0,\omega_q:\R_+\to\R_+ $, continuous at $0$, increasing,   null at $0$,   with $\omega_0$  bounded, satisfying
\begin{align}
\label{unifomega0}
|h( y, \nu , a,i)-h( y', \nu , a,i)|& \le \omega_0(|y-y'|),
\\ 
\label{unifomegaq}
|h( y, \nu , a,i)-h( y, \nu' , a,i)|&\le \omega_q(\calw_q(\nu,\nu'))\,\Big(1+\chi(\|\nu\|_{\calw_q})+\chi(\|\nu'\|_{\calw_q})\Big),
\end{align}
for every $\nu,\nu'\in \calp_q(\R^d\times A)$, $y,y'\in\R^d$, $a\in A$,  $i\in S$.   
\end{Hp}

Note that the requirement  of boundedness of   $\omega_0$ does not involve any loss of generality. Indeed, since $|h|$ is bounded,  we can replace $\omega_0$ by $\omega_0 \wedge 2\sup|h|$.

\begin{Proposition} \label{stabtateeqprovv} Suppose that Assumptions (\nameref{HpA1})-(\nameref{HpUniformContinuityh}) hold true and that $p\ge 2$ and $r\ge 2$ satisfy
\[
q\le r\left(1-\frac{1}{p}\right).
\] Then, for any $L>0$, there exists a constant $K$, depending on $L,p,r$ and on the constants  occurring in the Assumptions, with the following property:
\newline
suppose that  $t\in[0,T]$, $(\eta,\xi)$ and $(\eta',\xi')$ are $\R^N\times \R^d$-valued and $\calf_t$-measurable random variables satisfying 
\[
\|\eta\|_p+\|\eta'\|_p+\|\xi\|_r+\|\xi'\|_r
\le L,\qquad
\P(\eta\in\R^N_+)=\P(\eta'\in\R^N_+)=1,\qquad
\E\,[\langle \eta,1_N\rangle]=\E\,[\langle \eta',1_N\rangle]=1;
\]
for any 
$\alpha\in\cala$, let $(X,Y)$, $(X',Y')$ denote the corresponding solutions to equation \eqref{provveq:1} given by Theorem 
\ref{existstateeqprovv}. Then
\begin{align}
\label{stabilstateeqY}
\E\, \Big[\sup_{s\in [t,T] }|Y_s-Y_s'|^r\Big]\le K\, \|\xi-\xi'\|_r^r,
\end{align}
\begin{equation}
\E\, \Big[\sup_{s\in [t,T] }| X_s-X_s'|^p\Big]\le K \,\bigg\{\|\eta-\eta'\|_p^p+ 
\omega_q\left(K\, \|\xi-\xi'\|_r
\right)^p +
\int_t^T \E\,\left[| X'_s|^p \,\omega_0(|Y_s-Y'_s|)^p\right]\,ds
\bigg\}.\label{stabilstateeq}
\end{equation}

In particular,   denoting by $(X,Y)$ and $(X^n,Y^n)$ the solutions corresponding to data $(\eta,\xi)$ and $(\eta^n,\xi^n)$ satisfying the assumptions above , if
\[
\|\xi-\xi^n\|_r+
\|\eta-\eta^n\|_p\to 0, \qquad n\to\infty,
\]
then
\[
\E\, \Big[\sup_{s\in [t,T] }|X_s-X_s^n|^p\Big]+
\E\, \Big[\sup_{s\in [t,T] }|Y_s-Y_s^n|^r\Big]\to 0.
\]
\end{Proposition}

\noindent {\bf 
Proof:} see the Appendix.

We need further assumptions on the rewards coefficients $f$ and $g$, similar to those required for $h$ in Assumptions (\nameref{HphLipschitz}) and (\nameref{HpUniformContinuityh}).

\bigskip

\begin{HpSingle}[$\boldsymbol{A.5}$] \label{HpUniformContinuityfg}
There exists a  constant $L\ge 0$ such that, whenever $\rho\in\calp_r(\R^d\times A)$,  and  $\nu_1,\nu_2\in\calp_q(\R^d\times A)$ are absolutely continuous with respect to $\rho$ with densities $\phi_1,\phi_2$ respectively,   the inequality
\begin{align}\label{liponfeg}
  &  \left| f( y, \nu_1 , a,i)- f( y, \nu_2 , a,i)\right|+\left| g( y, \nu_1 , i)- g( y, \nu_2 ,i)\right|
\\&\nonumber\qquad\le L \int_{\R^d\times A} |\phi_1(y',a)-\phi_2(y',a)| \,  \rho(dy'\,da) 
    \cdot\Big( 1+ \chi(\|\rho\|_{\calw_r}) +\chi(\|\nu_1\|_{\calw_q})+\chi(\|\nu_2\|_{\calw_q}) +|y|^\ell
\Big),
\end{align}
holds for every $y\in\R^d$, $a\in A$, $i\in S$.
\newline
There exist  functions $\omega_{0}^f,\omega_{0}^g,\omega_{q}^f,\omega_{q}^g:\R_+\to\R_+ $, continuous at $0$, increasing,   null at $0$,   with $\omega_0^f$ and $\omega_0^g$  bounded, satisfying
\begin{align}
\label{unifomega0f}
|f( y, \nu , a,i)-f( y', \nu , a,i)|& \le \omega_{0}^f(|y-y'|)\,
(1+|y|^\ell+|y'|^\ell+\chi(\|\nu\|_{\calw_q})),
\\ 
\label{unifomega0g}
     |g( y, \nu , i)-g( y', \nu , i)|&\le \omega_{0}^g(|y-y'|)\,
(1+|y|^\ell+|y'|^\ell+\chi(\|\nu\|_{\calw_q})),
\\ 
 \label{unifomegaqf}
|f( y, \nu , a,i)-f( y, \nu' , a,i)| &\le \omega_q^f(\calw_q(\nu,\nu'))\,\Big(1+\chi(\|\nu\|_{\calw_q})+\chi(\|\nu'\|_{\calw_q})+|y|^\ell\Big),
\\ \label{unifomegaqg}
|g( y, \nu , i)-g( y, \nu' , i)|&\le \omega_q^g(\calw_q(\nu,\nu'))\,\Big(1+\chi(\|\nu\|_{\calw_q})+\chi(\|\nu'\|_{\calw_q})+|y|^\ell\Big),
\end{align}
for every $\nu,\nu'\in \calp_q(\R^d\times A)$, $y,y'\in\R^d$, $a\in A$,  $i\in S$. 
\end{HpSingle}

\bigskip

\begin{Proposition} \label{stabrewardJ} 
Suppose that Assumptions (\nameref{HpA1})-(\nameref{HpUniformContinuityfg})   hold true and that $p\ge 2$ and $r\ge 2$ satisfy
\[
q\le r\left(1-\frac{1}{p}\right),\qquad 
\ell\le r\left(1-\frac{1}{p}\right).
\] Then, for any $L>0$, there exists a constant $K$, depending on $L,p,r$ and on the constants and functions occurring in the Assumptions, with the following property:
\newline
suppose that  $t\in[0,T]$, $(\eta,\xi)$ and $(\eta',\xi')$ are $\R^N\times \R^d$-valued and $\calf_t$-measurable random variables satisfying 
\[
\|\eta\|_p+\|\eta'\|_p+\|\xi\|_r+\|\xi'\|_r
\le L,\qquad
\P(\eta\in\R^N_+)=\P(\eta'\in\R^N_+)=1,\qquad
\E\,[\langle \eta,1_N\rangle]=\E\,[\langle \eta',1_N\rangle]=1;
\]
then,
for any 
$\alpha\in\cala$, we have
\begin{align}
\label{Junifcontuno}
|J(t,\eta,\xi,\alpha)- J(t,\eta',\xi,\alpha)|\le K\, \|\eta-\eta'\|_p\,;
\end{align}
if we additionally assume
\[
\|\eta'\|_{p_1}
\le L,
\]
for some $p_1>p$, then, denoting by $(X,Y)$, $(X',Y')$   the  solutions to equation \eqref{provveq:1} corresponding to $(\eta,\xi)$, $(\eta,\xi')$ respectively, we have, for any 
$\alpha\in\cala$,
\begin{align}
\nonumber
|J(t,\eta,\xi,\alpha)- J(t,\eta,\xi',\alpha)|&\le K\bigg\{   \omega_q\left(K\, \|\xi-\xi'\|_r\right) +\omega_q^f\left(K\, \|\xi-\xi'\|_r\right) +\omega_q^g\left(K\, \|\xi-\xi'\|_r\right) 
\\\nonumber
&  +
\Big(\E\,\Big[\omega_0\Big(\sup_{s\in [t,T] }|Y_s-Y'_s|\Big)^{\beta}\Big]\Big)^{1/\beta}
 +  \Big( \E  \,\Big[\omega_{0}^f\Big(\sup_{s\in[t,T]}|Y_s-Y_s'|\Big)^{\gamma} \Big]\Big)^{1/\gamma} 
 \\&\label{Junifcontdue}
+  \Big( \E  \,\Big[\omega_{0}^g\Big(\sup_{s\in[t,T]}|Y_s-Y_s'|\Big)^{\gamma} \Big]\Big)^{1/\gamma} 
\bigg\},
\end{align}
where $\beta=\frac{pp_1}{p_1-p}>0 $, $\gamma= \frac{ r}{r(1-1/p_1)-\ell }>0$.

\end{Proposition}

\noindent {\bf 
Proof:} see the Appendix.

\subsection{Law invariance}

This terminology is  used to indicate that the value function  depends on the initial conditions only through their laws, as shown in  Theorem \ref{lawinv} below. We follow the approach introduced in \cite{DeCrescenzoFuhrmanKharroubiPham}.

Let us fix a starting time $t\in [0,T]$ for the controlled state equation 
\eqref{provveq:1} and an initial condition  $(\eta,\xi)$, under the assumptions of Theorem \ref{existstateeqprovv}.
Since $\eta$ and $\xi$ are measurable with respect to $\calf_t=\calf_t^W\vee \sigma(U)$, they can be written in the form 
\[
\eta=\underline\eta(U,W_{[0,t]}),
 \quad 
\xi=\underline\xi(U,W_{[0,t]}),\qquad \P-a.s.
\]
where
\[
\underline\eta:(0,1)\times C([0,t];\R^d)\to \R^N,
\quad
\underline\xi:(0,1)\times C([0,t];\R^d)\to \R^d,
\]
are Borel measurable functions and $W_{[0,t]}=(W_s)_{s\in[0,T]} $ is the Brownian path on $[0,t]$, that we consider as a random element in the space $C([0,t];\R^d)$ of continuous functions $[0,t]$ to $\R^d$. Similarly, since the control process $\alpha$ is predictable with respect to  $\F=(\calf_t)_{t\ge0}$, we have, up to  indistinghuishability,
\[
\alpha_s=\underline{\alpha}_t(s,U,W_{[0,t]},W_{s\wedge \cdot}-W_t), \qquad s\ge t,
\]
where  $\underline{\alpha}_t$ is a Borel measurable function
\[
\underline{\alpha}_t:[t,\infty)\times (0,1)\times C([0,t];\R^d)\times C([t,\infty);\R^d)\to A.
\] 
The above representations are proved by monotone class arguments, see for instance Proposition 10 in \cite{Claisse} or  Lemma 2.3 in \cite{Rudà}.
In the results that follow we denote by $m$ the Lebesgue measure on $(0,1)$ and by $\calw_t$ the Wiener measure on  $C([0,t];\R^d)$.

\begin{Lemma}\label{shiftmeas}
Suppose that a Borel measurable function
\[
\tau: (0,1)\times C([0,t];\R^d)\to (0,1)\times C([0,t];\R^d)
    \]
    satisfies
\(
\tau_\sharp (m\otimes \calw_t)=m\otimes \calw_t
\),    
namely that it
preserves the product measure $m\otimes \calw_t$.
Then setting
\begin{align*}
\eta'=\underline\eta(\tau(U,W_{[0,t]})), \quad
\xi'=\underline\xi(\tau(U,W_{[0,t]})), \quad
\alpha'_s= \underline{\alpha}_t(s,\tau(U,W_{[0,t]}),W_{s\wedge \cdot}-W_t), \qquad s\ge t,
\end{align*}
we have
\begin{align*}
J(t,\eta,\xi,\alpha)=J(t,\eta',\xi',\alpha').
\end{align*}
\end{Lemma}

\begin{proof}
 We tacitly define $\alpha'_s$ to be equal to any given $a_0\in A$ for $s\in [0,t)$, so that $\alpha'$ is an admissible control as well. Below we denote by $\call$ the law under $\P$ of various arrays of random elements. By our assumptions we have 
\begin{align*}
    \call(U,W_{[0,t]})=
\call(\tau(U,W_{[0,t]})),    \qquad
(U,W_{[0,t]}) \text{ independent of } (W_{s}-W_t)_{s\ge t}.
\end{align*}
It follows that
\begin{align*}
    \call((U,W_{[0,t]}), (W_{s}-W_t)_{s\ge t})=
\call(\tau(U,W_{[0,t]}), (W_{s}-W_t)_{s\ge t}),
\end{align*}
which implies
\begin{align*}
\call(\eta,\xi,\alpha)=
\call(\eta',\xi',\alpha').
\end{align*}
In view of the pathwise uniqueness property in Theorem 
\ref{existstateeqprovv}, denoting  $(X,Y)$, $(X',Y')$   the  solutions to equation \eqref{provveq:1} corresponding to initial data and controls $(\eta,\xi,\alpha)$ and $(\eta',\xi',\alpha')$ respectively, we deduce that
\begin{align*}
\call(X,Y,\alpha)=
\call(X',Y',\alpha'),
\end{align*}
which gives the required conclusion.
\end{proof}

\begin{Theorem} \label{lawinv} Suppose that Assumptions (\nameref{HpA1})-(\nameref{HpUniformContinuityfg}) hold true and that $p\ge 2$ and $r\ge 2$ satisfy
\[
q\le r\left(1-\frac{1}{p}\right),\qquad 
\ell\le r\left(1-\frac{1}{p}\right).
\]
Suppose that, for $t\in[0,T]$,   $(\eta,\xi)$ and $(\widehat\eta,\widehat\xi)$ are $\R^N\times \R^d$-valued and $\calf_t$-measurable random variables satisfying 
\[
\|\eta\|_p+\|\widehat\eta\|_p+\|\xi\|_r+\|\widehat\xi\|_r
<\infty,\qquad
\P(\eta\in\R^N_+)=\P(\widehat\eta\in\R^N_+)=1,\qquad
\E\,[\langle \eta,1_N\rangle]=\E\,[\langle \widehat\eta,1_N\rangle]=1.
\]
If, in addition, we have
$\P_{\eta,\xi}= \P_{\widehat\eta,\widehat\xi}$, 
then we also have $V(t,\xi,\eta)=V(t,\widehat\xi,\widehat\eta)$.
\end{Theorem}

\begin{proof}
We will show that, for any $\alpha\in\cala$ and $\epsilon>0$, one can find 
$\alpha^\epsilon\in\cala$ such that 
\begin{align}\label{Jconvuno}
J(t,\widehat\eta,\widehat\xi,\alpha^\epsilon)\to J(t,\eta,\xi,\alpha), \qquad 
\epsilon\to 0.
\end{align}
This implies that  $V(t,\eta,\xi)\le V(t,\widehat\eta,\widehat\xi)$.
Interchanging the roles of $(\xi,\eta)$ and $(\widehat\xi,\widehat\eta)$ one proves the opposite inequality and then the conclusion follows.

Then let $\alpha\in\cala$ and $\epsilon>0$ be given. As noted earlier, since  $(\eta,\xi)$ and  $(\widehat\eta,\widehat\xi)$  are $\calf_t$-measurable, 
they can be written in the form 
\[
\eta=\underline\eta(U,W_{[0,t]}),
 \quad 
\widehat\eta=\underline{\widehat\eta}(U,W_{[0,t]}),
 \quad 
\xi=\underline\xi(U,W_{[0,t]}), \quad 
\widehat\xi=\underline{\widehat\xi}(U,W_{[0,t]}),\qquad \P-a.s.
\]
where
\[
(\underline\eta,\underline\xi),(\underline{\widehat\eta},\underline{\widehat\xi}):(0,1)\times C([0,t];\R^d)\to \R^N\times \R^d,
\]
are Borel measurable functions. Similarly,  since $\alpha$ is $\F$-predictable, there exists  a Borel measurable function
\[
\underline{\alpha}_t:[t,\infty)\times (0,1)\times C([0,t];\R^d)\times C([t,\infty);\R^d)\to A
\] 
such that
\[
\alpha_s=\underline{\alpha}_t(s,U,W_{[0,t]},W_{s\wedge \cdot}-W_t), \qquad s\ge t.
\]
Since the law of $(U,W_{[0,t]})$ is $m\otimes \calw_t$ and since $\P_{\eta,\xi}= \P_{\widehat\eta,\widehat\xi}$, it follows that 
$(\underline\eta,\underline\xi)$ and $(\underline{\widehat\eta},\underline{\widehat\xi})$, considered as random variables on $(0,1)\times C([0,t];\R^d)$, have the same law under the probability $m\otimes \calw_t$. Since $(0,1)\times C([0,t];\R^d)$ is a Polish space and the measure $m\otimes \calw_t$ is nonatomic, there exists 
a Borel measurable function
\[
\tau^\eps: (0,1)\times C([0,t];\R^d)\to (0,1)\times C([0,t];\R^d)
    \]
    satisfying
\(
\tau^\epsilon_\sharp (m\otimes \calw_t)=m\otimes \calw_t
\) and
\begin{align}\label{tauepsuno}
|\underline{\eta}(\tau^\epsilon(u,w))-\underline{\widehat\eta}(u,w) |\le \epsilon, \quad
|\underline{\xi}(\tau^\epsilon(u,w))-\underline{\widehat\xi}(u,w) |\le \epsilon, \qquad u\in(0,1),\,w\in C([0,t];\R^d);
\end{align}
see for instance \cite{CarmonaDelarue1}, Lemma 5.23.
Let us define
\begin{align*}
\eta^\epsilon=\underline\eta(\tau^\epsilon(U,W_{[0,t]})), \quad
\xi^\epsilon=\underline\xi(\tau^\epsilon(U,W_{[0,t]})), \quad
\alpha^\epsilon_s= \underline{\alpha}_t(s,\tau^\epsilon(U,W_{[0,t]}),W_{s\wedge \cdot}-W_t), \qquad s\ge t.
\end{align*}
This way \eqref{tauepsuno} implies
\begin{align}\label{tauepsdue}
\|{\eta}^\epsilon-\widehat\eta \|_\infty\le \epsilon, \quad
\|{\xi}^\epsilon-\widehat\xi  \|_\infty\le \epsilon.
\end{align}
We will show that the process
$\alpha^\epsilon$ satisfies \eqref{Jconvuno}. To this end we first note that, by  Lemma \ref{shiftmeas}, we have
$J(t,\eta,\xi,\alpha)=J(t,\eta^\epsilon,\xi^\epsilon,\alpha^\epsilon)$. We will prove that \eqref{Jconvuno} holds true by showing that
\begin{align}\label{Jconvdue}
J(t,\widehat\eta,\widehat\xi,\alpha^\epsilon)- J(t,\eta^\epsilon,\xi^\epsilon,\alpha^\epsilon)\to 0, \qquad 
\epsilon\to 0.
\end{align}

In a first step we assume the additional condition that $\|\eta\|_{p_1}=\|\widehat\eta\|_{p_1}<\infty$ for some $p_1>p$. We have
\begin{align*}
|J(t,\widehat\eta,\widehat\xi,\alpha^\epsilon)- J(t,\eta^\epsilon,\xi^\epsilon,\alpha^\epsilon)|\le 
|J(t,\widehat\eta,\widehat\xi,\alpha^\epsilon)- J(t,\widehat\eta,\xi^\epsilon,\alpha^\epsilon)|+
|J(t,\widehat\eta,\xi^\epsilon,\alpha^\epsilon)- J(t,\eta^\epsilon,\xi^\epsilon,\alpha^\epsilon)|.
\end{align*}
By \eqref{Junifcontuno},
\(|J(t,\widehat\eta,\xi^\epsilon,\alpha^\epsilon)- J(t,\eta^\epsilon,\xi^\epsilon,\alpha^\epsilon)|\le K\, \|\widehat\eta-\eta^\epsilon\|_p\le K\cdot\epsilon\to0
\). The first term is estimated using \eqref{Junifcontdue}:
denoting by $(\widehat X,\widehat Y)$ and  $(\widehat X^\epsilon,\widehat Y^\epsilon)$  the  solutions  corresponding  
to $(\widehat \eta,\widehat \xi)$ and  $(\widehat \eta,\xi^\epsilon)$
respectively, we have
\begin{align}\nonumber
|J(t,\widehat\eta,\widehat\xi,\alpha^\epsilon)- J(t,\widehat\eta,\xi^\epsilon,\alpha^\epsilon)|&\le K\bigg\{ \omega_q\left(K\, \|\widehat\xi-\xi^\epsilon\|_r\right)+\omega_q^f\left(K\, \|\widehat\xi-\xi^\epsilon\|_r\right)+\omega_q^g\left(K\, \|\widehat\xi-\xi^\epsilon\|_r\right)
\\\nonumber
&  +
\Big(\E\,\Big[\omega_0\Big(\sup_{s\in [t,T] }|\widehat Y_s-\widehat Y^\epsilon_s|\Big)^{\beta}\Big]\Big)^{1/\beta}
 +  \Big( \E  \,\Big[\omega_{0}^f\Big(\sup_{s\in[t,T]}|\widehat Y_s-\widehat Y^\epsilon_s|\Big)^{\gamma} \Big]\Big)^{1/\gamma} 
\\
&+  \Big( \E  \,\Big[\omega_{0}^g\Big(\sup_{s\in[t,T]}|\widehat Y_s-\widehat Y^\epsilon_s|\Big)^{\gamma} \Big]\Big)^{1/\gamma} 
\bigg\},
\label{omega0hat}
\end{align}
where $\beta>0 $ and $\gamma>0$ are suitable constants. The first three summands on the right-hand side tend to $0$ as $\epsilon\to 0$, by \eqref{tauepsdue} and the continuity of \(\omega_q\), \(\omega_q^f\), \(\omega_q^g\). We show that the same occurs for the other ones. Indeed, first note that by \eqref{stabilstateeqY} we have
\begin{align*}
\E\, \Big[\sup_{s\in [t,T] }|\widehat Y_s-\widehat Y^\epsilon_s|^r\Big]&\le K\, \|\widehat\xi- \xi^\epsilon\|_r^r\le K\cdot\epsilon^r\to 0.
\end{align*}
If we had,  for some constant $\delta>0$ and some sequence $\epsilon_n\to0$,
\begin{align*}
\E  \,\Big[\omega_{0}\Big(\sup_{s\in[t,T]}|\widehat Y_s-\widehat Y^{\epsilon_n}_s|\Big)^{\beta} \Big]\ge\delta,
\end{align*}
then we might extract a subsequence $\{n_{k}\}$ such that 
$\widehat Y_s^{\eps_{n_{k}}}\to \widehat Y_s$ a.s., uniformly in $s$,  for $k\to\infty$. Since $\omega_0$ is continuous at zero and bounded we would get a contradiction with the dominated convergence theorem.
The last two summands in the right-hand side of \eqref{omega0hat} also tend to $0$ by similar arguments.

So far we have proved the required conclusion under the additional assumption that $\|\eta\|_{p_1}=\|\widehat\eta\|_{p_1}<\infty$ for some $p_1>p$. In the general case we adopt an approximation procedure: define the    truncation function $T_m(t)=(t\wedge m)\vee (-m)$, $t\in\R$,
and let $T_m(\eta)$ and $T_m(\widehat\eta)$ be defined componentwise. Then  
$\P_{T_m(\eta),\xi}= \P_{T_m(\widehat\eta),\widehat\xi}$ and $T_m(\eta)$, $T_m(\widehat\eta)$ are bounded. By what we already proved
we  have $V(t,T_m(\eta),\xi)=V(t,T_m(\widehat\eta),\widehat\xi)$. By the inequality \eqref{Junifcontuno}, for $m\to\infty$,
\begin{align}
|V(t, \eta,\xi)- V(t,T_m(\eta),\xi)|
\le \sup_{\alpha\in\cala}
|J(t,\eta,\xi,\alpha)- J(t,T_m(\eta),\xi,\alpha)|\le K\, \|\eta-T_m(\eta)\|_p\to 0.
\end{align}
Similarly we have 
$V(t,T_m(\widehat\eta),\widehat\xi)\to V(t,\widehat\eta,\widehat\xi)$ and we conclude that 
$V(t, \eta,\xi) =V(t, \widehat\eta,\widehat\xi) $.
\end{proof}

\subsection{The value function and its properties} \label{SectionPropv}

In this section we assume that Assumptions (\nameref{HpA1})-(\nameref{HpUniformContinuityfg}) hold true and that
$p\ge 2$ and $r\ge 2$ satisfy
\[
q\le r\left(1-\frac{1}{p}\right),\qquad 
\ell\le r\left(1-\frac{1}{p}\right).
\]
Recall the definition of the set $\cald_{p,r}$ in 
\eqref{defdpr}. Given $\mu\in\cald_{p,r}$ and $t\in[0,T]$, let us take an $\calf_t$-measurable random variable
$(\eta,\xi)$ with values in 
$\R^N\times \R^d$ satisfying 
$\mu=\P_{\eta,\xi}$ and let us set
\begin{align}\label{defvaluefunction}
v(t,\mu)= V(t,\eta,\xi).    
\end{align}
By the previous result, $v(t,\mu)$ does not depend on the choice of $(\eta,\xi)$ and so this formula defines a function $v:[0,T]\times \cald_{p,r}\to\R$, which is the value function of our control problem.

The existence of such a pair $(\eta,\xi)$ can be verified as follows. Given an arbitrary probability $\mu$ on $\R^N\times \R^d$, by a classical result one can find a Borel measurable map $\Phi:(0,1)\to\R^N\times\R^d$ such that $\Phi_\sharp m=\mu$, i.e. carrying the Lebesgue measure $m$ to $\mu$. Denote $(\underline{\eta}, \underline{\xi})$ the components of $\Phi$. Then $(\eta,\xi)=(\underline{\eta}(U), \underline{\xi}(U))$ is the required pair. This explicit construction also makes it clear that $(\eta,\xi)$ can be chosen to be a function of $U$ alone, hence $\calf_0$-measurable. 

Below we will deal with a variant of this construction. Suppose we have  $\mu$ and a sequence $\mu_n$, probabilities on $\R^N\times\R^d$. Let $\gamma_n\in\calp((\R^N\times\R^d)\times(\R^N\times\R^d))$ have marginals $\mu$ and $\mu_n$. Then one can find Borel measurable  functions 
\[
(\underline{\eta}, \underline{\xi}):(0,1)\to \R^N\times \R^d,\qquad \text{and}\qquad (\underline{\eta}_n, \underline{\xi}_n):(0,1)\to \R^N\times \R^d
\]
such that 
$(\underline{\eta}, \underline{\xi},\underline{\eta}_n, \underline{\xi}_n)_\sharp m=\gamma_n$. In other words, since the first marginal of $\gamma_n$ is the same for all $n$ one can choose the functions $\underline{\eta}$ and $\underline{\xi}$ to be the same for all $n$. This is not entirely trivial and it is proved below in Proposition \ref{varskorohod} of the Appendix.
It follows that the random variable
\[
(\eta,\xi,\eta_n,\xi_n)=(\underline{\eta}(U), \underline{\xi}(U),\underline{\eta}_n(U), \underline{\xi}_n(U))
\]
has law $\gamma_n$ and it is $\calf_0$-measurable.

The first property of the value function is the dynamic programming principle, which is an immediate consequence of Proposition \ref{propDPP}:

\begin{Proposition} \label{propDPPbis}
For $0\le t\le s\le T$ and $\mu\in\cald_{p,r}$ we have
\begin{align}
v(t,\mu)\!=\!\sup_{\alpha\in\cala}
\bigg\{\!
\E\left[\int_t^s \langle X_r ,f\Big(Y_r, \E\,\left[\langle X_r,1_N\rangle\,|\, (Y_r,\alpha_r)=(y,a)\right]\,\P_{Y_r,\alpha_r}(dy\,da),\alpha_r\Big)\rangle\,dr\right]\! +\! v(s, \P_{X_s,Y_s})\!\bigg\}
\label{DPPv}
\end{align}
where $(\eta,\xi)$ is any pair of $\calf_t$-measurable random variables in $\R^N\times \R^d$ such that 
$\mu=\P_{\eta,\xi}$, and
$Y_r=Y_r^{t,\xi}$ and $X_r=X_r^{t,\eta,\xi,\alpha}$ for $r\in [t,T]$. 

\end{Proposition}

Next we proceed to prove that the value function $v$ is continuous. The proof is based on the following technical result, whose proof is postponed to the Appendix.

\begin{Lemma} \label{Jcont}  Suppose that, for $n\ge 1$,      $(\eta,\xi)$ and $(\eta_n,\xi_n)$ are  $\R^N\times \R^d$-valued  random variables which   satisfy
\[
\|\eta\|_p+\|\eta_n\|_p+\|\xi\|_r +\|\xi_n\|_r
<\infty,\qquad
\P(\eta\in\R^N_+)=\P(\eta_n\in\R^N_+)=1,\qquad
\E\,[\langle \eta,1_N\rangle]=\E\,[\langle  \eta_n,1_N\rangle]=1,
\]
and we have
$\|\eta-\eta_n\|_p\to0$, $\|\xi-\xi_n\|_r\to0$ for $n\to\infty$. 
Then  
\begin{enumerate}
    \item [(i)] If $\eta,\xi,\eta_n,\xi_n$ are  $\calf_{0}$-measurable  then
\[
\sup_{\alpha\in\cala,\,t\in[0,T]}|J(t,\eta_n,\xi_n,\alpha)- J(t,\eta,\xi,\alpha)|\to 0.
\]
    \item [(ii)]   
Suppose that, instead of a single sequence $\{\eta_n\}$, we have a family of random sequences $\{\eta_n^\beta\}$, depending on an arbitrary index $\beta$, each  satisfying the above conditions, and in addition
\begin{align}\label{etaennebeta}
    \sup_\beta \|\eta_n^\beta-\eta\|_p\to 0,\qquad n\to\infty.
\end{align}
If $\{t_n\}$ is an arbitrary sequence in $[0,T]$,  $(\eta,\xi)$ are  $\calf_{0}$-measurable,  $(\eta_n^\beta,\xi_n)$ are   $\calf_{t_n}$-measurable, then
\begin{align}\label{unifUIbis}
\sup_{\alpha\in\cala} \sup_{\beta} |  
J(t_n,\eta_n^\beta,\xi_n,\alpha)- J(t_n,\eta,\xi,\alpha)|\to 0,\qquad n\to\infty.
\end{align}
\end{enumerate}
\end{Lemma}

Recall that the value function   $v:[0,T]\times \cald_{p,r}\to\R$ depends on a measure argument that ranges in the set $\cald_{p,r}$ introduced in 
\eqref{defdpr}. In order to deal with continuity properties we need an appropriate notion of convergence for sequences of measures in this set, taking into account the different summability properties of the two marginals. To this end, consider the optimal transport problem, for $\mu,\nu\in\cald_{p,r}$:
\begin{align}
    \label{opttransp}
    \tilde \calw(\mu,\nu)=\inf_{\gamma\in \Pi(\mu,\nu)}\left\{
\int_{\R^N\times\R^d\times \R^N\times\R^d}
\Big(|x-x'|^p+|y-y'|^r\Big)\,\gamma(dx\,dy\,dx'\,dy')
    \right\},
\end{align}
where $\Pi(\mu,\nu)$ denotes the set of couplings of $\mu$ and $\nu$.
By standard results in optimal transport theory,  a minimizer in \eqref{opttransp} exists. The function $\tilde \calw$ is not a metric, but it can be used to define a notion of convergence in $\cald_{p,r}$: we say that 
$\mu_n\to\mu$ in $\cald_{p,r}$ if $\tilde \calw(\mu_n,\mu)\to 0$ as $n\to\infty$. 
Of course, if  $p=r$, then $\tilde \calw^{1/p}$ coincides with the usual Wasserstein metric $\calw_p$  and we recover on $\cald_{p,r}$ the usual notion of convergence inherited from on  $\calp_p(\R^N\times \R^d)$. In the following section we will introduce a topology that induces this notion of convergence in the general case. For the moment, we are  ready to prove a sequential continuity property for the value function $v$.

\begin{Theorem} \label{Continuityofvinciandtime} Let Assumptions (\nameref{HpA1})-(\nameref{HpUniformContinuityfg}) hold true and let
$p\ge 2$ and $r\ge 2$ satisfy
\[
q\le r\left(1-\frac{1}{p}\right),\qquad 
\ell\le r\left(1-\frac{1}{p}\right).
\]
Suppose that $\mu_n,\mu\in\cald_{p,r}$, $\tilde \calw(\mu_n,\mu)\to 0$ as $n\to\infty$. Then
\begin{enumerate}
\item[(i)] $v(t,\mu_n)\to v(t,\mu)$ uniformly for $t\in[0,T]$;
\item[(ii)]
If in addition $t_n\to t$, then  $v(t_n,\mu_n)\to v(t,\mu)$.  \end{enumerate}
\end{Theorem}

\begin{proof}
 Denote by $\gamma_n$ an optimal coupling in the definition of $\tilde \calw(\mu_n,\mu)$. As explained above, one can find $\calf_0$-measurable random variables $(\eta,\xi)$ and $(\eta_n,\xi_n)$, with values in $\R^N\times\R^d$, such that $(\eta,\xi,\eta_n,\xi_n)$ has law $\gamma_n$. Then
\begin{align*}
    \|\eta_n-\eta\|_p^p+\|\xi_n-\xi\|_r^r=
\E\,\Big[ |\eta_n-\eta|^p+|\xi_n-\xi|^r\Big]    = \tilde \calw(\mu_n,\mu)
    \to0.
\end{align*}
By definition of the value functions $v$ and $V$ and by point $(i)$ of Lemma \ref{Jcont} we have
\begin{align*}
\sup_{t\in[0,T]}|v(t,\mu_n)-v(t,\mu)|&= \sup_{t\in[0,T]}|V(t,\eta_n,\xi_n)-V(t,\eta,\xi)|
\\&\le \sup_{t\in[0,T],\,\alpha\in\cala}|J(t,\eta_n,\xi_n,\alpha)-J(t,\eta,\xi,\alpha)|\to 0,
\end{align*}
which proves $(i)$.

In view of $(i)$, to prove point $(ii)$ it is enough to show that for every $\mu\in\cald_{p,r}$ we have $v(t_n,\mu)\to v(t,\mu)$. Take $(\eta,\xi)$, $\calf_0$-measurable with $\P_{\eta,\xi}=\mu$. We distinguish two cases.

First suppose $t_n\ge t$. First we have
\begin{align}\label{destracontJ}
 |v(t_n,\mu)-v(t,\mu)|&=  |V(t_n,\eta,\xi)-V(t,\eta,\xi)|
\le \sup_{\alpha\in\cala}|J(t_n,\eta,\xi,\alpha)-J(t,\eta,\xi,\alpha)|.
\end{align}
We prove that the right-hand side of \eqref{destracontJ} tends to zero. 
By the flow property \eqref{flowprop},
\begin{align*}
&J(t,\eta,\xi,\alpha)-J(t_n,\eta,\xi,\alpha)=
J(t_n,X_{t_n}^{t,\eta,\xi,\alpha},Y_{t_n}^{t,\xi},\alpha)
-J(t_n,\eta,\xi,\alpha)
\\&\quad +
\E\bigg[\int_t^{t_n} \langle X_{s}^{t,\eta,\xi,\alpha} ,f\Big(Y_{s}^{t,\xi}, \E\,\left[\langle X_{s}^{t,\eta,\xi,\alpha},1_N\rangle\,|\, (Y_{s}^{t,\xi},\alpha_s)=(y,a)\right]\,\P_{Y_{s}^{t,\xi},\alpha_s}(dy\,da),\alpha_s\Big)\rangle\,ds \bigg].
\end{align*} 
From Corollary \ref{continizi}  we know that 
\begin{align*}
\sup_{\beta\in\cala}\|X_{t_n}^{t,\eta,\xi,\beta}-\eta\|_p\to0,\quad \|Y_{t_n}^{t,\xi}-\xi\|_r\to0, \qquad
t_n\to t.    
\end{align*}
By \eqref{unifUIbis}, replacing $(\eta_n^\beta,\xi_n)$ by $(X_{t_n}^{t,\eta,\xi,\beta},Y_{t_n}^{t,\xi})$, it follows that
\[
\sup_{\alpha\in\cala}
|J(t_n,X_{t_n}^{t,\eta,\xi,\alpha},Y_{t_n}^{t,\xi},\alpha)
-J(t_n,\eta,\xi,\alpha)|
\le
\sup_{\alpha\in\cala,\,\beta\in\cala}
|J(t_n,X_{t_n}^{t,\eta,\xi,\beta},Y_{t_n}^{t,\xi},\alpha)
-J(t_n,\eta,\xi,\alpha)|
\to 0.
\]
By an estimate entirely similar to \eqref{estimJint} we also obtain
\begin{align*}
&
\sup_{\alpha\in\cala}
\bigg|\E\bigg[\int_t^{t_n} \langle X_{s}^{t,\eta,\xi,\alpha} ,f\Big(Y_{s}^{t,\xi}, \E\,\left[\langle X_{s}^{t,\eta,\xi,\alpha},1_N\rangle\,|\, (Y_{s}^{t,\xi},\alpha_s)=(y,a)\right]\,\P_{Y_{s}^{t,\xi},\alpha_s}(dy\,da),\alpha_s\Big)\rangle\,ds\bigg]\bigg|
\\&\qquad \le K\,(t_n-t)\to 0,
\end{align*}
where the constant $K$
 may depend on $\|\eta\|_p$ and $\|\xi\|_r$.

 Now consider the second case $t_n\le t$. From the dynamic programming principle \eqref{DPPv} we obtain
 \begin{align*}
&v(t_n,\mu)=\sup_{\beta\in\cala} \bigg\{v(t, \P_{X_t^{t_n,\eta,\xi,\beta},Y_t^{t_n,\xi}})
\\
&\quad+
\E\left[\int_{t_n}^t \langle X_s^{t_n,\eta,\xi,\beta} ,f\Big(Y_s^{t_n,\xi}, \E\,\left[\langle X_s^{t_n,\eta,\xi,\beta},1_N\rangle\,|\, (Y_s^{t_n,\xi},\alpha_s)=(y,a)\right]\,\P_{Y_s^{t_n,\xi},\alpha_s}(dy\,da),\alpha_s\Big)\rangle\,ds\right] \bigg\},
\end{align*}
 and so we have
\begin{align*}
&|v(t_n,\mu)-v(t,\mu)|\le \sup_{\beta\in\cala} \Big| v(t, \P_{X_t^{t_n,\eta,\xi,\beta},Y_t^{t_n,\xi}})-v(t,\mu)\Big|
\\
&+\sup_{\beta\in\cala} \bigg|
\E\left[\int_{t_n}^t \langle X_s^{t_n,\eta,\xi,\beta} ,f\Big(Y_s^{t_n,\xi}, \E\,\left[\langle X_s^{t_n,\eta,\xi,\beta},1_N\rangle\,|\, (Y_s^{t_n,\xi},\alpha_s)=(y,a)\right]\,\P_{Y_s^{t_n,\xi},\alpha_s}(dy\,da),\alpha_s\Big)\rangle\,ds\right] \bigg|.
\end{align*}
As before, and similar to \eqref{estimJint}, the last term is $\le K\,(t_n-t)\to 0$. Finally we have
\begin{align*} \sup_{\beta\in\cala} \Big| v(t, \P_{X_t^{t_n,\eta,\xi,\beta},Y_t^{t_n,\xi}})-v(t,\mu)\Big|
&=
\sup_{\beta\in\cala} \Big| V(t,  X_t^{t_n,\eta,\xi,\beta},Y_t^{t_n,\xi})-V(t,\eta,\xi)\Big|
\\
&\le 
\sup_{\beta\in\cala}\sup_{\alpha\in\cala} \Big|J(t,  X_t^{t_n,\eta,\xi,\beta},Y_t^{t_n,\xi},\alpha)-J(t,\eta,\xi,\alpha)\Big|,
\end{align*}
which tends to $0$ as $n\to\infty$ by \eqref{unifUIbis}, because
\[
\sup_{\beta\in\cala}\|X_t^{t_n,\eta,\xi,\beta}-\eta\|_p\to0, \qquad \|Y_t^{t_n,\xi}-\xi\|_r\to 0,
\]
again 
by   Corollary \ref{continizi}.
\end{proof}

\section{The HJB equation}
\label{sec-HJB}

In this section we suppose that Assumptions (\nameref{HpA1})-(\nameref{HpUniformContinuityfg}) hold true and that $p\ge 2$ and $r\ge 2$ satisfy
\[
q\le r\left(1-\frac{1}{p}\right),\qquad 
\ell\le r\left(1-\frac{1}{p}\right).
\]

\subsection{Some spaces of measures}\label{subsec:spacesmeas}

It is convenient to  rewrite the separated problem 
\eqref{provveq:1}-\eqref{provvfunz:1}   in a way similar to a standard McKean-Vlasov equation, by rewriting  the probability measure $\E\,\left[\langle X_s,1_N\rangle\,|\, (Y_s,\alpha_s)=(y,a)\right]\P_{Y_s,\alpha_s}(dy\,da)$ (on $\R^d\times A$)
as a function of the joint law of $X_s$, $Y_s$, $\alpha_s$.
To this end recall the set $\cald_{p,r}$ introduced in 
\eqref{defdpr}:
\begin{align*}
\cald_{p,r}=\bigg\{ \mu\in\calp(\R^N\times \R^d)\,:\, \mu_1\in\calp_p(\R^N),\, \mu_2\in\calp_r(\R^d),
\,\mu_1(\R^N_+)=1,\, \langle \bar \mu_1,1_N\rangle= 1\bigg\},
\end{align*}
where 
$\mu_1$ and $\mu_2$ denote the marginals of $\mu$ and   $\bar \mu_1\in\R^N$ denotes the mean of $\mu_1$.
Let us also set
\begin{align*}
    \cald= \Big\{ \pi\in\calp(\R^N\times \R^d\times A)\,:\, \pi_{12}\in\cald_{p,r}\Big\},
\end{align*}
where $\pi_{12}$ is the marginal of $\pi$ on $\R^N\times \R^d$.
Take $q\ge 0$ with $q\le r(1-p^{-1})$ and
define functions $\Gamma:\cald\to \calp_q(\R^d\times A)$ and  $\Gamma_1:\cald_{p,r}\to \calp_q(\R^d)$ setting 
\[
\Gamma(\pi)(dy\,da)=\int_{\R^N}\langle x,1_N\rangle\,\pi(dx\,dy\,da), \qquad\Gamma_1(\mu)(dy)=\int_{\R^N}\langle x,1_N\rangle\,\mu(dx\,dy).
\]
Note that $\Gamma(\pi)$ and $\Gamma_1(\mu)$ are nonnegative measures due to the condition $\mu_1(\R^N_+)=1$ and they are probabilities since $\Gamma(\pi)(\R^d\times A)=\Gamma_1(\mu)(\R^d)= \langle \bar \mu_1,1_N\rangle=1 $. 
An application of the Hölder inequality shows that 
$\Gamma(\pi)\in\calp_q(\R^d\times A)$ and 
$\Gamma_1(\mu)\in\calp_q(\R^d)$.  
$\Gamma(\pi)$ is also absolutely continuous with respect to $\pi_{23}\in\calp(\R^d\times A)$, the  marginal of $\pi$ on $\R^d\times A$; indeed, disintegrating $\pi$ in the form $\pi(dx\,dy\,da)=\pi^{y,a}(dx)\pi_{23}(dy\,da)$ for a suitable kernel $\pi^{y,a}(dx)$, we have 
\[
\Gamma(\pi)(dy\,da)=\left(\int_{\R^N}\langle x,1_N\rangle\,\pi^{y,a}(dx)\right)\,\pi_{23}(dy\,da).
\]
Similarly, $\Gamma_1(\mu)$ is absolutely continuous with respect to $\mu_2$.
Now, if $\pi=\P_{X_s,Y_s,\alpha_s}$, then for any bounded measurable $\phi:\R^d\times A\to\R$,
\begin{align*}
\int_{\R^d\times A}\phi(y,a)\, \Gamma(\pi)(dy\,da)&=
\int_{\R^N\times\R^d\times A}\phi(y,a)\langle x,1_N\rangle\,  \pi(dx\,dy\,da)
\\&
=
\E\,[\phi(Y_s,\alpha_s) \,\langle X_s,1_N\rangle]
=
\E\,[\phi(Y_s,\alpha_s)\, \E\,[\langle X_s,1_N\rangle\,|\,Y_s,\alpha_s]]
\\&=
\int_{\R^d\times A}\phi(y,a)\,\E\,[ \langle X_s,1_N\rangle\,|\,(Y_s,\alpha_s)=(y,a)]\,\P_{Y_s,\alpha_s}(dy\,da).
\end{align*}
This proves that $\Gamma(\pi)(dy\,da)=\E\,[ \langle X_s,1_N\rangle\,|\,(Y_s,\alpha_s)=(y,a)]\,\P_{Y_s,\alpha_s}(dy\,da)$.
Similarly, if $\mu=\P_{X_s,Y_s}$, then
$\Gamma_1(\mu)(dy)=\E\,[ \langle X_s,1_N\rangle\,|\,Y_s=y]\,\P_{Y_s}(dy)$.

With this notation, equation \eqref{provveq:1} becomes
\begin{align} 
\left\{
\begin{array}{rcl}
    dX_s &=& \Lambda^T X_s  \,ds + diag(X_s)\,h \Big(Y_s, \Gamma(\P_{X_s,Y_s,\alpha_s}),\alpha_s\Big)\,dW_s,
    \\dY_s&=&\sigma(Y_s)\,dW_s,
    \\
    X_t&=&\eta,
    \\
    Y_t&=&\xi,
    \end{array}
    \right.
\end{align}
and the reward functional \eqref{provvfunz:1} can also be rewritten as
\begin{align}
    J(t,\eta,\xi,\alpha)=
\E\left[\int_t^T \big\langle X_s ,f\big(Y_s,\Gamma(\P_{X_s,Y_s,\alpha_s}),\alpha_s\big) \big\rangle\,ds + 
\big\langle X_T , g\big(Y_T,\Gamma_1(\P_{X_T,Y_T})\big) \big\rangle \right].
\end{align}
We have now rewritten the separated problem in a more usual form. However, due to the occurence of the mappings $\Gamma$ and $\Gamma_1$,    we can not directly apply known results on optimal control of McKean-Vlasov systems.

\bigskip

In order to proceed further we need to introduce an appropriate topology on the space $\cald_{p,r}$. Recall that in \eqref{opttransp} we have defined a mapping 
$\tilde \calw: \cald_{p,r}\times \cald_{p,r}\to \R_+$ which is clearly symmetric and such that 
$  \tilde \calw(\mu,\nu)=0$ if and only if $\mu=\nu$. 
$\tilde  \calw$ is not a metric because it does not satisfy the triangle inequality. However, a standard application of Dudley's glueing lemma in optimal transport theory (see e.g. Lemma 8.4 in \cite{AmbrosioBruéSemola}) and the elementary numerical inequality $(a+b)^c\le 2^{c-1}(a^c+b^c)$ (for $a\ge0$, $b\ge0$, $c\ge1$) shows that setting $K=\max(2^{p-1},2^{r-1})$ one has the  inequality
\[
\tilde \calw(\mu,\nu)\le K\,(\tilde \calw(\mu,\rho)+\tilde \calw(\rho,\nu)), \qquad \mu,\nu,\rho\in \cald_{p,r}.
\]
For its properties $\tilde \calw$ is called a semimetric on the set $\cald_{p,r}$. Some authors prefer the term quasimetric. A semimetric can be used to introduce a topology. For details  we refer to the paper \cite{PalSt09}, which contains a clear discussion and short  self-contained proofs of the facts that we are going to recall, and gives  references to the related literature. 

Given $\tilde \calw$, one can define the semimetric ball centered at $\mu\in\cald_{p,r}$ with radius $r>0$:
\[
B(\mu,r)=\{\nu\in \cald_{p,r}\,:\, \tilde \calw(\mu,\nu)<r\}.
\]
We say that $A\subset \cald_{p,r}$ is open if for any $\mu\in A$ one can find $r>0$ such that $B(\mu,r)\subset A$. The collection $\tau$ of open sets is a topology in $\cald_{p,r}$, as it can be easily verified. From now on we endow $\cald_{p,r}$   with the topology $\tau$ and any related concept, for instance Borel sets and measurability, will refer to it. 

We warn the reader that the semimetric balls $B(\mu,r)$ may fail to be open, as the example in \cite{PalSt09} shows. In the Proposition of Section 2 of \cite{PalSt09} it is proved that there exists a (genuine) metric $d$ on $\cald_{p,r}$ such that
\begin{align}
    \label{equivmetr}
    d(\mu,\nu)\le \tilde \calw(\mu,\nu)^\delta\le 4\, d(\mu,\nu),
\qquad \mu,\nu\in \cald_{p,r},
    \end{align}
where the constant $\delta\in(0,1]$ is chosen to satisfy $(2K)^\delta=2$. 
It follows immediately that  $\tau$ coincides with the topology induced by the metric  $d$. As a consequence, the notion of continuity  coincides with the notion of sequential continuity,  which can be equivalently expressed in terms of $d$ or of $\tilde \calw$. For instance,  Theorem \ref{Continuityofvinciandtime} states that the value function $v$ is continuous (for the topology $\tau$). Similarly, the notion of uniform continuity can be expressed, with the obvious formulations, equivalently in terms of $\tilde \calw$ or of $d$. As noted earlier, if  $p=r$, then $\tau$ is the topology  inherited from the Wasserstein space  $\calp_p(\R^N\times \R^d)$.

\subsection{The HJB equation}

In order to formulate an HJB equation, which turns out to be a parabolic equation for a real function defined on $[0,T]\times\cald_{p,r}$, we need to introduce appropriate notions of derivative for     functions $v:\cald_{p,r}\to\R$. We start from the concept of linear functional derivative, see e.g. \cite{CarmonaDelarue1}-\cite{CarmonaDelarue2}. 

\begin{Definition}
\label{Defderivative} Given a  function
$v:\cald_{p,r}\to \R$,   a measurable function 
\begin{align}
\delta_m  v:  \cald_{p,r} \times \R^N\times \R^d\to\R, \; \qquad \; (\mu,x,y)   \; \mapsto \; \delta_m  v(\mu;x,y)     
\end{align}
is called a  linear functional derivative of $v$ if for any $\mu,\nu\in \cald_{p,r}$ there exists a constant $C_{\mu,\nu}>0$ such that
    $$\left|
    \delta_m  v(\mu+\theta(\nu-\mu);x,y) \right|\le C_{\mu,\nu}\, (1+|x|^p+|y|^r),  $$
    for every  $x\in \R^N$, $y\in\R^d$, $\theta\in [0,1]$, and
\begin{align}
v(\nu)-v(\mu)  =   \int_0^1    \int_{\R^N\times\R^d} \delta_m  v(\mu + \theta(\nu-\mu);x,y) \,(\nu-\mu)(dx\,dy) \,
   d \theta.
\end{align}
\end{Definition}

We also define a space $\tilde C^{1,2}$ of regular test functions $v:[0,T]\times\cald_{p,r}\to\R$. The required conditions allow to prove a suitable  It\^o formula in Proposition \ref{Itogeneral}.

\begin{Definition}
We say that 
a measurable function
 $v:[0,T]\times \cald_{p,r}\to \R$ is of class $\tilde C^{1,2}([0,T]\times \cald_{p,r})$ if \begin{enumerate}
    \item for every $\mu\in \cald_{p,r}$ the function 
$t\mapsto \upsilon(t,\mu)$ is differentiable on $[0,T]$; 
\item 
denoting $\partial_t v$ its time derivative, the function
$\partial_tv(t,\mu)$ is a continuous function of $(t,\mu)\in [0,T]\times \cald_{p,r}$;
\item 
for every $t\in[0,T]$
a derivative $\delta_m  v(t,\mu; x,y)$ exists and it is measurable in all its arguments;
\item
$\delta_m  v(t,\mu;x,y)$ is twice continuously differentiable on $\R^N\times \R^d$ as a function of $(x,y)$
and the gradient and the Hessian matrix
\begin{align*}
\partial_{(x\,y)}\delta_m  v&=
\left(\begin{array}{cc}\partial_x\delta_m  v&\partial_y\delta_m  v
\end{array}\right): [0,T]\times \cald_{p,r}  \times \R^N\times \R^d \to \R^{N+d},
\\
\partial^2_{(x\,y)}\delta_m  v&=\left(\begin{array}{cc}
\partial^2_x\delta_m  v &
\partial_y\partial_x\delta_m  v
\\
\partial_x\partial_y\delta_m  v&
\partial^2_y\delta_m  v
\end{array}\right):[0,T]\times \cald_{p,r}  \times \R^N\times \R^d \to \R^{(N+d)\times (N+ d)},
\end{align*}
are globally bounded: there exists a constant $C\ge 0$
such that
\begin{align}\label{growthpartialv}
\Big|\partial_{(x\,y)}\delta_m  v(t,\mu;x,y) \Big| +\Big|\partial_{(x\,y)}^2\delta_m  v(t,\mu;x,y)\Big| &\le C,
\end{align}
for every $t\in [0,T]$,  $\mu\in \cald_{p,r}$, $x\in\R^N$, $y\in\R^d$;

\item for every   compact set $H$ of   $\R^N\times\R^d$,
the functions  $\partial_{(x\,y)}\delta_m  v(t,\mu;x,y)$ and  $\partial^2_{(x\,y)}\delta_m  v(t,\mu;x,y)$ 
are continuous 
 functions of $(t,\mu)\in [0,T]\times \cald_{p,r}$, uniformly in $(x,y)\in H$; more precisely, whenever $t_n\to t$, $\tilde \calw(\mu_n,\mu)\to 0$,  and $H\subset \R^N\times \R^d$ is compact
we have
$$\sup_{(x,y)\in H}\Big|
\partial_{(x\,y)}\delta_m  v(t_n,\mu_n;x,y) 
-
\partial_{(x\,y)}\delta_m  v(t,\mu;x,y)\Big|\to 0,
$$
$$\sup_{(x,y)\in H}\left|
\partial_{(x\,y)}^2\delta_m  v(t_n,\mu_n;x,y) 
-
\partial_{(x\,y)}^2\delta_m  v(t,\mu; x,y)\right|\to 0.
$$ 
\end{enumerate}
\end{Definition}

\begin{Remark}\label{unifcontpartialv}
By standard arguments, if $K$ is a compact subset of $\cald_{p,r}$ then 
the functions  
$\partial_{(x\,y)}\delta_m  v(t,\mu;x,y)$ and  $\partial^2_{(x\,y)}\delta_m  v(t,\mu;x,y)$   
are uniformly continuous functions of $(t,\mu)\in [0,T]\times K$, uniformly in $(x,y)\in H$;
namely: for every compact sets $H\subset \R^N\times \R^d$, $K\subset \cald_{p,r}$, and every $\epsilon>0$, there exists $\delta>0$ such that
$$\Big|
\partial_{(x\,y)}\delta_m  v(t',\mu';x,y) 
-
\partial_{(x\,y)}\delta_m  v(t,\mu;x,y) \Big|+\Big|
\partial_{(x\,y)}^2\delta_m  v(t',\mu';x,y)
-
\partial_{(x\,y)}^2\delta_m  v(t,\mu;x,y)\Big|< \epsilon
$$ 
whenever $(x,y)\in H$,  $t,t'\in [0,T]$, $  \mu,\mu'\in K$, $\tilde \calw(\mu,\mu')<\delta$, $|t-t'|<\delta$. \qed
\end{Remark}

\begin{Proposition}\label{Itogeneral}
For any control process $\alpha\in\cala$, let  $(X_s,Y_s)_{s\in[t,T]}$ be the unique solution to the state   equation \eqref{provveq:1} provided by Theorem \ref{existstateeqprovv}.   Denote $\mu_s=\P_{X_s,Y_s}$ and define processes
$B$ and 
$\Sigma$, with values  in $\R^{N + d}$ and  $\R^{(N + d)\times d}$ respectively, setting 
\begin{align} 
B_s= 
\left(\begin{array}{c}
\Lambda^T X_s \\
0
\end{array}\right),
\qquad 
\Sigma_s=\left(\begin{array}{c}
diag(X_s)\,h \Big(Y_s, \Gamma(\P_{X_s,Y_s,\alpha_s}),\alpha_s\Big)
\\
\sigma(Y_s)
\end{array}\right).
\end{align}
 Then, for any $v\in\tilde C^{1,2}([0,T]\times \cald_{p,r})$, we have
\begin{align*}
v\left(s,\mu_s\right)-v\left(t,\mu_t\right)&= \int_t^s \Big\{\partial_tv\left(r,\mu_r\right)+\E\,\Big[\partial_{(x\,y)}\delta_m  v\left(r,\mu_r;X_r,Y_r\right)\,B_r
    \\&\quad +
    \frac{1}{2} \,Trace[
(\Sigma_r\Sigma_r^T)\,\partial^2_{(x\,y)}\delta_m  v\left(r,\mu_r;X_r,Y_r\right)]\Big]\Big\}\,dr, \qquad s\in[t,T].
\end{align*} 

\end{Proposition}

\begin{proof} The
state   equation \eqref{provveq:1} can be written in the form 
\begin{align*} 
    d(X_s\;Y_s) =B_s \,ds +
    \Sigma_s\,dW_s, \qquad (X_t,Y_t)=(\eta,\xi).
\end{align*}
The proof of the Proposition can be achieved along the lines of well known arguments,
see e.g. Section 5.6 in \cite{CarmonaDelarue1}. 
For instance the proof of Theorem 4.6 in \cite{DeCrescenzoFuhrmanKharroubiPham} uses notation close to ours and can be repeated with minor changes and even simplifications (replace the set $U$ in \cite{DeCrescenzoFuhrmanKharroubiPham} by a singleton and the squared distance ${\bf d}^2$ by $\tilde \calw$).
\end{proof}

\vspace{1mm}

For convenience we rewrite the It\^o formula in a different way. First we use the differential notation:
\begin{align*}
\frac{d}{ds} v\left(s,\mu_s\right)\!=\! \partial_tv \left(s,\mu_s\right)+\E\Big[\partial_{(x\,y)}\delta_m  v\left(s,\mu_s;X_s,Y_s\right)\,B_s
   + \frac{1}{2} Trace[
(\Sigma_s\Sigma_s^T)\,\partial^2_{(x\,y)}\delta_m  v\left(s,\mu_s;X_s,Y_s\right)]\Big].
\end{align*} 
Next we define   $B: \R^N\to \R^{N+d}$ and
$\Sigma: \R^N\times \R^d\times \cald\times A\to \R^{(N + d)\times d}$ setting
\begin{align} 
B(x)= 
\left(\begin{array}{c}
\Lambda^T x \\
0
\end{array}\right), 
\qquad
\Sigma(x,y,\pi,a)=
\left(\begin{array}{c}
diag(x)\,h (y, \Gamma(\pi),a)
\\
\sigma(y)
\end{array}\right),
\end{align}
so that
$B_s  =B (  X_s)$,  $\Sigma_s=\Sigma(  X_s, Y_s,  \P_{X_s,Y_s,\alpha_s},\alpha_s)$, and the It\^o formula becomes
\begin{align*}
\frac{d}{ds}\,v\left(s,\P_{X_s,Y_s}\right)&= \partial_sv \left(s,\P_{X_s,Y_s}\right)+\E\,\Big[\partial_{(x\,y)}\delta_m  v\left(s,\P_{X_s,Y_s};X_s,Y_s\right)\,B(X_s)
    \\&\quad +
    \frac{1}{2} Trace[
(\Sigma\Sigma^T)\left(X_s,Y_s,\P_{X_s,Y_s,\alpha_s},\alpha_s\right)\partial^2_{(x\,y)}\delta_m  v\left(s,\P_{X_s,Y_s};X_s,Y_s\right)]\Big].
\end{align*}
Next we note that
$\partial_{(x\,y)}\delta_m  v\,B= \partial_x\delta_m  v\,\,\Lambda^Tx$, and since 
\[
\Sigma\Sigma^T= 
\left(\begin{array}{cc}
 diag(x)\,(hh^T) (y, \Gamma(\pi),a)\,diag(x)&
 diag(x)\,h (y, \Gamma(\pi),a)\,\sigma(y)^T
\\
\sigma(y)\, h^T (y, \Gamma(\pi),a)\,diag(x)
 &
 \sigma\sigma^T(y)
\end{array}\right)
\]
we obtain
\begin{align*}
Trace[
\Sigma\Sigma^T\partial^2_{(x\,y)}\delta_m  v]&= Trace[
 diag(x)\,(hh^T) (y, \Gamma(\pi),a)\,diag(x)\partial^2_x\delta_m  v]
 \\&\quad
 +2Trace[
  diag(x)\,h (y, \Gamma(\pi),a)\,\sigma(y)^T\, 
\partial_x\partial_y\delta_m  v]+Trace[
\sigma\sigma^T(y)\partial^2_y\delta_m  v] 
 . 
\end{align*}
So the It\^o formula can written in the form
\begin{align}\label{itosolution}
    \frac{d}{ds}\,v\left(s,\P_{X_s,Y_s}\right)=\partial_sv\left(s,\P_{X_s,Y_s}\right)+\E\,\Big[\call^{\alpha_s}v\left(s,X_s,Y_s,\P_{X_s,Y_s,\alpha_s}\right)\Big]
\end{align}
provided that we define
\begin{align*}
&\call^{a}v\left(s,x,y,\pi\right)= \partial_x\delta_m v(s,\pi_{12};x,y)\Lambda^Tx
\\&\quad 
+\frac{1}{2}\,
Trace[
 diag(x)\,(hh^T) (y, \Gamma(\pi),a)\,diag(x)\partial^2_x\delta_m  v(s,\pi_{12};x,y)]
    \\&\quad 
    +Trace[
  diag(x)\,h (y, \Gamma(\pi),a) \sigma(y)^T
\partial_x\partial_y\delta_m  v(s,\pi_{12};x,y)]+\frac{1}{2}\,
Trace[
\sigma\sigma^T(y)\partial^2_y\delta_m  v(s,\pi_{12};x,y)] , 
\end{align*}
for $s\in[0,T]$, $x\in\R^N_+$, $y\in\R^d$, $\pi\in\cald$, $a\in A$, and recalling that $\pi_{12}$ denotes the marginal of $\pi$ on $\R^N\times \R^d$.
This leads to considering the following  form of the HJB equation, a parabolic equation for an unknown function $w:[0,T]\times\cald_{p,r}\to\R$:
\begin{align} \label{HJB}
\left\{
\begin{array}{l}
\dis    -\partial_tw(t,\mu) -
\sup_{\pi\in \cald\,:\,\pi_{12}=\mu} \int_{\R^N\times\R^d\times A}\left[\call^{a}w\left(t,x,y,\pi\right) +\langle x,f(y,\Gamma(\pi),a)\rangle
    \right] \pi(dx\,dy\,da)=0,
    \\\dis
w\left(T,\mu\right)=\int_{\R^N\times \R^d}\langle x,g(y,\Gamma_1(\mu))\rangle\,\mu(dx\, dy),
   \qquad\qquad t\in [0,T],\; \mu\in\cald_{p,r}.
   \end{array}
    \right.
\end{align}

Our first result is the following verification theorem, where $v$ denotes the value function defined in \eqref{defvaluefunction}.

\begin{Theorem} \label{verifth}
    Suppose that $w\in\tilde C^{1,2}([0,T]\times\cald_{p,r})$ is a solution to \eqref{HJB}. Then $w\ge v$.

Suppose further that one can find a  measurable function $\widehat a (t,x,y,\mu)$, with values in $A$, such that at any $t$ the $\sup$ in \eqref{HJB} is attained at some point $\widehat\pi$ of the form
\[
\widehat\pi(dx\,dy\,da)=\delta_{\widehat a (t,x,y,\mu)}(da)\,\mu(dx\,dy);
\]
also assume that the closed-loop system
\begin{align*} 
\left\{
\begin{array}{rcl}
    dX_s &=& \Lambda^T X_s  \,ds + diag(X_s)\,h \Big(Y_s, \Gamma(\P_{X_s,Y_s,\widehat a (s,X_s,Y_s,\P_{X_s,Y_s})}),\widehat a (s,X_s,Y_s,\P_{X_s,Y_s})\Big)\,dW_s,
\\dY_s&=&\sigma(Y_s)\,dW_s,
    \\
    X_t&=&\eta,
    \\
    Y_t&=&\xi,
    \end{array}
    \right.
\end{align*}
has a solution, in the sense of Theorem  \ref{existstateeqprovv}. Then   the process $\widehat \alpha$ defined by 
\[
\widehat \alpha_s=
\widehat a (s,X_s,Y_s,\P_{X_s,Y_s}), \qquad s\in[t,T],
\]
belongs to $\cala$, we have $w=v$ and $\widehat \alpha$ is optimal. 
\end{Theorem}

\begin{proof}
The argument is standard, see for instance \cite{PhamWei2} and \cite{DjePosTan22} in the context of McKean-Vlasov systems, so we only give a sketch of proof.

Let $\alpha$ be an admissible control and $(X,Y)$ the corresponding trajectory starting from $(\xi,\eta)$ at time $t$, with $\mu=\P_{\xi,\eta}$.  Writing the reward functional in the form
\begin{align*}
J(t,\eta,\xi,\alpha)=&
\int_t^T\int_{\R^N\times\R^d\times A} \big\langle x ,f\big(y,\Gamma(\P_{X_s,Y_s,\alpha_s}),a\big)\big\rangle\,\P_{X_s,Y_s,\alpha_s}(dx\,dy\,da)\,ds 
\\&+ \int_{\R^N\times\R^d} \big\langle x , g\big(y,\Gamma_1(\P_{X_T,Y_T})\big) \big\rangle \,\P_{X_T,Y_T}(dx\,dy) ,
\end{align*}
and applying the It\^o formula \eqref{itosolution} we obtain
\begin{align*}
w(t,\mu)&=
J(t,\eta,\xi,\alpha)-
\int_t^T\Big\{\partial_tw(s,\P_{X_s,Y_s)}
\\&+ \int_{\R^N\times\R^d\times A}\Big[
\call^{a}w\left(s,x,y,\P_{X_s,Y_s,\alpha_s}\right)+
\big\langle x ,f\big(y,\Gamma(\P_{X_s,Y_s,\alpha_s}),a\big) \big\rangle\Big]\,\P_{X_s,Y_s,\alpha_s}(dx\,dy\,da)\Big\}\,ds .
\end{align*}
As $w$ is a solution to the HJB equation, the term $\{\ldots\}$ in curly brackets is $\le0$, which implies $w(t,\mu)\ge J(t,\eta,\xi,\alpha)$ and so $w\ge v$. When the control equals $\widehat\alpha$ we have $\{\ldots\}=0$ and so $w(t,\mu)= J(t,\eta,\xi,\widehat\alpha)$, which implies the optimality of $\widehat\alpha$ and the equality $w=v$.
\end{proof}

\begin{Remark}\label{verifoptpartialobs}
When all the assumptions of the previous theorems hold, and the initial conditions are deterministic, then the optimal control constructed above is indeed predictable with respect to the smaller filtration $\cala^W$ and it is optimal for the partially observed control problem as well: compare with the comments following Proposition \ref{AeAW}. \qed 
\end{Remark}

\begin{Remark}
In the simpler case where there is no dependence of the coefficients on the law of the control, namely the control problem takes the form
\begin{align*} 
\left\{
\begin{array}{rcl}
    dX_s &=& \Lambda^T X_s  \,ds + diag(X_s)\,h \big(Y_s, \Gamma_1(\P_{X_s,Y_s}),\alpha_s\big)\,dW_s,
    \\dY_s&=&\sigma(Y_s)\,dW_s,
    \\
    X_t&=&\eta,
    \\
    Y_t&=&\xi,
    \end{array}
    \right.
\end{align*}
\begin{align*}
J(t,\eta,\xi,\alpha)=
\E\left[\int_t^T \big\langle X_s ,f\big(Y_s,\Gamma_1(\P_{X_s,Y_s}),\alpha_s\big)\big\rangle\,ds + \langle X_T , g\Big(Y_T,\Gamma_1(\P_{X_T,Y_T})\Big)\rangle \right],
\end{align*}
then $\call^{a}v\left(x,y,\pi\right)$ is in fact a function of the form $\call^{a}v\left(x,y,\mu\right)$ for $\mu\in\cald_{p,r}$
and we are led to considering a modified HJB equation:
\begin{align} 
\left\{
\begin{array}{l}
\dis    -\partial_tw(t,\mu) -
    \int_{\R^N\times \R^d} \sup_{a\in A}\left[\call^{a}w\left(t,x,y,\mu\right) +\langle x,f(y,\Gamma_1(\mu),a)\rangle
    \right] \mu(dx\,dy)=0,
    \\\dis
w\left(T,\mu\right)=\int_{\R^N\times \R^d}\langle x,g(y,\Gamma_1(\mu))\rangle\,\mu(dx\, dy),
   \qquad\qquad t\in [0,T],\; \mu\in\cald_{p,r}.
   \end{array}
    \right.
\end{align}
The following analogue of Theorem  \ref{verifth} holds, with similar proof:
if a   solution  $w\in\tilde C^{1,2}([0,T]\times\cald_{p,r})$ exists, then $w\ge v$; and if in addition the $\sup$ is attained at some point
\[
a=\widehat a (t,x,y,\mu),
\]
for some measurable function $\widehat a$,
and the closed loop equation admits a solution, then $w=v$ and $\widehat a$ is an optimal control law. 
\qed 
\end{Remark}

\medskip

We end this Section proving viscosity properties of  the value function. We introduce the following notion of viscosity solution to the HJB equation where, in view of the continuity results on the value function that we proved in Section \ref{SectionPropv}, it is not required to introduce upper and lower envelopes.

\begin{Definition}
\begin{enumerate}[(i)]
\item An upper semicontinuous function $w:[0,T]\times \cald_{p,r}\rightarrow \R$ is called a viscosity subsolution to   \eqref{HJB}  if
\[
w\left(T,\mu\right)\le \int_{\R^N\times \R^d}\langle x,g(y,\Gamma_1(\mu))\rangle\,\mu(dx\, dy),
   \qquad \mu\in\cald_{p,r},
\]
and for any $\varphi\in \tilde C^{1,2}([0,T]\times \cald_{p,r})$  and $(t,\mu)\in [0,T)\times \cald_{p,r}$ such that
\[
(w-\varphi)(t,\mu)  =  \max_{[0,T]\times \cald_{p,r}}(w-\varphi) 
\]
we have
\[ -\partial_t\varphi(t,\mu) -
\sup_{\pi\in \cald\,:\,\pi_{12}=\mu} \int_{\R^N\times\R^d\times A}\left[\call^{a}\varphi\left(t,x,y,\pi\right) +\langle x,f(y,\Gamma(\pi),a)\rangle
\right] \pi(dx\,dy\,da)\le 0.
\]
                   
\item A lower semicontinuous function $w:[0,T]\times \cald_{p,r}\rightarrow \R$ is called a viscosity supersolution to   \eqref{HJB}  if
       \[
w\left(T,\mu\right)\ge \int_{\R^N\times \R^d}\langle x,g(y,\Gamma_1(\mu))\rangle\,\mu(dx\, dy),
   \qquad \mu\in\cald_{p,r},
\]
and for any $\varphi\in \tilde C^{1,2}([0,T]\times \cald_{p,r})$  and $(t,\mu)\in [0,T)\times \cald_{p,r}$ such that
\[
(w-\varphi)(t,\mu)  =  \min_{[0,T]\times \cald_{p,r}}(w-\varphi) 
\]
we have
\[ -\partial_t\varphi(t,\mu) -
\sup_{\pi\in \cald\,:\,\pi_{12}=\mu} \int_{\R^N\times\R^d\times A}\left[\call^{a}\varphi\left(t,x,y,\pi\right) +\langle x,f(y,\Gamma(\pi),a)\rangle
\right] \pi(dx\,dy\,da)\ge 0.
\]

        \item We say that $w$ is a viscosity solution to   \eqref{HJB}  if  $w$ is both a viscosity subsolution and supersolution to   \eqref{HJB}. 

    \end{enumerate}
\end{Definition}

\begin{Theorem}\label{thsolvisc}
   The value function $v$ is a viscosity solution to \eqref{HJB}.
    
\end{Theorem}

\begin{proof}
    The required arguments are classical: for the case of McKean-Vlasov systems one can consult for instance  \cite{PhamWei1}
or \cite{CosGozKhaPhaRos23}. We postpone the proof to the appendix, where we show how to deduce the required conclusion under our assumptions.  
\end{proof} 
 
\begin{Remark}
If a function $w\in\tilde C^{1,2}([0,T]\times \cald_{p,r})$
is a viscosity solution to  \eqref{HJB}, it follows from the previous definition  that   $w$ is a classical solution, as one can take $\varphi=w$ as a test function. In particular, if the value function $v$ belongs to $\tilde C^{1,2}([0,T]\times \cald_{p,r})$ then it is a classical solution to  \eqref{HJB}. \qed
\end{Remark}

 \section{Examples and applications} \label{SectionExamples}

In this section we first exhibit some examples of functions satisfying Assumptions (\nameref{HphLipschitz}) and (\nameref{HpUniformContinuityh}). In spite of their unusual appearance,  these assumptions are not too restrictive and they are in fact satisfied by   large classes of functions of probability measures. We conclude this section with a financial application to an optimal liquidation trading problem.

\begin{Example}
    \label{ExampleCylindricalFct}
    Let \(h\) be a cylindrical function, i.e.
    \begin{equation} \label{EqCylindricalFunction}
     h(y,a,\nu,i):=F \bigg( y,a, \int_{\R^d \times A} \psi_1(y,a) \,\nu(dy,da), \ldots, \int_{\R^d \times A} \psi_k(y,a)\, \nu(dy,da), i \bigg),
    \end{equation}
    where \(F: \R^d \times A \times \R^k \times S \rightarrow \R^d\) is a bounded function which is Lipschitz in \(\R^k\) uniformly with respect to \((y,a,i) \in \R^d \times A \times S\) and, for all \(j=1, \ldots, k\),  \(\psi_j:\R^d\times A\to\R\) is measurable and satisfies a polynomial growth condition of order \(m \geq 0\), i.e.\ \[ |\psi_j(y,a)| \leq C(1+ |y| + d(a,a_0))^m \text{ for all } j=1,\ldots,k. \] 
    Then, Assumption (\nameref{HphLipschitz}) is satisfied with \(\chi(z)=z^m\), \(r \geq 2m\), \(q \geq m\) for any \(p \geq 2\). 
    Moreover, assume that \(F\) is also uniformly continuous in \(y \in \R^d\)  uniformly with respect to \((a,i,z) \in A \times S \times \R^k\), i.e.\ there exists a modulus of continuity \(\omega_F\) such that
    \[ |F(y,a,z,i)-F(y',a,z,i)| \leq \omega_F(|y-y'|) \]
    for all \(y,y' \in \R^d\), \(a \in A\), \(i \in S\) and \(z \in \R^k\).
    Furthermore, let the function \(\psi_j\) be locally Lipschitz for \(j=1,\ldots,k\), i.e.\ there exist a constant \(C>0\) and an exponent \(v \geq 0\) such that
    \begin{equation} \label{LocalLipschitzConditionpsij}
        |\psi_j(y,a)-\psi_j(y',a')| \leq C(|y-y'| +d(a,a')) (1+|y|+|y'|+d(a,a_0)+d(a',a_0))^v
    \end{equation}
    for all \(y,y' \in \R^d\), \(a,a' \in A\) , \(j=1, \ldots, k\).
    Then,  Assumptions (\nameref{HpUniformContinuityh}) are satisfied with \(\omega_0=\omega_F\), \(\omega_q(z)=Cz\) for some constant \(C\), \(\chi(z)=z^v\) provided that \(v \leq q-1\).

    Notice that the class of functions satisfying \eqref{EqCylindricalFunction} includes two notable examples. The first one is the class of functions of the mean of the form
    \begin{equation} \label{ExampleMean}
        h(\nu):=F \bigg( \int_{\R^d \times A} (y,a)\; \nu(dy,da) \bigg),
    \end{equation}
    where \(F\) is a bounded and Lipschitz function,  which is well defined when \(A\) is a  bounded subset of a Euclidean space \(\R^D\). These functions satisfy Assumption (\nameref{HphLipschitz}) for \(\chi(z)=z\), \(r \geq 2\), \(q \geq 1\). Moreover, Assumptions (\nameref{HpUniformContinuityh}) are satisfied with \(\omega_q(z)=L_F z\), where \(L_F\) is the Lipschitz constant of the function \(F\).
    
    The second remarkable subset of functions satisfying \eqref{EqCylindricalFunction} is given by the functions of the \(\gamma\)-th moment of their measure argument, i.e.\
    \begin{equation} \label{ExampleMoment}
        h(\nu):= F \bigg( \int_{\R^d \times A} (|y| + d(a,a_0))^{\gamma} \nu(dy,da) \bigg)
    \end{equation}
    for some \(\gamma \geq 1\),  where \(F\) is bounded and Lipschitz. These functions satisfy Assumption (\nameref{HphLipschitz}) for \(\chi(z)=z^k\), \(r \geq 2\gamma\), \(q \geq \gamma \). Moreover, as \(k=1\) and \(\psi(y,a)=(|y|+d(a,a_0))^\gamma\) is a locally Lipschitz function of the form \eqref{LocalLipschitzConditionpsij} with \(v=\gamma-1\), Assumptions (\nameref{HpUniformContinuityh}) are satisfied with \(\omega_q(z)=C z\) and \(\chi(z)=z^{\gamma-1}\) and \(q \geq \gamma\).
\end{Example}

\begin{Example}
    \label{ExampleCovariance}
    Let \(h\) be a function of the covariance of a probability measure $\nu$, i.e.\ \(h\) depends on the first marginal $\nu_1$ on $\R^d$ of \(\nu\) and it is defined by
\begin{equation} \label{EqExampleCovariance}
        h(\nu):=F \bigg( \int_{\R^d } zz^T \nu_1(dz) - \bigg( \int_{\R^d } z \nu_1(dz) \bigg) \bigg( \int_{\R^d } z \nu_1(dz) \bigg)^T\bigg),
    \end{equation}
where \(F\) is a bounded and Lipschitz function. Then,  Assumption (\nameref{HphLipschitz}) is satisfied with \(q \geq 2\), \(r \geq 4\) and \(\chi(z)=z^2\) (or, in general, \(\chi\) polynomial function). Moreover, Assumptions (\nameref{HpUniformContinuityh}) 
are satisfied with \(\omega_q(z)=C z\) for some constant \(C\) and \(\chi\) polynomial function for \(q \geq 2\).
\end{Example}
We underline that the functions in Example \ref{ExampleCylindricalFct} and Example \ref{ExampleCovariance} can also be used as the functions $f$ or $g$ in the gain functional, as they satisfy also Assumptions (\nameref{HpUniformContinuityfg}), even if we remove the boundedness assumption on $F$ and we only require that \(|F(y,a,i,0)| \leq C\) for any \(y \in \R^d, a \in A, i \in S\).

\begin{Example}[An optimal liquidation trading problem]
We consider the example of optimal liquidation problem with permanent price impact induced by trade crowding as in \cite{CarLeh18}:  the inventory of the trader that has to liquidate a certain number of stock shares $I_0$ $>$ $0$ 
at a trading  rate $\alpha$ $=$ $(\alpha_t)_t$,  is driven by  
\begin{align}
d I_t &= \; - \alpha_t dt + \eps d B_t^1,     
\end{align}
where $B^1$ is a Brownian motion on $(\Omega,\Fc,\Q)$, and $\eps$ is a small noise around the inventory speed. The transaction price of the 
stock shares is given by 
\begin{align}
P_t &= \; P_t^0 -  \nu  \int_0^t \E_{\Q}[\alpha_s] d s,      
\end{align}
where $P^0$ is the reference stock price in absence of trading, and $\nu$ $>$ $0$ is the positive parameter for the 
permanent price impact due to the average trading of large number of market participants. We assume that $P^0$ is governed by 
\begin{align}
d P_t^0 &= \; M_t d t + \gamma d B_t^2,    
\end{align}
where $M$ is the unobserved drift of the stock price, modeled as a regime switching Markov chain (e.g. with two states: bearish or bullish), and $\gamma$ is the volatility with $B^2$ a Brownian motion independent of $B^1$.  The goal of the trader is to maximize the reward function 
\begin{align}
J(\alpha) &= \; \E_{\Q} \Big[ \int_0^T \alpha_t ( P_t - \eta \alpha_t) d t - \theta I_T^2 \Big],    
\end{align}
where $\eta$ $\geq$ $0$ is a positive constant parameter representing the linear temporary price impact, and $\theta$ $\geq$ $0$ is a parameter for penalizing the remaining inventory at maturity $T$. 

This example fits into our partially observed mean-field control problem \eqref{stateeq1}-\eqref{rewardfunz1} with observation process $Y$ $=$ $(I,P)$ and 
\begin{align}
\sigma \; = \; \Big( \begin{array}{cc}
                      \eps & 0 \\
                      0 &  \gamma 
                     \end{array}  \Big), & \qquad 
h(Y_t,\Q_{Y_t,\alpha_t},\alpha_t,M_t) \; = \;   \Big( \begin{array}{c}
                      - \frac{1}{\eps} \alpha_t   \\
                      \frac{1}{\gamma} \big(  M_t - \nu \E_{\Q}[\alpha_t] \big) 
                     \end{array}  \Big), \\
f(Y_t,\Q_{Y_t,\alpha_t},\alpha_t,M_t) \; = \;  \alpha_t (P_t - \eta \alpha_t), & \qquad 
g(Y_T,\Q_{Y_T},M_T) \; = \; - \theta I_T^2. 
\end{align}
\end{Example}

\appendix
 
\section{Appendix} \label{AppendixProofs}

In this section we have postponed the proofs of several results, for better readibility of the main text.

\paragraph{Proof of Corollary \ref{flowcor}.}

    The thesis is obvious for the first component \(Y\): it follows from the flow property of solutions to classical SDEs. For the second component, call for simplicity 
    \[ K_r:=X_r^{t,\eta,\xi,\alpha}  \quad Y_r:=Y_r^{t,\xi} \quad \text{for all } r \in [s,T]. \]
    Then, by definition \(K\) satisfies
    \begin{align*}
        K_r=X_r^{t,\eta,\xi,\alpha}&= \eta + \int_t^r \Lambda^T X_u^{t,\eta,\xi,\alpha} du \\
        &+ \int_t^r diag(X_u^{t,\eta,\xi,\alpha}) h(Y_u, \E[\langle X_u^{t,\eta,\xi,\alpha}, 1_N \rangle| (Y_u,\alpha_u)=(y,a) ] \P_{Y_u,\alpha_u}(dy,da), \alpha_u) dW_u \\
        &=\eta + \int_t^s \Lambda^T X_u^{t,\eta,\xi,\alpha} du + \int_s^r \Lambda^T X_u^{t,\eta,\xi,\alpha} du \\
        &+ \int_t^s diag(X_u^{t,\eta,\xi,\alpha}) h(Y_u, \E[\langle X_u^{t,\eta,\xi,\alpha}, 1_N \rangle| (Y_u,\alpha_u)=(y,a) ] \P_{Y_u,\alpha_u}(dy,da), \alpha_u) dW_u \\
        &+\int_s^r diag(X_u^{t,\eta,\xi,\alpha}) h(Y_u, \E[\langle X_u^{t,\eta,\xi,\alpha}, 1_N \rangle| (Y_u,\alpha_u)=(y,a) ] \P_{Y_u,\alpha_u}(dy,da), \alpha_u) dW_u .
    \end{align*}
    Since \(X_u^{t,\eta,\xi,\alpha}=K_u\) for \(u \in [s,T]\),
    \begin{align*}
        K_r&=  X_s^{t,\eta,\xi,\alpha} + \int_s^r \Lambda^T K_u du  \\
        &+ \int_s^r diag(K_u) h(Y_u, \E[\langle K_u, 1_N \rangle| (Y_u,\alpha_u)=(y,a) ] \P_{Y_u,\alpha_u}(dy,da), \alpha_u) dW_u. 
    \end{align*}
    As a consequence, \(K\) solves the following equation:
    \begin{equation} \label{ProofFlowPPt1}
        \begin{cases}
            dK_r = \Lambda^T K_r dr + diag(K_r) h(Y_r, \E[\langle K_r, 1_N \rangle| (Y_r,\alpha_r)=(y,a) ] \P_{Y_r,\alpha_r}(dy,da), \alpha_r) dW_r, \quad r \in [s,T] &\\
            K_s=X_s^{t,\eta,\xi,\alpha}.
        \end{cases}
    \end{equation}
    Now notice that the process \(\tilde{K}:=X^{s,X_s^{t,\eta,\xi,\alpha},Y_s^{t,\xi}, \alpha}\) is the unique solution to the following equation:
    \begin{equation} \label{ProofFlowPPt2}
        \begin{cases}
            d\tilde{K}_r = \Lambda^T \tilde{K}_r dr + diag(\tilde{K}_r) h(Y_r^{s,Y_s^{t,\xi}}, \E[\langle \tilde{K}_r, 1 \rangle| (Y_r^{s,Y_s^{t,\xi}},\alpha_r)=(y,a) ] \P_{Y_r^{s,Y_s^{t,\xi}},\alpha_r}(dy,da), \alpha_r) dW_r, \, r \in [s,T] &\\
            \tilde{K}_s=X_s^{t,\eta,\xi,\alpha}.
        \end{cases}
    \end{equation}
    Recalling that \(Y_r^{t,\xi}=Y_r^{s,Y_s^{t,\xi}}\) for \(r \in [s,T]\), we observe that \eqref{ProofFlowPPt1} and \eqref{ProofFlowPPt2} are the same stochastic differential equation on \([s,T]\) with the same initial condition; by Theorem \ref{existstateeqprovv}, we conclude that
    \[ X_r^{s,X_s^{t,\eta,\xi,\alpha},Y_s^{t,\xi}, \alpha} = X_r^{t,\eta,\xi,\alpha} \quad \text{for all } r \in [s,T].  
    \]
\qed
 
\medskip

\paragraph{Proof of Proposition \ref{AeAW}.}   

Firstly notice that, by definition of the filtrations \(\F\) and \(\F^W\) (see Assumption (\nameref{HpU})), it is clear that \(\cala^W \subseteq \cala\). As a consequence,  
\[ \sup_{\alpha\in\cala} J(0,\eta,\xi,\alpha) \geq \sup_{\alpha\in\cala^W} J(0,\eta,\xi,\alpha). \]
We thus only need to prove the opposite inequality.

\textit{Step 1:}  We will prove that
\[ \sup_{\alpha\in\cala} J(t,\eta,\xi,\alpha) \leq \sup_{\alpha\in\cala^W} J(t,\eta,\xi,\alpha) \]
for any fixed   \(t > 0\). Choose any \(\alpha \in \cala\). By a monotone class argument (arguing e.g.\ as in Proposition 10 in \cite{Claisse} or in Lemma 2.3 in \cite{Rudà} for details),   there exists a measurable function
\[ \underline{\alpha}: [0,T] \times \calc([0,T]; \R^d) \times [0,1] \rightarrow A  \]
such that, for \(\P\)-a.s. \(\omega \in \Omega\), \(\alpha_s(\omega)=\underline{\alpha}(s, W_{s \wedge .}(\omega),U(\omega))\), where \(W_{s \wedge .}\) denotes the Brownian trajectory associated to \(\omega\) stopped at time \(s\). Now construct the process \(\tilde{W}=(\tilde{W}_s)_{s \in [0,T]}\) defined as follows:
\begin{equation*}
    \tilde{W}_s:= \begin{cases}
        \sqrt{2}(W_{\frac{s+t}{2}}-W_{\frac{t}{2}}) \qquad s \in [0,t) &\\
        \sqrt{2}(W_{t}-W_{\frac{t}{2}})+W_s-W_t \qquad s \in [t,T].&
    \end{cases}
\end{equation*}
It can be verified that \(\tilde{W}\) is a standard Brownian motion on \([0,T]\)  with respect to the filtration \((\calf_{\frac{s+t}{2}}^W)_{s \in [0,T]}\); it is also independent of \(\calf_{\frac{t}{2}}^W\). Moreover, define the standard gaussian random variable \(Z:= \frac{2}{\sqrt{t}}W_{\frac{t}{4}}\); denote by \(F_Z\) its cumulative distribution function and observe that \(\tilde{U}:=F_Z(Z)\) is a uniform random variable supported in \([0,1]\) and measurable with respect to \(\calf_{\frac{t}{4}}^W \subseteq \calf_{\frac{t}{2}}^W\). Therefore, \((\tilde{W}, \tilde{U})\) is another pair of a Brownian motion on \([0,T]\) and an independent uniform random variable. Finally, construct the control \(\tilde{\alpha}\) defined as follows:
\[ \tilde{\alpha}_s(\omega):= \underline{\alpha}(s, \tilde{W}_{s \wedge \cdot}(\omega), \tilde{U}(\omega)). \]
\(\tilde{\alpha}\) does not depend on the additional source of randomness \(U\); in particular, it belongs to \(\cala^W\). Moreover, it is clear that
\[ \call(W, U, \alpha)= \call(\tilde{W}, \tilde{U}, \tilde{\alpha}), \]
where $\call$ denotes the law under $\P$.
As the initial conditions \((\xi, \eta)\) are deterministic, they are both \(\sigma(U)\)-measurable and \(\sigma(\tilde{U})\)-measurable. Then it follows  that 
\[ \call(X^{t,\eta,\xi,\alpha}, Y^{t,\xi}, B, \alpha) = \call(X^{t,\eta,\xi,\tilde{\alpha}}, Y^{t,\xi}, \tilde{B}, \tilde{\alpha})
\]
Therefore, \(J(t,\eta,\xi,\alpha)=J(t,\eta,\xi,\tilde{\alpha})\). We conclude that for any \(\alpha \in \cala\) we can find a \(\tilde{\alpha} \in \cala^W\) such that \(J(t,\eta,\xi,\alpha)=J(t,\eta,\xi,\tilde{\alpha})\), which implies that \(\sup \limits_{\alpha \in \cala} J(t,\eta,\xi,\alpha) \leq \sup \limits_{\alpha \in \cala^W} J(t,\eta,\xi,\alpha).  \)

\textit{Step 2:} 
In Step 1 we have proved that
\[\sup_{\alpha\in\cala} J(t,\eta,\xi,\alpha)=\sup_{\alpha\in\cala^W} J(t,\eta,\xi,\alpha)
    \]
holds for any $t>0$. Using the continuity condition \eqref{rightcontatzero} and the fact that $\cala^W\subset\cala$, we can pass to the limit as $t\to 0$ and obtain the conclusion. We notice that 
\eqref{rightcontatzero} holds under the Assumptions (\nameref{HpA1})-(\nameref{HpUniformContinuityfg}): this will be proved below, see Theorem \ref{Continuityofvinciandtime}-(ii), in particular the proof that the right-hand side of \eqref{destracontJ} tends to zero. \qed

\medskip

\paragraph{Proof of Proposition \ref{stabtateeqprovv}.}  

In this proof we denote by   $C$ any generic constant which depends on the assumptions and by $K$ any constant that also depends on $L,p,r$, but the values of $C$ and $   K$  may change from line to line. We will also freely use
  the a priori estimates 
\begin{align}
    \label{stapdoppia}
    \|X_s\|_p \le C\, \|\eta\|_p,\quad
\|X_s'\|_p \le C\, \|\eta'\|_p,
\quad\|Y_s\|_r \le C\, (1+\|\xi\|_r),
\quad
\|Y_s'\|_r \le C\,(1+ \|\xi'\|_r),
\quad s\in[t,T],
\end{align}
which follow from
 Theorem \ref{existstateeqprovv} (compare
    \eqref{stap}-\eqref{stapY}). 

Due to the Lipschitz character of $\sigma$, the estimate \eqref{stabilstateeqY} is standard.

Let us denote \[
\bar\eta=\eta-\eta',\quad
\bar\xi=\xi-\xi',
\quad
\bar X_s=X_s-X_s',
\quad \bar Y_s=Y_s-Y_s'.
\]
Subtracting the equations for $X$ and $X'$ we obtain
\begin{align}  
    d\bar X_s =&\Lambda^T \bar X_s  \,ds + I_s
    \,dW_s,\qquad
  \bar  X_t=\bar\eta.
\end{align}
where we set
\begin{align*}
    I_s=  &\;
diag(X_s)\,h \Big(Y_s,\E\,\left[\langle X_s,1_N\rangle\,|\, (Y_s,\alpha_s)=(y,a)\right]\,\P_{Y_s,\alpha_s}(dy\,da),\alpha_s\Big)
   \\&
    -diag(X_s')\,h \Big(Y_s',\E\,\left[\langle X_s',1_N\rangle\,|\, (Y_s',\alpha_s)=(y,a)\right]\,\P_{Y'_s,\alpha_s}(dy\, da),\alpha_s\Big).
  \end{align*}
By the Burkholder-Davis-Gundy and the H\"older inequality we have, for $s\in [t,T]$,
\begin{align}\nonumber
\E\,\sup_{u\in [t,s]}|\bar X_u|^p& 
\le C\, \E|\bar\eta|^p+
C\, \E\,\left[\left(\int_t^s |\Lambda^T\bar X_u|\,du\right)^p\right]
+ C\, \E\,\left[\left(\int_t^s |I_u|^2\,du\right)^{p/2}\right]
\\\label{BDGsuXbar}&
\le C\, \E|\bar\eta|^p+
 C\,  \int_t^s \E\,\left[| \bar X_u|^p\right]\,du 
+ C\,  \int_t^s \E\,\left[|I_u|^p\right]\,du.
\end{align}
  Adding and subtracting terms we have $I_s= I_s^1+I_s^2+I_s^3+I_s^4$, where
\begin{align*}
&    I_s^1= diag(\bar X_s)\,h \Big(Y_s,\E\,\left[\langle X_s,1_N\rangle\,|\, (Y_s,\alpha_s)=(y,a)\right]\,\P_{Y_s,\alpha_s}(dy\,da),\alpha_s\Big),
\\&
I_s^2=
diag(X_s')\Big\{h \Big(Y_s,\E\,\left[\langle X_s,1_N\rangle\,|\, (Y_s,\alpha_s)=(y,a)\right]\,\P_{Y_s,\alpha_s}(dy\,da),\alpha_s\Big)
\\&
\qquad\qquad\qquad\qquad\qquad\qquad
-h \Big(Y_s',\E\,\left[\langle X_s,1_N\rangle\,|\, (Y_s,\alpha_s)=(y,a)\right]\,\P_{Y_s,\alpha_s}(dy\,da),\alpha_s\Big)
\Big\},
\\&
    I_s^3=
diag(X_s')\Big\{h \Big(Y_s',\E\,\left[\langle X_s,1_N\rangle\,|\, (Y_s,\alpha_s)=(y,a)\right]\,\P_{Y_s,\alpha_s}(dy\,da),\alpha_s\Big)
\\&
\qquad\qquad\qquad\qquad\qquad\qquad
-h \Big(Y_s',\E\,\left[\langle X_s',1_N\rangle\,|\, (Y_s,\alpha_s)=(y,a)\right]\,\P_{Y_s,\alpha_s}(dy\,da),\alpha_s\Big)\Big\},
\\&
    I_s^4=
diag(X_s')\Big\{
h\Big(Y_s',\E\,\left[\langle X_s',1_N\rangle\,|\, (Y_s,\alpha_s)=(y,a)\right]\,\P_{Y_s,\alpha_s}(dy\,da),\alpha_s\Big)
\\&
\qquad\qquad\qquad\qquad\qquad\qquad
-h \Big(Y_s',\E\,\left[\langle X_s',1_N\rangle\,|\, (Y_s',\alpha_s)=(y,a)\right]\,\P_{Y'_s,\alpha_s}(dy\,da),\alpha_s\Big)
\Big\}.
\end{align*}
We proceed to estimate these terms. 

\medskip

\emph{Estimate of $I^1$, $I^2$.}
Using the boundedness of $h$ and  \eqref{unifomega0}:
\begin{align*}
    |I_s^1|\le C\, |\bar X_s|, \qquad
    |I_s^2|\le C\, | X'_s|\,\omega_0(|Y_s-Y'_s|) .
\end{align*}
This implies
\begin{align*}
 \E\,\left[   |I_s^1|^p\right] \le C\,\E\,\left[ |\bar X_s|^p\right], \qquad
  \E\,\left[  |I_s^2|^p\right]\le C\, \E\,\left[| X'_s|^p \,\omega_0(|Y_s-Y'_s|)^p\right].
\end{align*}

\medskip

\emph{Estimate of  $I^3$.}
Using Assumption (\nameref{HphLipschitz}), with $\rho= \P_{Y_s,\alpha_s}$, we have
\begin{align*}
&|I_s^3|^2\le C\, | X'_s|^2\int_{\R^d
\times A}\Big|\E\,\left[\langle X_s,1_N\rangle\,|\,(Y_s,\alpha_s)=(y,a)\right]-
\E\,\left[\langle X_s',1_N\rangle\,|\, (Y_s,\alpha_s)=(y,a)\right]
\Big|^2\,\P_{Y_s,\alpha_s}(dy\,da)
\\&
\cdot\Big( 1+ \chi(\|\P_{Y_s,\alpha_s}\|_{\calw_r} )+\chi(\|\E\,\left[\langle X_s,1_N\rangle\,|\,(Y_s,\alpha_s)=\cdot \right]\P_{Y_s,\alpha_s}\|_{\calw_q})
\\&\qquad\qquad 
+\chi(\|\E\,\left[\langle X_s',1_N\rangle\,|\,(Y_s,\alpha_s)=\cdot  \right]\P_{Y_s,\alpha_s}\|_{\calw_q})
\Big).
\end{align*}
In this formula, the integral over $\R^d\times A$ is equal to
\begin{align*}
&
\int_{\R^d\times A}|\E\,\left[\langle \bar X_s,1_N\rangle\,|\,(Y_s,\alpha_s)=(y,a)\right]
|^2\,\P_{Y_s,\alpha_s}(dy\, da)
\\&\quad =\E\,\left[|\E\,\left[\langle \bar X_s,1_N\rangle\,|\,(Y_s,\alpha_s)\right]
|^2
\right]
       \le\E\,
\left[|\langle \bar X_s,1_N\rangle|^2\right]
\le \E\,
\left[| \bar X_s |^2\right].
\end{align*}
We proceed to estimate the remaining terms. We have
\begin{align}
\label{legge riflimitata}
\|\P_{Y_s,\alpha_s}\|_{\calw_r}\le C\,\left(\E\,\left[
|Y_s|^r+1 
\right] \right)^{1/r}\le C\,(1+\|\xi\|_r) \le K,
\end{align}
where the first inequality follows from the fact that the metric on the action space $A$ is bounded, the second one from \eqref{stapdoppia}.
Similarly, recalling that $q\le r(1-p^{-1})$
and using also the H\"older  inequality, \begin{align}\nonumber
\|\E\,\left[\langle X_s,1_N\rangle\,|\,(Y_s,\alpha_s)=\cdot \right]\,\P_{Y_s,\alpha_s}\|_{\calw_q}&\le C\,\left(\E\,\left[|X_s|(
|Y_s|^q+1 )
\right] \right)^{1/q}
\le C\, \|X_s\|_p^{1/q}(1+\|Y_s\|_r)
\\\label{leggidenslimituno}
& \le C\, \|\eta\|_p^{1/q}(1+\|\xi\|_r)\le K.
\end{align}
In a similar way,
\begin{align}\label{leggidenslimitdue}
\|\E\,\left[\langle X_s',1_N\rangle\,|\,(Y_s,\alpha_s)=\cdot \right]\,\P_{Y_s,\alpha_s}\|_{\calw_q}\le C\, \|\eta'\|_p^{1/q}(1+\|\xi\|_r)\le K,
\end{align}
and using \eqref{stapdoppia} once more we conclude that
\begin{align*}
 \E\,\left[   |I_s^3|^p\right]
 \le
K\, (\E\,\left[   | X_s'|^p\right]) (\E\,\left[   |\bar X_s|^2\right])^{p/2}
\le
K\, \| \eta'\|_p^p (\E\,\left[   |\bar X_s|^2\right])^{p/2}
 \le
K\, \E\,\left[   |\bar X_s|^p\right].
\end{align*}

\medskip

\emph{Estimate of  $I^4$.}
By \eqref{unifomegaq} we  have $ |I_s^4|\le C\, | X'_s|\,\omega_q(\calw_q(\mu,\nu))\,\left(1+\chi(\|\mu\|_{\calw_q})+\chi(\|\nu\|_{\calw_q})\right)$, where we set for short
\begin{align*}
\mu(dy\,da)=\E\,\left[\langle X_s',1_N\rangle\,|\, (Y_s,\alpha_s)=(y,a)\right]\,\P_{Y_s,\alpha_s}(dy\,da), 
\\
\nu(dy\,da)=\E\,\left[\langle X_s',1_N\rangle\,|\, (Y_s',\alpha_s)=(y,a)\right]\,\P_{Y_s',\alpha_s}(dy\,da).
\end{align*}
Proceeding as in \eqref{leggidenslimituno}-\eqref{leggidenslimitdue} we have $\|\mu\|_{\calw_q}+\|\nu\|_{\calw_q}\le K$. 
In order to estimate $\calw_q(\mu,\nu)$ we use the Kantorovich duality: we take $f_1,f_2:\R^d\times A\to\R$ bounded Lipschitz and satisfying
\[
f_1(y_1,a_1)+f_2(y_2,a_2)\le (|y_1-y_2|+d(a_1,a_2))^q, \qquad y_1,y_2\in\R^d,\;a_1,a_2\in A,
\]
and we compute
\begin{align*}
&\int_{\R^d\times A} f_1(y_1,a_1)\,\mu(dy_1\,da_1) +     \int_{\R^d\times A} f_2(y_2,a_2)\,\nu(dy_2\,da_2)
    \\&=
    \E\,\left[
f_1(Y_s,\alpha_s)\, \E\,\left[\langle X_s',1_N\rangle\,|\, Y_s,\alpha_s\right]\right]+
 \E\,\left[
f_2(Y'_s,\alpha_s)\, \E\,\left[\langle X_s',1_N\rangle\,|\, Y'_s,\alpha_s\right]\right]
 \\&=
\E\,\left[
(f_1(Y_s,\alpha_s)+f_2(Y_s',\alpha_s))\, \langle X_s',1_N\rangle\right]
\\&\le
\E\,\left[(
|Y_s-Y_s'|^q)\, \langle X_s',1_N\rangle\right]
\le
\E\,\left[|X_s'|\,
|Y_s-Y_s'|^q \right],
\end{align*}
Taking the supremum over $f_1,f_2$ it follows that
$\calw_q(\mu,\nu)^q\le 
\E\,\left[|X_s'|\,
|Y_s-Y_s'|^q\right]
$. By the H\"older  inequality, noting that $q p'\le r$, due to the condition $q\le r(1-p^{-1})$, we obtain
\begin{align*} 
\calw_q(\mu,\nu)\le \left(\E\,\left[|X_s'|
|Y_s-Y'_s|^q
\right] \right)^{1/q}
\le C\, \|X_s'\|_p^{1/q}\|Y_s-Y_s'\|_r.
\end{align*}
By
\eqref{stapdoppia} and \eqref{stabilstateeqY} we have
\begin{align}\label{stimaWq} 
\calw_q(\mu,\nu)
 \le K\, \|\eta'\|_p^{1/q}\|\xi-\xi'\|_r\le K\,  \|\xi-\xi'\|_r.
\end{align}
We  finally obtain
\begin{align*}
\E\,\left[|I_s^4|^p\right]\le K\,\E\,\left[ |X'_s|^p\right]\,\omega_q\left(K\, \|\xi-\xi'\|_r
\right)^p
\le K\, \omega_q\left(K\, \|\xi-\xi'\|_r
\right)^p.
\end{align*}

\medskip

Replacing the obtained estimates in \eqref{BDGsuXbar} we have
\begin{align*}
\E\,\sup_{u\in [t,s]}|\bar X_u|^p& 
\le K\, \E\left[|\bar\eta|^p\right]
+K\,  \int_t^s \E\,\left[| \bar X_u|^p\right]\,du
+
 K\,  \int_t^s \E\,\left[| X'_u|^p \,\omega_0(|Y_u-Y'_u|)^p\right]\,du
\\&
\quad + K\,    
\omega_q\left(K\, \|\xi-\xi'\|_r
\right)^p
\end{align*}
and 
\eqref{stabilstateeq} follows from  the Gronwall lemma.

To prove the final convergence result 
we apply 
 \eqref{stabilstateeqY}-\eqref{stabilstateeq} to the pairs 
$(X^n,Y^n)$ and  $(X,Y)$. We first obtain 
$\E\, \Big[\sup_{s\in [t,T] }|Y_s^n-Y_s|^r\Big]\le C\, \|\xi^n-\xi\|_r^r\to 0$.
Next we have
\begin{align*}
\E\, \Big[\sup_{s\in [t,T] }| X_s^n-X_s|^p\Big]&\le K\, \bigg\{\|\eta^n-\eta\|_p^p
+ \omega_q\left(K\, \|\xi^n -\xi\|_r
\right)^p +
\int_t^T \E\,\left[| X_s|^p \,\omega_0(|Y^n_s-Y_s|)^p\right]\,ds
\bigg\}
\end{align*}
and it remains to prove that the last integral tends to zero as well. Suppose on the contrary that for some $\delta>0$ and some subsequence $\{n_k\}$ we had
\begin{align*}
\int_t^T \E\,\left[| X_s|^p \,\omega_0(|Y^{n_k}_s-Y_s|)^p\right]\,ds\ge\delta.
\end{align*}
Then we might extract a subsubsequence $\{n_{k_j}\}$ such that 
$Y_s^{n_{k_j}}\to Y_s$ a.s., uniformly in $s$,  for $j\to\infty$. Since $\omega_0$ is continuous at zero and bounded, and since 
$\E\,
\int_t^T |X_s|^p\,ds<\infty$, we would get a contradiction with the dominated convergence theorem.
\qed

\medskip

\paragraph{Proof of Proposition \ref{stabrewardJ}.} 

We will again use \eqref{stapdoppia} several times, as well as  the inequality 
$ \|X_s'\|_{p_1} \le C\, \|\eta'\|_{p_1}   $ which also follows from Theorem \ref{existstateeqprovv} (formula \eqref{stap}, with $p_1$ instead of $p$).

We first prove 
\eqref{Junifcontuno},
noting that $Y=Y'$, since they have a common starting point  $\xi$.
Adding and subtracting terms we have
\begin{align*}
J(t,\eta,\xi,\alpha)-
J(t,\eta',\xi,\alpha)=I^1+I^2+I^3+I^4,
\end{align*}
where
\begin{align*}
I^1&=
\E\bigg[\int_t^T \langle X_s -X'_s,f\Big(Y_s, \E\,\left[\langle X_s,1_N\rangle\,|\, (Y_s,\alpha_s)=(y,a)\right]\,\P_{Y_s,\alpha_s}(dy\,da),\alpha_s\Big)\rangle\,ds \bigg],
\\ 
I^2&=
\E\bigg[\int_t^T \langle X_s' ,f\Big(Y_s, \E\,\left[\langle X_s,1_N\rangle\,|\, (Y_s,\alpha_s)=(y,a)\right]\,\P_{Y_s,\alpha_s}(dy\,da),\alpha_s\Big)
\\ &\quad\quad\quad \quad -
 f\Big(Y_s, \E\,\left[\langle X_s',1_N\rangle\,|\, (Y_s,\alpha_s)=(y,a)\right]\,\P_{Y_s,\alpha_s}(dy\,da),\alpha_s\Big)\rangle\,ds \bigg],
\\I^3 &= \langle X_T -X'_T, g\Big(Y_T, \E\,\left[\langle X_T,1_N\rangle\,|\, (Y_T,\alpha_T)=(y,a)\right]\,\P_{Y_T,\alpha_T}(dy\,da)\Big)\rangle,
\\I^4&=  \langle X_T' , g\Big(Y_T, \E\,\left[\langle X_T,1_N\rangle\,|\, (Y_T,\alpha_T)=(y,a)\right]\,\P_{Y_T,\alpha_T}(dy\,da)\Big)
\\&\quad\quad\quad\quad 
-g\Big(Y_T, \E\,\left[\langle X_s',1_N\rangle\,|\, (Y_T,\alpha_T)=(y,a)\right]\,\P_{Y_T,\alpha_T}(dy\,da)\Big)
\rangle .
\end{align*}
Denoting by $p'$ the exponent conjugate to $p$ we have
\begin{align*}
    I^1\le
\left(\E \int_t^T | X_s -X'_s|^p\,ds\right)^{1/p}\left( \E\int_t^T\Big|f\Big(Y_s, \mu,\alpha_s\Big)\Big|^{p'}\,ds \right)^{1/p'},
\end{align*}
where we set for short $\mu(dy\,da)=\E\,\left[\langle X_s,1_N\rangle\,|\, (Y_s,\alpha_s)=(y,a)\right]\,\P_{Y_s,\alpha_s}(dy\,da)$. 
Proceeding as in \eqref{leggidenslimituno}-\eqref{leggidenslimitdue} we have $\|\mu\|_{\calw_q}\le K$. 
Using the growth condition \eqref{growthfg},
\begin{align*}
 \E\Big[   \Big|f\Big(Y_s, \mu,\alpha_s\Big)\Big|^{p'}\Big]\le C\,\E\left( 1+ |Y_s|^{\ell p'} + \chi(\|\mu\|_{\calw_q})^{p'}\right) \le K \,(1+ \|Y_s\|_{\ell p'}^{\ell p'})
    \le K \,(1+ \|Y_s\|_{r}^{\ell p'})
\end{align*}
because we have $\ell p'\le r$, due to the condition $\ell\le r(1-p^{-1})$. 
It follows that
\begin{align*}
    I^1&\le K
\left( \int_t^T \| X_s -X'_s\|_p^p\,ds\right)^{1/p}\left( \int_t^T(1+\|Y_s\|_{r}^{\ell p'})\,ds \right)^{1/p'} \le K\, \|\eta-\eta'\|_p,
\end{align*}
where, in the last inequality,  we have used  \eqref{stabilstateeq} (with $\xi=\xi'$, $Y=Y'$)
and    \eqref{stapdoppia}.

Next we estimate $I^2$.
We set for short
\begin{align*}
\rho(dy\,da)&= \P_{Y_s,\alpha_s}(dy\,da), 
\\
\nu_1(dy\,da)&=\E\,\left[\langle X_s,1_N\rangle\,|\, (Y_s,\alpha_s)=(y,a)\right]\,\P_{Y_s,\alpha_s}(dy\,da), 
\\
\nu_2(dy\,da)&=\E\,\left[\langle X_s',1_N\rangle\,|\, (Y_s,\alpha_s)=(y,a)\right]\,\P_{Y_s,\alpha_s}(dy\,da).
\end{align*}
 Using \eqref{liponfeg} we have
\begin{align*}
I^2&\le 
\E\bigg[\int_t^T | X_s'| \,|f(Y_s,\nu_1,\alpha_s) -
 f(Y_s, \nu_2,\alpha_s )|\,ds \bigg]
 \\& 
 \le C\,\E\bigg[ \int_t^T\! | X'_s|\int_{\R^d
\times A}\Big|\E\left[\langle X_s,1_N\rangle\,|\,(Y_s,\alpha_s)=(y,a)\right]-
\E\left[\langle X_s',1_N\rangle\,|\, (Y_s,\alpha_s)=(y,a)\right]
\Big|\,\P_{Y_s,\alpha_s}(dy\,da)
\\&\quad 
\cdot\Big( 1+ \chi(\|\rho\|_{\calw_r}) +\chi(\|\nu_1\|_{\calw_q})+\chi(\|\nu_2\|_{\calw_q} )+|Y_s|^\ell
\Big)\,ds\bigg].
\end{align*}
Proceeding as for 
\eqref{legge riflimitata},
\eqref{leggidenslimituno},
\eqref{leggidenslimitdue}  we have
$\|\rho\|_{\calw_r}+\|\nu_1\|_{\calw_q}+\|\nu_2\|_{\calw_q} \le K$.
The integral over $\R^d\times A$ is equal to
\begin{align*}
&
\int_{\R^d\times A}|\E\,\left[\langle  X_s-X'_s,1_N\rangle\,|\,(Y_s,\alpha_s)=(y,a)\right]
|\,\P_{Y_s,\alpha_s}(dy\, da)
\\&\quad =\E\,\left[|\E\,\left[\langle   X_s-X_s',1_N\rangle\,|\,(Y_s,\alpha_s)\right]|
\right]
       \le\E\,
\left[|\langle  X_s-X_s',1_N\rangle|\right]
\le \E\,
\left[|   X_s-X_s' |\right].
\end{align*}
Using  \eqref{stabilstateeq}, with $\xi=\xi'$ and $Y=Y'$,
we obtain $\E\,
\left[|   X_s-X_s' |\right]\le K\, \|   \eta-\eta' \|_p$. It follows that
\begin{align*}
   |I^2|
 \le
K \, \|   \eta-\eta' \|_p\int_t^T \E\,[| X_s'|\,(1+|Y_s|^\ell)]   \,ds.
\end{align*}
Noting again that  $\ell p'\le r$,
\begin{align*}
\int_t^T \E\,[| X_s'|\,(1+|Y_s|^\ell)]   \,ds
\le  K\,
\int_t^T \| X_s'\|_p \,( 1+ \|Y_s\|^\ell_{\ell p'})\,ds
\le  K\,
\int_t^T \| X_s'\|_p \,( 1+ \|Y_s\|^\ell_{r})\,ds,
\end{align*}
and from \eqref{stapdoppia} we obtain
\begin{align*}   |I^2|
 \le K
\, \, \|   \eta-\eta' \|_p\,\| \eta'\|_p \,(1+\|\xi\|_r)\le
K\, \|   \eta-\eta' \|_p.
\end{align*}
Similar estimates on $I^3$ and $I^4$ lead to 
\eqref{Junifcontuno}.

\medskip

Now we assume 
\(\|\eta'\|_{p_1}
\le L\) 
for some $p_1>p$ and we prove \eqref{Junifcontdue}.
From \eqref{stabilstateeq}, setting $\eta=\eta'$ and using the  Hölder  inequality we obtain
\begin{align*}
\E\, \Big[\sup_{s\in [t,T] }| X_s-X_s'|^p\Big]\le K \,\bigg\{ 
\omega_q\left(K\, \|\xi-\xi'\|_r
\right)^p +
\int_t^T \| X'_s\|^p_{p_1} \Big(\E\,\left[\omega_0(|Y_s-Y'_s|)^{\frac{pp_1}{p_1-p}}\right]\Big)^{\frac{p_1-p}{p_1}}\,ds
\bigg\}
\end{align*}
and since 
$\| X'_s\|_{p_1}\le C\, \|\eta'\|_{p_1}
\le K$ by \eqref{stapdoppia},
we have 
\begin{align}\label{stimaxperxi}
\E\, \Big[\sup_{s\in [t,T] }| X_s-X_s'|^p\Big]\le K
 \,\bigg\{ 
\omega_q\left(K\, \|\xi-\xi'\|_r\right)^p +\Big(
\E\,\Big[\omega_0(\sup_{s\in [t,T] }|Y_s-Y'_s|)^{\frac{pp_1}{p_1-p}}\Big]\Big)^{\frac{p_1-p}{p_1}}
\bigg\}.
\end{align}
We will use this inequality, together with 
\eqref{stabilstateeqY}, to conclude the proof.
Adding and subtracting  we obtain
$
J(t,\eta,\xi,\alpha)-
J(t,\eta,\xi',\alpha)=\sum_{j=1}^8I^j$, 
where
\begin{align*}
I^1&=
\E\bigg[\int_t^T \langle X_s -X'_s,f\Big(Y_s, \E\,\left[\langle X_s,1_N\rangle\,|\, (Y_s,\alpha_s)=(y,a)\right]\,\P_{Y_s,\alpha_s}(dy\,da),\alpha_s\Big)\rangle\,ds \bigg],
\\ 
I^2&=
\E\bigg[\int_t^T \langle X_s' ,f\Big(Y_s, \E\,\left[\langle X_s,1_N\rangle\,|\, (Y_s,\alpha_s)=(y,a)\right]\,\P_{Y_s,\alpha_s}(dy\,da),\alpha_s\Big)
\\ &\quad\quad\quad \quad -
 f\Big(Y'_s, \E\,\left[\langle X_s,1_N\rangle\,|\, (Y_s,\alpha_s)=(y,a)\right]\,\P_{Y_s,\alpha_s}(dy\,da),\alpha_s\Big)\rangle\,ds \bigg],
\\ 
I^3&=
\E\bigg[\int_t^T \langle X_s' ,f\Big(Y'_s, \E\,\left[\langle X_s,1_N\rangle\,|\, (Y_s,\alpha_s)=(y,a)\right]\,\P_{Y_s,\alpha_s}(dy\,da),\alpha_s\Big)
\\ &\quad\quad\quad \quad -
 f\Big(Y'_s, \E\,\left[\langle X_s',1_N\rangle\,|\, (Y_s,\alpha_s)=(y,a)\right]\,\P_{Y_s,\alpha_s}(dy\,da),\alpha_s\Big)\rangle\,ds \bigg],
 \\ 
I^4&=
\E\bigg[\int_t^T \langle X_s' ,f\Big(Y'_s, \E\,\left[\langle X_s',1_N\rangle\,|\, (Y_s,\alpha_s)=(y,a)\right]\,\P_{Y_s,\alpha_s}(dy\,da),\alpha_s\Big)
\\ &\quad\quad\quad \quad -
 f\Big(Y'_s, \E\,\left[\langle X_s',1_N\rangle\,|\, (Y_s',\alpha_s)=(y,a)\right]\,\P_{Y_s',\alpha_s}(dy\,da),\alpha_s\Big)\rangle\,ds \bigg],
\\I^5 &= \langle X_T -X'_T, g\Big(Y_T, \E\,\left[\langle X_T,1_N\rangle\,|\, (Y_T,\alpha_T)=(y,a)\right]\,\P_{Y_T,\alpha_T}(dy\,da)\Big)\rangle,
\\I^6&=  \langle X_T' , g\Big(Y_T, \E\,\left[\langle X_T,1_N\rangle\,|\, (Y_T,\alpha_T)=(y,a)\right]\,\P_{Y_T,\alpha_T}(dy\,da)\Big)
\\&\quad\quad\quad\quad 
-g\Big(Y_T', \E\,\left[\langle X_T,1_N\rangle\,|\, (Y_T,\alpha_T)=(y,a)\right]\,\P_{Y_T,\alpha_T}(dy\,da)\Big)
\rangle ,
\\I^7&=  \langle X_T' , g\Big(Y_T', \E\,\left[\langle X_T,1_N\rangle\,|\, (Y_T,\alpha_T)=(y,a)\right]\,\P_{Y_T,\alpha_T}(dy\,da)\Big)
\\&\quad\quad\quad\quad 
-g\Big(Y_T', \E\,\left[\langle X_T',1_N\rangle\,|\, (Y_T,\alpha_T)=(y,a)\right]\,\P_{Y_T,\alpha_T}(dy\,da)\Big)
\rangle ,
\\I^8&=  \langle X_T' , g\Big(Y_T', \E\,\left[\langle X_T',1_N\rangle\,|\, (Y_T,\alpha_T)=(y,a)\right]\,\P_{Y_T,\alpha_T}(dy\,da)\Big)
\\&\quad\quad\quad\quad 
-g\Big(Y_T', \E\,\left[\langle X_T',1_N\rangle\,|\, (Y_T',\alpha_T)=(y,a)\right]\,\P_{Y_T',\alpha_T}(dy\,da)\Big)
\rangle .
\end{align*}
These terms are  estimated as follows:
\begin{align*}
|I^1|+|I^3|+|I^5|+|I^7|\le K\,\Big( \E\, \Big[\sup_{s\in [t,T] }| X_s-X_s'|^p\Big]\Big)^{1/p},
\end{align*}
\begin{align*}
|I^4| \le 
K\,\omega_q^f(\|\xi-\xi'\|_r), \qquad
|I^8|\le 
K\,\omega_q^g(\|\xi-\xi'\|_r),
\end{align*}
\begin{align*}
|I^2| &\le 
K\, \Big(\E  \,\Big[\omega_{0}^f(\sup_{s\in[t,T]}|Y_s-Y_s'|)^{\frac{ r}{r(1+1/p_1)-\ell }} \Big]\Big)^{\frac{r(1+1/p_1)-\ell }{ r}},
 \\
|I^6|&\le 
 K\, \Big(\E  \,\Big[\omega_{0}^g(\sup_{s\in[t,T]}|Y_s-Y_s'|)^{\frac{ r}{r(1+1/p_1)-\ell }} \Big]\Big)^{\frac{r(1+1/p_1)-\ell }{ r}},
\end{align*}
and taking into account \eqref{stimaxperxi} we arrive at \eqref{Junifcontdue}.

We only show the estimates on $I^2$ and $I^4$, since the other terms are treated in a similar way. 
We set for short
\begin{align*}
\nu(dy\,da)&=\E\,\left[\langle X_s,1_N\rangle\,|\, (Y_s,\alpha_s)=(y,a)\right]\,\P_{Y_s,\alpha_s}(dy\,da).
\end{align*}
Using \eqref{unifomega0f}  we obtain
\begin{align*}
    |I^2|\le 
\E\,\bigg[\int_t^T | X_s'| \,\omega_{0}^f(|Y_s-Y_s'|)\,
(1+|Y_s|^\ell+|Y_s'|^\ell + \chi(\|\nu\|_{\calw_q})) ds\bigg].
\end{align*}
Proceeding as for 
\eqref{leggidenslimituno},
\eqref{leggidenslimitdue}  we have
$\|\nu\|_{\calw_q} \le K$, 
and denoting by $p_1'$ the exponent conjugate to $p_1$ we obtain
\begin{align*}
    |I^2|\le K\,
\left(\int_t^T \E\,[| X_s'|^{p_1} ]\,ds\right)^{1/p_1} \left( \E\int_t^T \omega_{0}^f(|Y_s-Y_s'|)^{p_1'}\,
(1+|Y_s|+|Y_s'|)^{\ell p_1'} ds\right)^{1/p_1'}.
\end{align*}
As recalled before,  we have $\E\,| X_s'|^{p_1}=\|X_s'\|_{p_1}^{p_1}\le C\,\|\eta'\|_{p_1}^{p_1}\le K$. Since $\ell\le r(1-1/p)$   and   $p_1>p$ we have $\ell p_1'< r$ and we have
\begin{align*}
    |I^2|& \le 
K\, \left( \E\int_t^T  
(1+|Y_s|+|Y_s'|)^{r} ds\right)^{\ell/r}
\left( \E\int_t^T \omega_{0}^f(|Y_s-Y_s'|)^{\frac{ p_1'r}{r-\ell p_1'}} ds\right)^{\frac{r-\ell p_1'}{rp_1'}}.
\end{align*}
Since, by \eqref{stapdoppia}, $\|Y_s\|_r+\|Y_s'\|_r\le C\,( \|\xi\|_r+\|\xi'\|_r)\le K$,  we   obtain
\begin{align*}
    |I^2|& \le 
K\, \Big(\E  \,\Big[\omega_{0}^f(\sup_{s\in[t,T]}|Y_s-Y_s'|)^{\frac{ r}{r(1-1/p_1)-\ell }} \Big]\Big)^{\frac{r(1-1/p_1)-\ell }{ r}}.
\end{align*}

Next we estimate $I^4$.
We set for short
\begin{align*}
\mu(dy\,da)&=\E\,\left[\langle X'_s,1_N\rangle\,|\, (Y_s,\alpha_s)=(y,a)\right]\,\P_{Y_s,\alpha_s}(dy\,da), 
\\
\nu(dy\,da)&=\E\,\left[\langle X_s',1_N\rangle\,|\, (Y'_s,\alpha_s)=(y,a)\right]\,\P_{Y'_s,\alpha_s}(dy\,da),
\end{align*}
and we have
\begin{align*}
    I^4&\le 
\E\,\bigg[\int_t^T | X_s'| \,\Big|f\Big(Y'_s, \mu,\alpha_s\Big) -
 f\Big(Y'_s, \nu,\alpha_s\Big)\Big|\,ds \bigg].
 \end{align*}
Using \eqref{unifomegaqf} we obtain 
\begin{align*}
\Big|f\Big(Y'_s, \mu,\alpha_s\Big) -
 f\Big(Y'_s, \nu,\alpha_s\Big)\Big|
 \le \omega_q^f(\calw_q(\nu,\nu'))\,\Big(1+\chi(\|\nu\|_{\calw_q})+\chi(\|\nu'\|_{\calw_q})+|Y'_s|^\ell\Big).
\end{align*}
Proceeding as in \eqref{leggidenslimituno}-\eqref{leggidenslimitdue} we have $\|\mu\|_{\calw_q}+\|\nu\|_{\calw_q}\le K$. 
By the same passages that led to  \eqref{stimaWq} we obtain $\calw_q(\mu,\nu)\le K\,  \|\xi-\xi'\|_r$. It follows that
\begin{align*}
|I^4|\le K\,\omega_q^f(K\,\|\xi-\xi'\|_r)\,\E\,\bigg[
 \int_t^T  | X_s'|(1+|Y'_s|^\ell)
 \,ds \bigg].
\end{align*}
Since  $\ell p'\le r$,
\begin{align*}
\int_t^T \E\,[| X_s'|\,(1+|Y_s|^\ell)]   \,ds
\le  K\,
\int_t^T \| X_s'\|_p \,( 1+ \|Y_s'\|^\ell_{\ell p'})\,ds
\le  K\,
\int_t^T \| X_s'\|_p \,( 1+ \|Y_s'\|^\ell_{r})\,ds,
\end{align*}
and from \eqref{stapdoppia}
we conclude that
$|I^4|\le K\,\omega_q^f(K\,\|\xi-\xi'\|_r)$.
\qed

\medskip

\paragraph{Proof of Lemma \ref{Jcont}.}  

\emph{Proof of $(i)$.}
For any integer $m\ge 1$
let us define a    truncation function $T_m$ setting $T_m(u)=(u\wedge m)\vee (-m)$ for $u\in\R$. Let $T_m(\eta)$ and $T_m(\eta_n)$ be defined componentwise. Then
\begin{align*}
\sup_{\alpha\in\cala,\,t\in[0,T]}|J(t,\eta_n,\xi_n,\alpha)-
J(t,\eta,\xi,\alpha)|\le I^1_{nm}+I^2_{nm}+I^3_{nm}+I^4_n,
\end{align*}
where
\begin{align*}
I^1_{nm}&=\sup_{\alpha\in\cala,\,t\in[0,T]}|J(t,\eta_n,\xi_n,\alpha)-
J(t,T_m(\eta_n),\xi_n,\alpha)|
\\ 
I^2_{nm}&=\sup_{\alpha\in\cala,\,t\in[0,T]}|
J(t,T_m(\eta_n),\xi_n,\alpha)-J(t,T_m(\eta_n),\xi,\alpha)|
\\I^3_{nm} &= \sup_{\alpha\in\cala,\,t\in[0,T]}|
J(t,T_m(\eta_n),\xi,\alpha)-J(t,\eta_n,\xi,\alpha)|
\\I^4_n&= \sup_{\alpha\in\cala,\,t\in[0,T]}|J(t,\eta_n,\xi,\alpha)-J(t,\eta,\xi,\alpha)|.
\end{align*}
We first apply the inequality \eqref{Junifcontuno} and we obtain
\begin{align*}
I^1_{nm}\le K\, \|\eta_n-T_m(\eta_n)\|_p,
\qquad
I^3_{nm}\le K\, \|\eta_n-T_m(\eta_n)\|_p,
\end{align*}
so that  
\begin{align}\label{UIquantit}
\limsup_n\,(I^1_{nm}+
I^3_{nm})\le K\,\limsup_n \|\eta_n-T_m(\eta_n)\|_p=
 K\, \limsup_n \Big(\E [|\eta_n|^p\,1_{|\eta_n|>m}]\Big)^{1/p}.
 \end{align} 
This tends to $0$ as $m\to\infty$: indeed, since $\{\eta_n\}$ is convergent in $L^p$, the sequence $\{|\eta_n|^p\}$ is uniformly integrable.
Given $\epsilon>0$ we fix $m$ large enough to have $\limsup_{n}(I^1_{nm}+
I^3_{nm})\le \epsilon$. It follows that
\begin{align*}
\limsup_{n\to\infty}
\sup_{\alpha\in\cala,\,t\in[0,T]}
|J(t,\eta_n,\xi_n,\alpha)-
J(t,\eta,\xi,\alpha)|\le \epsilon + \limsup_{n\to\infty} \,(I^2_{nm}+I^4_n).
\end{align*}
Applying again \eqref{Junifcontuno} we have
$I^4_{n}\le K\, \|\eta_n-\eta\|_p \to0$ as $n\to\infty$. To conclude the proof of point $(i)$ it is sufficient to prove that, for fixed $m\ge1$, we have
\begin{align}\label{Inmunif}
 \lim_{n\to\infty}I^2_{nm}=0.
\end{align}
Since     $T_m(\widehat\eta)$ is bounded (by a constant that depends on $m$) we can apply the inequality \eqref{Junifcontdue}: there exists a constant $K=K_m$, depending on $m$ but independent of $n,\alpha,t$, such that
\begin{align}
\nonumber
&
|J(t,T_m(\eta_n),\xi_n,\alpha)- J(t,T_m(\eta_n),\xi,\alpha)|\\\nonumber
& \qquad
\le K_m\bigg\{  \omega_q\left(K\, \|\xi_n-\xi\|_r\right) +\omega^f_q\left(K\, \|\xi_n-\xi\|_r\right)+\omega_q^g\left(K\, \|\xi_n-\xi\|_r\right)
\\&\qquad\nonumber
+\Big(\E\,\Big[\omega_0\Big(\sup_{s\in [t,T] }|Y^{t,\xi_n}_s- Y^{t,\xi}_s|\Big)^{\beta}\Big]\Big)^{1/\beta}
 +  \Big( \E  \,\Big[\omega_{0}^f\Big(\sup_{s\in[t,T]}|Y^{t,\xi_n}_s-Y^{t,\xi}_s|\Big)^{\gamma} \Big]\Big)^{1/\gamma} 
 \\ 
& \qquad
+  \Big( \E  \,\Big[\omega_{0}^g\Big(\sup_{s\in[t,T]}|Y^{t,\xi_n}_s-Y^{t,\xi}_s|\Big)^{\gamma} \Big]\Big)^{1/\gamma} 
\bigg\},\label{unifmn}
\end{align}
 for every $\alpha\in\cala$ and $t\in[0,T]$,
where $\beta>0 $, $\gamma>0$ are suitable constants. Note that the right-hand side of \eqref{unifmn} does not depend on $\alpha$. It remains to check that it tends to $0$, for fixed $m$, when $n\to\infty$, uniformly for $t\in[0,T]$. We know that $\|\xi_n-\xi\|_r\to0$, so it is clear that the first three terms in the right-hand side of \eqref{unifmn} vanish. 
Suppose that we have,  for some constant $\delta>0$, some subsequence ${n_k}$, and some $t_{k}\in[0,T]$,
\begin{align*}
\E  \,\Big[\omega_{0}\Big(\sup_{s\in[t_{k},T]}|Y^{t_k,\xi_{n_k}}_s- Y^{t_k,\xi}_s|\Big)^{\beta} \Big]\ge \delta.
\end{align*}
We note that
by \eqref{stabilstateeqY} we have
\begin{align}
\E\, \Big[\sup_{s\in [t_{k},T] }|Y^{t_k,\xi_{n_k}}_s-  Y^{t_k,\xi}_s|^r\Big]\le K\, \|\xi_{n_k}-\xi\|_r^r\to0, \qquad k\to\infty.
\end{align}
Then we can extract a subsubsequence $\{{k_j}\}$ such that 
\begin{align*}
\sup_{s\in [t_{{k_j}},T] }|Y^{t_{k_j},\xi_{n_{k_j}}}_s-  Y^{t_{k_j},\xi}_s|\to0, \qquad \P-a.s. 
\end{align*}
as $j\to\infty$. Since $\omega_0$ is continuous at zero and bounded we would get a contradiction with the dominated convergence theorem. This proves that $\E\,\Big[\omega_0\Big(\sup_{s\in [t,T] }|Y^{t,\xi_n}_s- Y^{t,\xi}_s|\Big)^{\beta}\Big]\to 0$ uniformly in $t$. The last two summands in the right-hand side of \eqref{unifmn} also tend to $0$ by similar arguments and this completes the proof of \eqref{Inmunif} and of point $(i)$. 

\medskip

\emph{Proof of $(ii)$.} We construct $T_m$ as before and we write
\begin{align*}
\sup_{\alpha\in\cala,\,\beta}|J(t_n,\eta_n,\xi_n,\alpha)-
J(t_n,\eta,\xi,\alpha)|\le I^1_{nm}+I^2_{nm}+I^3_{nm}+I^4_n,
\end{align*}
where
\begin{align*}
I^1_{nm}&=\sup_{\alpha\in\cala,\,\beta}|J(t_n,\eta_n,\xi_n,\alpha)-
J(t_n,T_m(\eta_n),\xi_n,\alpha)|
\\ 
I^2_{nm}&=\sup_{\alpha\in\cala,\,\beta}|
J(t_n,T_m(\eta_n),\xi_n,\alpha)-J(t_n,T_m(\eta_n),\xi,\alpha)|
\\I^3_{nm} &= \sup_{\alpha\in\cala,\,\beta}|
J(t_n,T_m(\eta_n),\xi,\alpha)-J(t_n,\eta_n,\xi,\alpha)|
\\I^4_n&= \sup_{\alpha\in\cala,\,\beta}|J(t_n,\eta_n,\xi,\alpha)-J(t_n,\eta,\xi,\alpha)|.
\end{align*}
The rest of the  proof is almost entirely the same as for point $(i)$.  The only required change  is the following: instead of showing that the right-hand side of \eqref{UIquantit} tends to $0$ as $m\to\infty$ now we need the stronger condition
\begin{align}\label{unifUIter}
\limsup_n \;\sup_\beta     \,\E [|\eta^\beta_n|^p\,1_{|\eta^\beta_n|>m}]\to0,\qquad m\to\infty.
\end{align}
This follows directly from \eqref{etaennebeta}. Indeed, first note that
\begin{align*}
\E \,[|\eta_n^\beta|^p\,1_{|\eta^\beta_n|>m}]
\le 2^{p-1}\E\, [|\eta^\beta_n-\eta |^p] +
2^{p-1}\E\, [|\eta |^p\,1_{|\eta^\beta_n|>m}].
\end{align*}
Given $\epsilon>0$, by \eqref{etaennebeta} we can choose $N$ so large that the first summand is $\le\epsilon$ for $n>N$ and for all $\beta$. Next  we choose $\delta>0$ (depending on $\eta,\epsilon,p$) such that  $\E\, [|\eta |^p\,1_{A}]\le \epsilon$ whenever $\P(A)\le \delta$. Then, taking $A=\{|\eta^\beta_n|>m\}$, we have
\begin{align*}
\P(A)^{1/p}=  \P(|\eta^\beta_n|>m)^{1/p} \le \frac{\|\eta^\beta_n\|_p}{m}  \le \frac{\|\eta^\beta_n-\eta\|_p}{m} +\frac{\|\eta\|_p}{m} \le \frac{\sup_{n,\beta}\|\eta^\beta_n-\eta\|_p}{m}    +\frac{\|\eta\|_p}{m} 
\end{align*}
so that we can find $M$ such that this is  $\le\delta^{1/p}$ for $m>M$. It follows that 
\begin{align*}
\sup_\beta\,\E \,[|\eta_n^\beta|^p\,1_{|\eta^\beta_n|>m}]
\le \epsilon  +
2^{p-1}\epsilon, \qquad \text{for } n>N,\;m>M.
\end{align*}
This proves \eqref{unifUIter} and  concludes the proof of point $(ii)$ and of the Lemma. \qed

\bigskip

\paragraph{Proof of Theorem \ref{thsolvisc}.}
The value function $v$ is continuous by Theorem \ref{Continuityofvinciandtime}, and it satisfies the terminal condition.

{\it Subsolution property.} Suppose that $(t,\mu)\in [0,T)\times \cald_{p,r}$ 
and for some $\varphi\in \tilde C^{1,2}([0,T]\times \cald_{p,r})$  and $(t,\mu)\in [0,T)\times \cald_{p,r}$ we have
\begin{align}
    \label{subsolv}
(v-\varphi)(t,\mu)  =  \max_{[0,T]\times \cald_{p,r}}(v-\varphi). 
\end{align} 
Take  $(\eta,\xi)$ to be any pair of $\calf_0$-measurable random variables in $\R^N\times \R^d$ such that 
$\mu=\P_{\eta,\xi}$, and fix $\epsilon>0$. 
From the dynamic programming principle (Proposition \ref{propDPPbis}), setting $s=t+\epsilon$ in formula 
\eqref{DPPv}, we deduce that there exist admissible controls $\alpha^\epsilon$ such that
\begin{align*}
-\epsilon\le \frac{1}{\epsilon}\,[
v(t+\epsilon, \P_{X_{t+\epsilon}^\epsilon,Y_{t+\epsilon}})-
v(t,\mu)]+\frac{1}{\epsilon}
\int_t^{t+\epsilon}\E\Big[\langle X^\epsilon_s ,f(Y_s, \Gamma(\P_{X_s^\epsilon,Y_s,\alpha^\epsilon_s}),\alpha^\epsilon_s)\rangle\Big] \,ds ,
\end{align*} 
where we set for simplicity 
$Y_s=Y_s^{t,\xi}$ and $X_s^\epsilon=X_r^{t,\eta,\xi,\alpha^\epsilon}$ for $s\in [t,T]$. 
By \eqref{subsolv}, the same inequality holds with $v$ replaced by $\varphi$. Applying the It\^o formula \eqref{itosolution} we obtain
\begin{align}\label{versosottosol}
-\epsilon & \le \;  \frac{1}{\epsilon}
\int_t^{t+\epsilon}\Big\{ \partial_t\varphi(s, \P_{X^\epsilon_{s},Y_{s}})+
 \E\Big[
\call^{\alpha^\epsilon}\varphi(s,X^\epsilon_s,Y_s, \P_{X_{s}^\epsilon,Y_{s},\alpha^\epsilon_s}) \\
& \qquad \qquad + \;  \langle X^\epsilon_s ,f(Y_s, \Gamma(\P_{X_s^\epsilon,Y_s,\alpha^\epsilon}),\alpha_s^\epsilon)\rangle\Big] \Big\}\,ds .
\end{align} 
Let us define
\begin{align*}
\rho_1(\epsilon)&= 
\frac{1}{\epsilon}
\int_t^{t+\epsilon}\partial_t\varphi(s, \P_{X^\epsilon_{s},Y_{s}})\,ds - 
\partial_t\varphi(t,\mu),
\\
\rho_2(\epsilon)&=
\frac{1}{\epsilon}
\int_t^{t+\epsilon}
 \E\Big[ \langle X^\epsilon_s ,f(Y_s, \Gamma(\P_{X_s^\epsilon,Y_s,\alpha^\epsilon}),\alpha_s^\epsilon)\rangle - \langle \eta,f(\xi, \Gamma(\P_{X_s^\epsilon,Y_s,\alpha^\epsilon}),\alpha_s^\epsilon)\rangle\Big] \,ds ,
 \\
\rho_3(\epsilon)&=
\frac{1}{\epsilon}
\int_t^{t+\epsilon}
 \E\Big[  \langle \eta,f(\xi, \Gamma(\P_{X_s^\epsilon,Y_s,\alpha^\epsilon}),\alpha_s^\epsilon)\rangle- \langle \eta,f(\xi, \Gamma(\P_{\eta,\xi,\alpha^\epsilon}),\alpha_s^\epsilon)\rangle\Big] \,ds ,
  \\
\rho_4(\epsilon)&=
\frac{1}{\epsilon}
\int_t^{t+\epsilon}
 \E\Big[ 
\call^{\alpha^\epsilon}\varphi(s,X^\epsilon_s,Y_s, \P_{X_{s}^\epsilon,Y_{s},\alpha^\epsilon_s})-
\call^{\alpha^\epsilon}\varphi(t,\eta,\xi, \P_{\eta,\xi,\alpha^\epsilon_s})
 \Big] \,ds .
\end{align*}
We claim that  $\rho_i(\epsilon)\to0 $ as $\epsilon\to0$, for every $i$. Assuming this for a moment, from \eqref{versosottosol} we obtain
\begin{equation}\label{versosottosoldue}
-\epsilon\le \sum_{i=1}^4\rho_i(\epsilon)+\partial_t\varphi(t,\mu)+\frac{1}{\epsilon}
\int_t^{t+\epsilon}
 \E\Big[
\call^{\alpha^\epsilon}\varphi(t,\eta,\xi, \P_{\eta,\xi,\alpha^\epsilon_s})
+ \langle \eta ,f(\xi, \Gamma(\P_{\eta,\xi,\alpha^\epsilon}),\alpha_s^\epsilon)\rangle\Big]\,ds.
\end{equation} 
We note that the measure $\pi:=\P_{\eta,\xi,\alpha^\epsilon}$ has marginal $\pi_{12}=\P_{\eta,\xi}=\mu $ and so
\begin{align*}
 &\E\Big[
\call^{\alpha^\epsilon}\varphi(t,\eta,\xi, \P_{\eta,\xi,\alpha^\epsilon_s})
+ \langle \eta ,f(\xi, \Gamma(\P_{\eta,\xi,\alpha^\epsilon}),\alpha_s^\epsilon)\rangle\Big]
\\&\qquad \le 
\sup_{\pi\in \cald\,:\,\pi_{12}=\mu} 
\int_{\R^N\times \R^d\times A} \Big[
\call^{a}\varphi(t,x,y, \pi)
+ \langle y ,f(y, \Gamma(\pi),a)\rangle\Big]\,\pi(dx\,dy\,da).
\end{align*} 
Since the right-hand side does not depend on $s$ it follows from \eqref{versosottosoldue} that
\begin{align*}
-\epsilon\le \sum_{i=1}^5\rho_i(\epsilon)+\partial_t\varphi(t,\mu)+\sup_{\pi\in \cald\,:\,\pi_{12}=\mu} 
\int_{\R^N\times \R^d\times A} \Big[
\call^{a}\varphi(t,x,y, \pi)
+ \langle y ,f(y, \Gamma(\pi),a)\rangle\Big]\,\pi(dx\,dy\,da),
\end{align*}
and letting $\epsilon\to 0$ we get the required conclusion. 

It remains to prove the claim. We only sketch the  argument, since the required estimates  are similar to previous  ones.   We first note that, by   Corollary \ref{continizi} we have, as $s\to t$, 
\begin{align}
    \label{ancoracontxy}
\tilde \calw(\P_{X_{s}^\epsilon,Y_{s}},\mu)\le \E\,[| X_{s}^\epsilon-\eta|^p]+\E\,[|Y_{s}-\xi|^r]
\le 
\sup_{\alpha\in\cala}\|X_s^{t,\eta,\xi,\alpha}-\eta\|_p^p+\|Y_s^{t,\xi}-\xi\|^r_r\to 0.
\end{align}
Then $\rho_1(\epsilon)\to0 $ follows from the continuity of $\partial_t\varphi$. 

To prove  that  
$\rho_2(\epsilon)\to0 $ we use our assumptions on $f$, namely the growth condition
\eqref{growthfg} in Assumption (\nameref{HpA1})  and the uniform continuity condition 
\eqref{unifomega0f} in Assumption (\nameref{HpUniformContinuityfg}), as well as the estimates on the solution to the state equation given by Theorem \ref{existstateeqprovv}.

The claim
$\rho_3(\epsilon)\to0 $ also follows from properties of $f$: the Lipschitz condition \eqref{liponfeg} and the uniform continuity condition \eqref{unifomegaqf} in Assumption (\nameref{HpUniformContinuityfg}).

To prove  that  
$\rho_4(\epsilon)\to0 $ we consider separately the various terms that occur in the operator $\call^a$. They can all be treated by similar arguments. For instance let us show how to prove that
\begin{align}
&I_\epsilon:=\frac{1}{\epsilon}
\int_t^{t+\epsilon}
 \E\Big[ Trace[diag(X_{s}^\epsilon)\,h (Y_{s}, \Gamma(\P_{X_{s}^\epsilon,Y_{s},\alpha^\epsilon_s}),\alpha^\epsilon_s) \sigma(Y_{s})^T
\partial_x\partial_y\delta_m \varphi(s,\P_{X_{s}^\epsilon,Y_{s}};X_{s}^\epsilon,Y_s)]
\\&\qquad\qquad\quad
-
Trace[
  diag(\eta)\,h (\xi, \Gamma(\P_{\eta,\xi,\alpha^\epsilon_s}),\alpha^\epsilon_s) \sigma(\xi)^T
\partial_x\partial_y\delta_m  \varphi(t,\mu;\eta,\xi)]
 \Big] \,ds \to 0.
\end{align}
 We have $I_\epsilon= I_\epsilon^1+I_\epsilon^2$, where we set
\begin{align*}
&I^1_\epsilon=\frac{1}{\epsilon}
\int_t^{t+\epsilon}
 \E\Big[ Trace[\{diag(X_{s}^\epsilon)\,h (Y_{s}, \Gamma(\P_{X_{s}^\epsilon,Y_{s},\alpha^\epsilon_s}),\alpha^\epsilon_s) \sigma(Y_{s})^T-
 diag(\eta)\,h (\xi, \Gamma(\P_{\eta,\xi,\alpha^\epsilon_s}),\alpha^\epsilon_s) \sigma(\xi)^T\} 
\\&\qquad\qquad\quad
\partial_x\partial_y\delta_m \varphi(s,\P_{X_{s}^\epsilon,Y_{s}};X_{s}^\epsilon,Y_s)]
 \Big] \,ds,
\end{align*}
\begin{align*}
&I_\epsilon^2=\frac{1}{\epsilon}
\int_t^{t+\epsilon}
 \E\Big[ Trace[ diag(\eta)\,h (\xi, \Gamma(\P_{\eta,\xi,\alpha^\epsilon_s}),\alpha^\epsilon_s) \sigma(\xi)^T
 \\&\qquad\qquad\quad
 \{
\partial_x\partial_y\delta_m \varphi(s,\P_{X_{s}^\epsilon,Y_{s}};X_{s}^\epsilon,Y_s)
-
\partial_x\partial_y\delta_m  \varphi(t,\mu;\eta,\xi)\}]
 \Big] \,ds .
\end{align*}
Recalling that  $\partial_x\partial_y\delta_m  \varphi$ is bounded  we have
\begin{align*}
&I^1_\epsilon\le\frac{C}{\epsilon}
\int_t^{t+\epsilon}
 \E\Big[ |diag(X_{s}^\epsilon)\,h (Y_{s}, \Gamma(\P_{X_{s}^\epsilon,Y_{s},\alpha^\epsilon_s}),\alpha^\epsilon_s) \sigma(Y_{s})^T-
 diag(\eta)\,h (\xi, \Gamma(\P_{\eta,\xi,\alpha^\epsilon_s}),\alpha^\epsilon_s) \sigma(\xi)^T| 
 \Big] \,ds,
\end{align*}
which tends to $0$ by \eqref{ancoracontxy}, the 
uniform continuity condition 
\eqref{unifomega0} in Assumption (\nameref{HpUniformContinuityh}),  the Lipschitz condition of Assumption (\nameref{HphLipschitz})  and
the uniform continuity condition 
\eqref{unifomegaq} in Assumption (\nameref{HpUniformContinuityh}). 
Similarly, 
Recalling that $h$ is bounded and $\sigma$ is Lipschitz, we see that  $I_\epsilon^2$ can be estimated as
\begin{align*}
&I_\epsilon^2\le\frac{C}{\epsilon}
\int_t^{t+\epsilon}
 \E\Big[ | \eta|(1+|\xi|)\;
 |
\partial_x\partial_y\delta_m \varphi(s,\P_{X_{s}^\epsilon,Y_{s}};X_{s}^\epsilon,Y_s)
-
\partial_x\partial_y\delta_m  \varphi(t,\mu;\eta,\xi)|
 \Big] \,ds ,
\end{align*}
which tends to $0$ by \eqref{ancoracontxy} and the regularity properties of $\varphi$. 
This way we conclude that $I_\epsilon\to0$ and by similar arguments that $\rho_4(\epsilon)\to 0$.

\bigskip

{\it Supersolution property.} Suppose that $(t,\mu)\in [0,T)\times \cald_{p,r}$ 
and for some $\varphi\in \tilde C^{1,2}([0,T]\times \cald_{p,r})$  and $(t,\mu)\in [0,T)\times \cald_{p,r}$ we have
\begin{align}
\label{supersolv}
(v-\varphi)(t,\mu)  =  \min_{[0,T]\times \cald_{p,r}}(v-\varphi). 
\end{align}
Fix $\pi\in\cald$ with $\pi_{12}=\mu$
and take $\calf_0$-measurable random variables $(\eta,\xi,\alpha):\Omega\to \R^N\times \R^d\times A$ such that $\P_{\eta,\xi,\alpha}=\pi$ so that in particular $\P_{\eta,\xi}=\mu$ (they can be constructed as functions of the auxiliary variable $U$ alone). As an admissible control choose 
the constant-in-time  process equal to the random variable $\alpha$ and denote $(X,Y)$ the corresponding trajectory starting from $(\eta,\xi)$ at time $t$. From the dynamic programming principle (Proposition \ref{propDPPbis}) we have, for $t\le t+h\le T$,
\begin{align*}
v(t+h, \P_{X_{t+h},Y_{t+h}})-
v(t,\mu)
+\int_t^{t+h}\E\Big[\langle X_s ,f(Y_s, \Gamma(\P_{X_s,Y_s,\alpha}),\alpha)\rangle\Big] \,ds 
\le 0.
\end{align*}
By \eqref{supersolv}, the same inequality holds with $v$ replaced by $\varphi$. Applying the It\^o formula \eqref{itosolution} we obtain
\begin{align*}
\int_t^{t+h}\Big\{ \partial_t\varphi(s, \P_{X_{s},Y_{s}})+
 \E\Big[
\call^\alpha\varphi(X_s,Y_s, \P_{X_{s},Y_{s},\alpha})
+ \langle X_s ,f(Y_s, \Gamma(\P_{X_s,Y_s,\alpha}),\alpha)\rangle\Big] \Big\}\,ds 
\le 0.
\end{align*} 
Using again the regularity properties of $\varphi$ and our assumptions on the coefficients, 
it can be checked that the term in curly brackets is a continuous function of $s$ as $s\to t$. Dividing by $h$ and letting $h\to 0$ we obtain
\begin{align*}
    0&\ge
\partial_t\varphi(t, \mu)+
 \E\Big[
\call^\alpha\varphi(\eta,\xi, \pi)
+ \langle \eta ,f(\xi, \Gamma(\pi),\alpha)\rangle\Big]
\\&=
\partial_t\varphi(t, \mu)+
 \int_{\R^N\times \R^d\times A}\Big[
\call^a\varphi(x,y, \pi)
+ \langle x,f(y, \Gamma(\pi),a)\rangle\Big]  \,\pi(dx\,dy\,da).
\end{align*}
The conclusion follows from the arbitrariness of $\pi$. 
\qed

\bigskip
\paragraph{A complement to the proof of Proposition \ref{propDPPbis}}
We conclude the Appendix with the following proposition, which was used in the construction described before Proposition \ref{propDPPbis}.

\begin{Proposition}\label{varskorohod}
    Let $\calx,\caly$ be Polish spaces and $(\gamma_n)$ be a sequence of probabilities on $\calx\times\caly$ such that their first marginal $\mu$ is the same for all  $n$. Then there exist Borel measurable mappings $\underline{\xi}:(0,1)\to\calx$, $\underline{\xi}_n:(0,1)\to \caly$ such that $(\underline{\xi},\underline{\xi}_n)_\sharp m=\gamma_n$, i.e. $(\underline{\xi},\underline{\xi}_n)$ carries the Lebesgue measure $m$ to $\gamma_n$.
\end{Proposition}

\begin{proof}
Let us disintegrate $\gamma_n$ in the form
\[
\gamma_n(dx\,dy)=\overline{\gamma}_n(x,dy)\,\mu(dx),
\]
for a suitable probability kernel $\overline{\gamma}_n(x,dy)$ from $\calx$ to $\caly$. By a classical construction of Skorohod (used for instance in the proof of Skorohod's representation theorem) one can find a Borel measurable map $X:(0,1)\to\calx$ such that $X_\sharp m=\mu$. By the same result, for every $x\in \calx$, there exists  Borel measurable maps $Y_n(x,\cdot):(0,1)\to\caly$ such that $Y_n(x,\cdot)_\sharp m= \overline{\gamma}_n(x,\cdot)$. Since $\overline{\gamma}_n(x,dy)$ is a kernel, the functions $Y_n:(0,1)\times (0,1)\to \caly$ can be chosen to be Borel measurable jointly in their two variables: this follows from the constructive proof of   Skorohod's result (see e.g. \cite{Zab96}, Theorem 3.1.1 for a detailed verification).  Now define $Z_n:(0,1)\times (0,1)\to \calx\times\caly$ setting
\[
Z_n(u,v)=\Big(X(u),Y(X(u),v)\Big), \qquad u,v,\in (0,1).
\]
An easy computation shows that $(Z_n)_\sharp (m\otimes m)=\gamma_n$. Now take any Borel measurable map $\Phi:(0,1)\to (0,1)\times (0,1)$, given again by the Skorohod construction, such that $\Phi_\sharp m=m\otimes m$. Defining $\underline{\xi},\underline{\xi}_n$ as the components of $Z_n\circ\Phi $ one checks the desired properties.
\end{proof}

\bibliographystyle{plain}
\bibliography{The_Complete_Bibliography}

\end{document}